\UseRawInputEncoding
\documentclass[12pt, reqno]{amsart}
\usepackage[margin=1in]{geometry}
\usepackage{amssymb,latexsym,amsmath,amscd,amsfonts}
\usepackage{latexsym}
\usepackage[mathscr]{eucal}
\usepackage{bm}
\usepackage{tabularx}
\usepackage{mathptmx}
\usepackage{amssymb}
\usepackage{amsthm}
\usepackage{dcolumn}
\usepackage[all]{xy}
\usepackage{enumitem}
\usepackage[utf8]{inputenc}
\usepackage{mathtools}

\def \qed {\hfill \vrule height6pt width 6pt depth 0pt}
\def\textmatrix#1&#2\\#3&#4\\{\bigl({#1 \atop #3}\ {#2 \atop #4}\bigr)}
\def\dispmatrix#1&#2\\#3&#4\\{\left({#1 \atop #3}\ {#2 \atop #4}\right)}
\newcommand{\beg}{\begin{equation}}
	\newcommand{\eeg}{\end{equation}}
\newcommand{\ben}{\begin{eqnarray*}}
	\newcommand{\een}{\end{eqnarray*}}

\newcommand{\Q}{\mathbb{H}}
\newcommand{\CQ}{\overline{\mathbb{H}}}
\newcommand{\C}{\mathbb{C}}

\newcommand{\D}{\mathbb{D}}
\newcommand{\DC}{\overline{\mathbb{D}}}
\newcommand{\T}{\mathbb{T}}

\newcommand{\Gg}{\mathbb{G}_2}

\newcommand{\E}{\mathbb E}
\newcommand{\Ebar}{\overline{\mathbb E}}

\newcommand{\lm}{\lambda}

\newcommand{\al}{\alpha}
\newcommand{\be}{\beta}
\newcommand{\om}{\omega}

\newcommand{\A}{\rm{Aut}}

\newcommand{\HB}{\mathbb{H}_\mu}
\newcommand{\CHB}{\overline{\mathbb{H}}_\mu}
\newcommand{\QN}{\mathbb{H}_{N}}
\newcommand{\CQN}{\overline{\mathbb{H}}_{N}}
\newcommand{\Pe}{\mathbb{P}}
\newcommand{\Pbar}{\overline{\mathbb{P}}}
\newcommand{\Pz}{\psi_{z_1, z_2}}
\newcommand{\w}{\omega}

\newtheorem{thm}{Theorem}[section]

\newtheorem{lem}[thm]{Lemma}

\newtheorem{prop}[thm]{Proposition}
\numberwithin{equation}{section} \theoremstyle{definition}
\newtheorem{defn}[thm]{Definition}
\newtheorem{rem}[thm]{Remark}

\newtheorem{eg}[thm]{Example}

\def\textmatrix#1&#2\\#3&#4\\{\bigl({#1 \atop #3}\ {#2 \atop #4}\bigr)}
\def\dispmatrix#1&#2\\#3&#4\\{\left({#1 \atop #3}\ {#2 \atop #4}\right)}
\begin{document}
	\title{On a family of domains associated with $\mu$-synthesis in $M_2(\mathbb C)$}
	
	\author{Indranil Biswas, Sourav Pal and Nitin Tomar}	
	
	\address[Indranil Biswas]{Department of Mathematics,
		Shiv Nadar University,
		NH91, Tehsil Dadri, Greater
		Noida,
		Uttar Pradesh 201314, India.} \email{indranil.biswas@snu.edu.in}		
	
	\address[Sourav Pal]{Mathematics Department, Indian Institute of Technology Bombay,
		Powai, Mumbai - 400076, India.} \email{sourav@math.iitb.ac.in, souravmaths@gmail.com}
	
	\address[Nitin Tomar]{Statistics and Mathematics Unit, Indian Statistical Institute, Bangalore, Karnataka - 560059, India.} \email{tomarnitin414@gmail.com}

	
	
	\begin{abstract}
		The study of special cases of the $\mu$-synthesis in $M_2(\mathbb C)$ gives rise to several domains in $\mathbb C^n$, e.g., the symmetrized bidisc $\mathbb G_2$, the tetrablock $\mathbb E$, the pentablock $\mathbb P$ and the hexablock $\mathbb H$. In this paper, we show that the domain $\mathbb H$ serves as a unifying domain for $\mathbb G_2, \mathbb E$ and $\mathbb P$ by proving that each of these domains is an analytic retract of $\mathbb H$. Further, we prove that the bidisc $\mathbb D^2$ and the Euclidean unit ball $\mathbb B_2$ are also analytic retracts of $\mathbb H$. Then the rational $\overline{\mathbb H}$-inner functions, i.e., the rational maps from $\mathbb D$ into $\overline{\mathbb H}$ whose boundary values lie in the distinguished boundary of $\mathbb H$, are characterized. Two different versions of Schwarz type lemma for $\mathbb H$ are established of which one provides a necessary and sufficient condition. Finally, capitalizing on the theory of $\mathbb H$, we present alternative proofs to a few known results on $\mathbb P$ from the literature.
	\end{abstract}	
	
	\maketitle
	
	\section{Introduction}\label{ch_intro} 	
	\noindent A family of four domains, which arise naturally from different cases of the $\mu$-synthesis into the space of $2\times 2$ matrices $M_2(\C)$, has attracted considerable attention over past two decades. We briefly describe the $\mu$-synthesis problem that originates in $H^{\infty}$ control theory, see \cite{Doyle, Francis}. Given a linear subspace $E$ of $M_n(\C)$, \textit{the structured singular value} associated with $E$ is the function $\mu_E: M_n(\C)\to[0,\infty)$ defined by
	\begin{equation} \label{eqn:NEW-01}
	\mu_E(A)=1\slash \inf\{\|X\|\in E,\ \det(I-AX)=0\}
	\end{equation}
	with the understanding that $\mu_E(A)=0$ if there does not exist $X\in E$ with $\det(I-AX)=0$. Here, $\|\cdot\|$ denotes the operator norm on $M_n(\mathbb{C})$ and $I$ is the identity matrix. For a linear subspace $E\subseteq M_n(\mathbb C)$, the corresponding $\mu$-synthesis problem asks the following: given distinct points $\lambda_1,\ldots,\lambda_m$ in the open unit disc $\mathbb D$ in the complex plane and matrices $A_1,\ldots,A_m\in M_n(\mathbb C)$, does there exist a holomorphic map $F:\mathbb D\to M_n(\mathbb C)$ satisfying $\mu_E(F(\lambda))\le1$ for every $\lambda\in\mathbb D$ such that	$F(\lambda_j)=A_j$ for $1\leq j\leq m$ ? When $E=M_n(\mathbb C)$, $\mu_E$ coincides with the operator norm and the corresponding $\mu$-synthesis problem reduces to the classical matricial Nevanlinna-Pick interpolation problem (see \cite{Foias_Frazho} for details). Also, when $E=\{\alpha I:\alpha\in\mathbb C\}$, $\mu_E$ becomes the spectral radius and the resulting $\mu$-synthesis problem is the spectral Nevanlinna-Pick problem, e.g., see \cite{Agler2004, Costara2005_II}. For further background on the control-theoretic origin of the structured singular value, we refer the reader to \cite{Doyle, DoyleII, Francis}. Suppose $E_1, E_2$ and $E_3$ are linear subspaces of $M_2(\C)$ consisting of all scalar matrices, diagonal matrices and upper triangular matrices with same diagonal entries, respectively. We now show the readers how the $\mu_{E_1}, \mu_{E_2}, \mu_{E_3}$ unit balls induce domains naturally and what are their relevance with the corresponding $\mu$ unit balls. For example, consider the linear subspace $E_2$ consisting of $2 \times 2$ diagonal matrices. It was shown in \cite{Abouhajar} that the $\mu$-synthesis with respect to this subspace induces the domain \textit{tetrablock}, $\mathbb E$. Indeed, for $X=\begin{pmatrix}
z_1 & 0\\
0& z_2
\end{pmatrix} \in E$ and $A=(a_{ij})_{i,j=1}^2 \in M_2(\mathbb{C})$, we have that
$\|X\|=\max\{|z_1|, |z_2|\}$ and $\det(I-AX)=1-a_{11}z_1-a_{22}z_2+\det(A)z_1z_2$ and it follows from the definition of $\mu_{E_2}$ (see Equation-(\ref{eqn:NEW-01})) that
\begin{align*}
\mu_{E_2}(A)<1
& \iff \|X\|>1 \ \text{for all} \ X \in E \ \text{with} \ \det(I-AX)=0\\
& \iff \det(I-AX)\ne 0 \ \text{for all} \ X \in E \ \text{with} \ \|X\| \leq 1\\
& \iff 1-a_{11}z_1-a_{22}z_2+\det(A)z_1z_2\ne 0 \ \text{for all} \ z_1, z_2 \in \overline{\mathbb{D}}\\
& \iff (a_{11}, a_{22}, \det(A)) \in \mathbb{E}.
\end{align*}
Thus, $A=(a_{ij})_{i,j=1}^2$ is in the $\mu_E$-unit ball if and only if the point $(a_{11}, a_{22}, \det(A))$ belongs to the tetrablock $\mathbb E$. It turns out that an interpolation problem between $\D$ and the $\mu_{E_2}$-unit ball is equivalent with a similar interpolation problem between $\D$ and the tetrablock $\mathbb E$,  e.g., see the last Section of \cite{Abouhajar}. Evidently, it is easier to deal with a bounded domain like $\mathbb E \subset \C^3$ than the norm-unbounded object, the $\mu_{E_2}$ unit ball in $\C^4$. This is in fact the main reason behind studying such a domain arising from a $\mu$-synthesis problem. Similarly, the domains associated with $\mu_{E_1}$ and $\mu_{E_3}$ are the symmetrized bidisc $\mathbb G_2$ and the pentablock $\mathbb P$, respectively. Indeed,
	{\fontsize{10.7}{14}\selectfont
		\begin{equation}\label{eqn_domains}
			\left.
			\begin{aligned}
				\Gg&=\{(\text{tr}(A), \det(A)): A \in M_2(\C), \mu_{E_1}(A)<1\}
				=\{(\text{tr}(A), \det(A)): A \in M_2(\C), \|A\|<1\},\\[1mm]
				\E&=\{(a_{11}, a_{22}, \det(A)): A \in M_2(\C), \mu_{E_2}(A)<1\}
				=\{(a_{11}, a_{22}, \det(A)): A \in M_2(\C), \|A\|<1\},\\[1mm]
				\Pe&=\{(a_{21}, \text{tr}(A), \det(A)): A \in M_2(\C), \mu_{E_3}(A)<1\}
				=\{(a_{21}, \text{tr}(A), \det(A)): A \in M_2(\C), \|A\|<1\}.
			\end{aligned}
			\right\}
		\end{equation}
	}
	\par \noindent 
	Here $\text{tr}(A)$ and $\det(A)$ denotes the trace and determinant of $A \in M_2(\C)$, respectively.  The domains $\mathbb{G}_2$, $\mathbb{E}$ and $\mathbb{P}$ were introduced in \cite{AglerYoung}, \cite{Abouhajar} and \cite{AglerIV}, respectively. Each of these domains has been studied extensively from complex geometric, function theoretic and operator theoretic perspective. We refer the reader to \cite{AglerYoung, Agler2003, AglerIII, Tirtha_Pal, tirtha-sourav1, Costara2004, zwonek1} for the symmetrized bidisc, \cite{AbouhajarI, Abouhajar, Tirtha, EdigarianI, Kosinski, zwonek1, Zwonek} for the tetrablock and \cite{AglerIV, KosinskiII, Guicong, JindalII, Pal_penta} for the pentablock. Since the subspaces $E_1, E_2$ and $E_3$ are all contained in the linear subspace $E_4$ consisting of all upper triangular matrices in $M_2(\C)$, it is natural to consider the $\mu$-synthesis problem corresponding to $E_4$. Motivated by this, the authors of \cite{Biswas-Pal} introduced the hexablock $\Q$, which is a domain in $\C^4$ defined as
	\[
	\Q=\left\{(a,x_1,x_2,x_3)\in\C\times\E:
	\sup_{z_1,z_2\in\D}\bigl|\Pz(a,x_1,x_2,x_3)\bigr|<1
	\right\},
	\]
	where $\Pz$ is the fractional linear map given by
	\begin{equation}\label{eqn_Pz}
		\Pz(a,x_1,x_2,x_3)=
		\frac{a\sqrt{(1-|z_1|^2)(1-|z_2|^2)}}
		{1-x_1z_1-x_2z_2+x_3z_1z_2}.
	\end{equation}
	To further understand the origin of $\Q$ from $\mu$-synthesis problem, we recall from \cite{Biswas-Pal} the sets 
	\[
	\HB=\{\pi(A) : A \in M_2(\C), \  \mu_{E_4}(A)<1 \} \quad \text{and} \quad \QN=\{\pi(A) : A \in M_2(\C), \  \|A\|<1 \},
	\]
	where $\pi(A)=(a_{21}, a_{11}, a_{22}, \text{det}(A))$ for $A=(a_{ij})_{i, j=1}^2 \in M_2(\C)$. The sets $\HB$ and $\QN$ are called \textit{$\mu$-hexablock} and \textit{normed hexablock}, respectively. As mentioned in \eqref{eqn_domains}, the domains $\Gg, \E$ and $\Pe$ admit equivalent descriptions in terms of the corresponding $\mu_E$-unit ball and the operator norm unit ball of $M_2(\C)$. Unlike the domains $\Gg, \E$ and $\Pe$, the sets $\HB$ and $\QN$ are distinct and neither is a domain in $\C^4$.  However, it turns out that
	\[
	\Q= \text{int}(\CHB)=\text{int}\left(\widehat{\CQN}\right), 
	\]
	where $\widehat{\CQN}$ is the polynomial convex hull of $\CQN$ and $\text{int}(X)$ denotes the interior of a given topological space $X$. We refer the reader to \cite{Biswas-Pal, Biswas-Pal2} for further details on the origin and geometric properties of $\Q$. The inclusion relations
	$
	E_1\subseteq E_2$, and $E_2, E_3\subseteq E_4\subseteq M_2(\C)
	$
	together with the monotonicity of the structured singular value imply that
	$
	r(A)=\mu_{E_1}(A)\leq \mu_{E_2}(A),
	\mu_{E_3}(A)\leq \mu_{E_4}(A)\leq \|A\|
	$
	for every $A\in M_2(\C)$. It is thus natural to study if these inclusion relations among the underlying linear subspaces $E_1, E_2, E_3$ and $E_4$ are reflected at the level of the associated domains $\Gg, \E, \Pe$ and $\Q$, respectively. More precisely, we ask whether the domains $\Gg$, $\E$ and $\Pe$ admit holomorphic embeddings into $\Q$. In Section \ref{connect}, we answer this question affirmatively and prove that all three domains are analytic retracts of $\Q$. We further show that the same conclusion holds for the bidisc $\D^2$ and the Euclidean biball $\mathbb{B}_2$. Indeed, we prove the following result.
	
		\begin{thm}
		For $(a, s, p, x_1, x_2, x_3) \in \C^6$, the following hold:
		\begin{enumerate}
			\item[$(1)$] $(a, x_3) \in \D^2$ $($respectively, $\DC^2$$)$ if and only if $(a, 0, 0, x_3) \in \Q$ $($respectively, $\CQ$$)$; \vspace{0.1cm}
			
			\item[$(2)$] $(a, x_1) \in \mathbb{B}_2$ $($respectively, $\overline{\mathbb{B}}_2$$)$ if and only if $(a, x_1, 0, 0) \in \Q$ $($respectively, $\CQ$$)$; \vspace{0.1cm}
			
			\item[$(3)$] $(s, p) \in \Gg$ $($respectively, $\Gamma_2$$)$ if and only if $(0, s\slash 2, s\slash 2, p) \in \Q$ $($respectively, $\CQ$$)$; \vspace{0.1cm}
			
			\item[$(4)$] $(x_1, x_2, x_3) \in \E$ $($respectively, $\Ebar$$)$ if and only if $(0, x_1, x_2, x_3) \in \Q$ $($respectively, $\CQ$$)$; \vspace{0.1cm}
			
			\item[$(5)$] $(a, s, p) \in \Pe$ $($respectively, $\Pbar$$)$ if and only if $(a, s\slash 2, s\slash 2, p) \in \Q$ $($respectively, $\CQ$$)$.
		\end{enumerate}	
	\end{thm}
 This is Theorem \ref{thm_final_embed} in this paper. Consequently, the domain $\Q$ provides a unified framework for studying the classical domains $\D^2$ and $\mathbb{B}_2$, as well as the domains arising from the $\mu$-synthesis problem, namely $\Gg$, $\E$ and $\Pe$. 
	
	\smallskip 
	
	For a bounded domain $\Omega \subset \C^n$,  a holomorphic map $f=(f_1, \dotsc, f_n): \D \to \overline{\Omega}$ is said to be \textit{rational $\overline{\Omega}$-inner}, or simply \textit{rational ${\Omega}$-inner} if each $f_i$ is a rational map that is holomorphic on $\D$ and $f(\T)$ is a subset of the distinguished boundary $b\Omega$ of $\Omega$. Recall that the distinguished boundary $b\Omega$ is the smallest closed subset of $\overline{\Omega}$ on which every function that is continuous on $\overline{\Omega}$ and holomorphic on $\Omega$ attains its maximum modulus.  Blaschke  \cite{Blaschke} proved that any rational $\overline{\D}$-inner function is a finite product of maps of the form $\w B_\alpha$, where $\w \in \T, \alpha \in \D$ and
	\[
	B_\alpha(\lambda)= \frac{\lambda-\alpha}{\overline{\alpha}\lambda-1} \quad (\lambda \in \D).
	\]
	A complete description of rational $\Gamma_2$-inner functions was given by Agler et al. in \cite{AglerV}. Here, $\Gamma_2$ is the closure of $\Gg$. Alsalhi and Lykova \cite{Alsalhi} characterized the rational $\Ebar$-inner functions. As a next step, the rational $\overline{\Pe}$-inner functions were described in \cite{Alsheri} and \cite{Jindal}. To do so, the authors of \cite{Alsheri, Jindal} utilized the representation of rational $\Gamma_2$-inner functions and proved that if $(a, s, p)$ is a rational $\Pbar$-inner function, then $(s, p)$ is a rational $\Gamma_2$-inner function. So, the problem reduces to finding necessary and sufficient conditions on the first component so that $(a, s, p)$ is a rational $\Pbar$-inner function. Using the same idea, we prove in Section \ref{inner} that if $(a, x_1, x_2, x_3)$ is a rational $\CQ$-inner function, then $(x_1, x_2, x_3)$ is a rational $\Ebar$-inner function. Now, building upon the theory of rational $\Ebar$-inner functions, we present a complete description of rational $\CQ$-inner functions in Section \ref{inner} as stated below.

	\begin{thm}
		Let $f=(a, x_1, x_2, x_3)$ be a rational $\CQ$-inner function and let $deg(x_3)=n$. Let $a \ne 0$ and let $a=a_{in}a_{out}$ be an inner-outer factorization of $a$. Then there exist polynomials $A, E_1, E_2, D$ of degree at most $n$ such that  the following hold.
		\begin{enumerate}
			\item $E_1(\lm)=E_2^{\sim n}(\lm)$ for $\lm \in \overline{\D}$,\\
			
			\item $A$ is outer and $D(\lm) \ne 0$ for all $\lambda  \in \overline{\D}$,\\
			
			\item $|E_i(\lm)| \leq |D(\lm)|$ on $\overline{\D}$, for $i=1,2$,\\
			
			\item $|A(\lambda)|^2=|D(\lambda)|^2-|E_1(\lambda)|^2$ on $\T$,\\	
			
			\item 
			$\displaystyle 
			(a, x_1, x_2, x_3)=\left(a_{in}\frac{A}{D}, \frac{E_1}{D}, \frac{E_2}{D}, \frac{D^{\sim n}}{D} \right)$ on $\overline{\D}$.
		\end{enumerate}
		Conversely, if $A, E_1, E_2$ and $D$ are polynomials of degree at most $n$ that satisfy conditions $(1)-(4)$ and $a, x_1, x_2$ and $x_3$ are defined as in $(5)$, then $f=(a, x_1, x_2, x_3)$ is a rational $\CQ$-inner function with degree of $x_3$ at most $n$.	
	\end{thm}
	 
The above result is established in Theorem \ref{thm10.12} of Section \ref{inner}. En route, we provide various examples of rational $\CQ$-inner functions. 
	
	\smallskip 
	
	The classical Schwarz lemma provides a necessary and sufficient condition for the solvability of a two-point interpolation problem from $\mathbb{D}$ to itself. This lemma has numerous generalizations, where the focus shifts to determine the solvability of a two-point interpolation problem from $\mathbb{D}$ to some other domain in $\mathbb{C}^n$. Agler and Young \cite{AglerII} proved a Schwarz lemma for the symmetrized bidisc. Abouhajar et al.  \cite{Abouhajar} obtained a Schwarz lemma for the tetrablock. In recent past, a Schwarz lemma for pentablock was studied by Alshehri and Lykova in \cite{Alsheri}. Motivated by these seminal works, we prove in Section \ref{Schwarz_lem} a Schwarz type lemma for the hexablock, where the interpolating function is rational $\CQ$-inner, that is, for given $\lambda_0 \in \D$ and $(a, x_1, x_2, x_3) \in \Q$, we find necessary and sufficient conditions for the existence of a rational $\CQ$-inner map $f:\D \rightarrow \Q$ such that $f(0)=(0, 0, 0, 0)$ and $f(\lambda_0)=(a, x_1, x_2, x_3)$. In particular, we have the following result. 
	
		\begin{thm}
		Let $\lambda_0\in\mathbb D\setminus\{0\}$ and
		$(a,x_1,x_2,x_3)\in\mathbb H$. Then the following
		are equivalent:
		
		\begin{enumerate}
			\item there exists a rational $\CQ$-inner
			map $f$ such that $f(0)=(0,0,0,0)$ and $f(\lambda_0)=(a,x_1,x_2,x_3)$;

			\item there exist $w_{12}, w_{21}, z_{12}\in\mathbb C$ satisfying $w_{12}w_{21}=x_1x_2-x_3$ and $|a| \leq |\lm_0w_{12}|$ such that 
			\begin{equation*}
				\begin{pmatrix}
					I-W_0W_0^*&I-W_0W_1^*\\[2mm]
					I-W_1W_0^*&
					\dfrac{I-W_1W_1^*}{1-|\lambda_0|^2}
				\end{pmatrix}
				\geq0,
				\quad 	\text{where} \quad W_0=
				\begin{pmatrix}
					0&z_{12}\\
					0&0
				\end{pmatrix} \ \ \text{and} \ \  W_1=
				\begin{pmatrix}
					x_1& w_{12}\\
					w_{21}&x_2
				\end{pmatrix}.
			\end{equation*}
		\end{enumerate}
	\end{thm}
This is Theorem \ref{thm_hexa_schwarz} in this paper. The theory of rational $\CQ$-inner functions developed in Section \ref{inner} plays a significant role in this endeavor. Furthermore, we have a set of necessary conditions for Schwarz lemma for the hexablock as presented below. 

		\begin{thm}
	Let $\lm_0 \in \D\setminus\{0\}$ and let $(a, x_1, x_2, x_3) \in \Q$. Then, in the following, $(1)\implies(2)\implies(3)\implies(4)\implies(5):$

	\begin{enumerate}
		\item there is a rational $\CQ$-inner function $f$ with $f(0)=(0, 0, 0, 0)$ and $f(\lm_0)=(a, x_1, x_2, x_3)$; 
		
		\item there exists $h \in$ Hol$(\D, \CQ)$ such that $h(0)=(0, 0, 0, 0)$ and $h(\lm_0)=(a, x_1, x_2, x_3)$;
		
		\item 
		$
		\underset{z_1, z_2 \in \D}{\sup} |\psi_{z_1, z_2}(a, x_1, x_2, x_3)| \leq |\lambda_0|
		$
		and
		\[
		\max\bigg\{\frac{|x_1-\overline{x}_2x_3|+|x_1x_2-x_3|}{1-|x_2|^2}, \ \frac{|x_2-\overline{x}_1x_3|+|x_1x_2-x_3|}{1-|x_1|^2}\bigg\} \leq |\lambda_0|;
		\]
		\item $(\lm_0^{-1}a, x_1, x_2, x_3) \in \CQ$	 and
		\[
		\max\bigg\{\frac{|x_1-\overline{x}_2x_3|+|x_1x_2-x_3|}{1-|x_2|^2}, \ \frac{|x_2-\overline{x}_1x_3|+|x_1x_2-x_3|}{1-|x_1|^2}\bigg\} \leq |\lambda_0|;
		\]		
		\item 
		$
		|a| \leq |\lm_0|\sqrt{1-|x_1|^2}
		$ 
		and
		\[
		\max\bigg\{\frac{|x_1-\overline{x}_2x_3|+|x_1x_2-x_3|}{1-|x_2|^2}, \ \frac{|x_2-\overline{x}_1x_3|+|x_1x_2-x_3|}{1-|x_1|^2}\bigg\} \leq |\lambda_0|.
		\]		
	\end{enumerate}
\end{thm}
This is Theorem \ref{necessary} in this paper. At every stage of this paper, we explore and find interaction of the hexablock with the symmetrized bidisc, the tetrablock and the pentablock. As an application of the theory of hexablock developed so far, we provide in Section \ref{sec_appl} alternative proofs of various results for the pentablock from  \cite{AglerIV, Alsheri, Jindal, KosinskiII}. 	
	
	\medskip 
	
	\noindent \textbf{Note.} The first arXiv version of \cite{Biswas-Pal} (which was 135 pages long) has been divided into several parts. The present paper is one of those parts and it contains the material from Chapters 10, 12, 13 and 14 of that version of \cite{Biswas-Pal}.

	\section{Connection with bidisc, biball, symmetrized bidisc, tetrablock and pentablock}	\label{connect}
	
	\noindent	In this section, we establish a connection of the hexablock $\Q$ with the bidisc $\D^2$, the biball $\mathbb{B}_2$, the symmetrized bidisc $\Gg$, the tetrablock $\E$ and the pentablock $\Pe$. Indeed, we show that all these domains can be embedded holomorphically in $\Q$ and each one of them is an analytic retract of $\Q$. Recall that an embedding is an injective continuous map that is a homeomorphism onto its image. By a holomorphic embedding, we mean an embedding which is also a holomorphic map. Also, a domain $\Omega_1 \subseteq \C^n$ is said to be an \textit{analytic retract} of a domain $\Omega_2 \subseteq \C^m$ if there are holomorphic maps $\phi: \Omega_1 \to \Omega_2$ and $\psi: \Omega_2 \to \Omega_1$ such that $\psi \circ \phi$ is the identity map $\text{id}_{\Omega_1}$ on $\Omega_1$. In particular, if $\Omega_1$ is an analytic retract of $\Omega_2$, then $\Omega_1$ is holomorphically embedded into $\Omega_2$. Let us recall that the domain $\Q$ is defined as
	\[
	\mathbb{H}=\left\{(a, x_1, x_2, x_3) \,\in\, \mathbb{C} \times \mathbb{E}\,\,:\,\,
	\sup_{z_1,\, z_2 \,\in\,
		\D}\left|\frac{a\sqrt{(1-|z_1|^2)(1-|z_2|^2)}}{1-x_1z_1-x_2z_2+x_3z_1z_2}\right|
	<1\right\},
	\]
	where $\E=\{(x_1, x_2, x_3) \in \C^3: 1-x_1z_1-x_2z_2+x_3z_1z_2 \ne 0 \ \text{for every} \ z_1, z_2 \in \DC \}$ is the tetrablock. An analogous description holds for the closure of $\Q$ (see Theorem 5.3 in \cite{Biswas-Pal}), which is given by
	\begin{align}\label{eqn5.3}
		\CQ=\left\{(a, x_1, x_2, x_3) \in \C \times \overline{\E} \ : \ \bigg|\frac{a\sqrt{(1-|z_1|^2)(1-|z_2|^2)}}{1-x_1z_1-x_2z_2+x_3z_1z_2}\bigg| \leq 1 \ \text{for every} \ z_1, z_2 \in \D \right\}.
	\end{align}
	Therefore, the sets $\Q$ and $\CQ$ are fibered over $\E$ and $\Ebar$, respectively, via the  map $(a, x_1, x_2, x_3) \mapsto (x_1, x_2, x_3)$. Also, $(0,x_1,x_2,x_3) \in \Q$ (or, $(0,x_1,x_2,x_3) \in \CQ$) for every $(x_1,x_2,x_3)\in\E$ (or, $(x_1, x_2, x_3) \in \Ebar$). We therefore begin by studying the connection between the domains $\E$ and $\Q$. For this purpose, we recall from the literature a few useful results related to the tetrablock.
	
	\begin{thm}[\cite{Abouhajar}, Theorems 2.2 \& 2.4]\label{tetrablock}
		For $(x_1, x_2, x_3) \in \C^3$, the following are equivalent:
		\begin{enumerate}
			\item $(x_1, x_2, x_3) \in \E$ $($ or, $(x_1, x_2, x_3) \in \Ebar$$)$; \vspace{0.1cm}
			
			\item $|x_1|^2+|x_2-\overline{x}_1x_3|+|x_1x_2-x_3|<1$ $($or, $|x_1|^2+|x_2-\overline{x}_1x_3|+|x_1x_2-x_3| \leq 1$ and $|x_2| \leq 1$$)$;
			
			\vspace{0.1cm}
			
			\item $1+|x_1|^2-|x_2|^2-|x_3|^2-2|x_1-\overline{x}_2x_3|>0$ and $|x_1|<1$ \\ $($or, $1+|x_1|^2-|x_2|^2-|x_3|^2-2|x_1-\overline{x}_2x_3| \geq 0$ and $|x_1|\leq 1$ $)$; \vspace{0.1cm}
			
			\item $|x_1-\overline{x}_2x_3|+|x_2-\overline{x}_1x_3| < 1-|x_3|^2$ $($or, $|x_1-\overline{x}_2x_3|+|x_2-\overline{x}_1x_3| \leq 1-|x_3|^2$ and $|x_1| \leq 1$$)$.
		\end{enumerate}	
	\end{thm}
	
	We also present a description of the distinguished boundary of $\E$. Recall that the distinguished boundary of a bounded domain $\Omega \subset \C^n$ is the smallest closed subset $b\Omega$ (equivalently, $b\overline{\Omega}$) of $\overline{\Omega}$ on which every function continuous on $\overline{\Omega}$ and holomorphic in $\Omega$ attains its maximum modulus.
	\begin{thm}[\cite{Abouhajar}, Theorem 7.1]\label{thm6.1}
		For $x=(x_1, x_2, x_3) \in \C^3$, the following are equivalent:
		\begin{enumerate}
			\item $x_1=\overline{x}_2x_3, |x_3|=1$ and $|x_2|\leq 1$;
			\item $x \in b\E$;
			\item $x \in \overline{\E}$ and $|x_3|=1$. 
		\end{enumerate}	
	\end{thm}  
	
	The characterization of $b\E$ leads to the following characterization of $b\Q$.
	
	\begin{thm}[\cite{Biswas-Pal2}, Theorem 3.23]\label{thmbQ}
		The distinguished boundary $b\Q$ of the hexablock is given by
		\[
		b\Q=\left\{(a, x_1, x_2, x_3) \in \C\times b\E : |a|^2+|x_1|^2=1 \right\}.
		\]
	\end{thm} 		
	
	The following theorem establishes a canonical relationship between $\E$ and $\Q$.
	
	\begin{thm}\label{thm_EQ}$($Connection of hexablock with tetrablock$)$
		A point $(x_1, x_2, x_3) \in \E$ if and only if $(0, x_1, x_2, x_3) \in \Q$. In particular, the map $e_{\E}: \E \to \Q$ given by $e_{\E}(x_1, x_2, x_3)=(0, x_1, x_2, x_3)$ is a holomorphic embedding, and $\E$ is an analytic retract of $\Q$.
	\end{thm}
	
	\begin{proof}
		We have by definition of $\Q$ that $(x_1, x_2, x_3) \in \E$ if and only if $(0, x_1, x_2, x_3) \in \Q$. The fact that $e_{\E}$ is a holomorphic embedding follows from the definition of $\Q$. Note that the map $\psi_\E: \Q \to \E$ defined as $\psi_\E(a, x_1, x_2, x_3)=(x_1, x_2, x_3)$ is a well-defined and holomorphic. Evidently, $\psi_\E \circ e_\E$ is the identity map on $\E$ and so, $\E$ is an analytic retract of $\Q$.
	\end{proof}

	Next, we consider the symmetrized bidisc $\Gg$ and its connection with the domain $\Q$. Recall that the domain $\Gg$ is the symmetrization of the bidisc $\D^2$, that is, $\Gg=\{(z_1+z_2, z_1z_2) : z_1, z_2 \in \D\}$. We now recall characterization of elements in $\Gg$ and its closure $\Gamma_2$ from the literature.
	
	\begin{thm}[\cite{AglerIII}, Theorem 2.1 \& \cite{Agler2003}, Theorem 1.1]\label{thmG_2}
		Let $(s,p)\in\C^2$. The following statements are equivalent:
		
		\begin{enumerate}
			\item $(s,p)\in \Gg$   $($or, $(s,p)\in \Gamma_2$$)$;
			
			\item $|s-\overline{s}p|<1-|p|^2$ 
			$($or, $|s|\le 2$ and
			$|s-\overline{s}p|\le 1-|p|^2$$)$;
			
			\item
			$
			2|s-\overline{s}p|+|s^2-4p|
			<4-|s|^2
			$ 
			$($or,
			$
			2|s-\overline{s}p|+|s^2-4p|
			\le 4-|s|^2$$)$.
		\end{enumerate}
	\end{thm}
	
	Next, we note down various characterizations of the distinguished boundary $b\Gamma_2$ of $\Gg$.
	
	\begin{thm}[\cite{Agler2003}, Theorem 1.3 \& \cite{Tirtha_Pal}, Theorem 1.3]\label{thm_bGam}
		For $(s, p) \in \C^2$, the following statements are equivalent:
		\begin{enumerate}
			\item $s=\overline{s}p, |p|=1$ and $|s|\leq 2$;
			\item $(s, p) \in b\Gamma_2$;
			\item $(s, p) \in \Gamma_2$ and $|p|=1$. 
		\end{enumerate}	
	\end{thm}

	We now present a relationship between the domains $\Gg$ and $\Q$.
	
	\begin{thm}\label{thm_GQretract}$($Connection of hexablock with symmetrized bidisc$)$
		A point $(s, p) \in \Gg$ if and only if $(0, s\slash 2, s\slash 2, p)\in \Q$. In particular, the map $e_{\Gg}: \Gg \to \Q$ defined by $e_{\Gg}(s, p)=(0,s\slash 2, s\slash 2, p)$ is a holomorphic embedding, and $\Gg$ is an analytic retract of $\Q$.
	\end{thm}
	
	\begin{proof}
		For any $(s, p) \in \Gg$, it follows from Theorems \ref{tetrablock} and \ref{thmG_2} that $(s\slash 2, s\slash 2, p) \in \E$. We have by Theorem \ref{thm_EQ} that $(0, s\slash 2, s\slash 2, p) \in \Q$. The converse follows by reversing the argument. Thus, the map $e_{\Gg}$ is a holomorphic embedding. Also, it is clear from Theorems \ref{tetrablock} and \ref{thmG_2} that $(x_1+x_2, x_3) \in \Gg$ for every $(x_1, x_2, x_3) \in \E$. Consequently, $(x_1+x_2, x_3) \in \Gg$ for every $(a, x_1, x_2, x_3) \in \Q$. Now, consider the map  $\psi_{\Gg}: \Q \to \Gg$ defined as $\psi_{\Gg}(a, x_1, x_2, x_3)=(x_1+x_2, x_3)$. Clearly, $\psi_{\Gg}$ is a holomorphic map that satisfies $\psi_{\Gg}\circ e_{\Gg}=\text{id}_{\Gg}$. Therefore, $\Gg$ is an analytic retract of $\Q$. 
	\end{proof}

	Next, we explore the connection of pentablock $\Pe$ with $\Q$. To do so, we need to state a few characterizations of the points in $\Pe$ and its closure $\Pbar$ from the literature. 
	
	\begin{thm}[\cite{AglerIV}, Theorems 5.2 \& 5.3]\label{pentablock}
		Let $a \in \C$ and let $(s, p)=(\lm_1+\lm_2, \lm_1\lm_2) \in \C^2$. For $(s, p) \in \Gg$ $($or, $(s, p) \in \Gamma_2$$)$, the following statements are equivalent:	
		\begin{enumerate}
			\item $(a, s, p) \in \mathbb{P}$ $($or, $(a, s, p) \in \Pbar$$)$; \vspace{0.1cm}
			\item $
			\displaystyle \sup_{z \in \D} \bigg|\displaystyle\frac{a(1-|z|^2)}{1-sz+pz^2}\bigg|<1$ $\left(\text{or, $
				\displaystyle \sup_{z \in \D} \bigg|\displaystyle\frac{a(1-|z|^2)}{1-sz+pz^2}\bigg|\leq 1$} \right)$; \vspace{0.2cm}
			\item   { \small
				$\displaystyle |a| < \frac{1}{2}|1-\overline{\lambda}_2\lambda_1|+\frac{1}{2}\sqrt{(1-|\lambda_1|^2)(1-|\lambda_2|^2)}$ 
				$\left(\text{or, $\displaystyle 
					|a|  \leq  \frac{1}{2}|1-\overline{\lambda}_2\lambda_1|+\frac{1}{2}\sqrt{(1-|\lambda_1|^2)(1-|\lambda_2|^2)}$} \right)$}.
		\end{enumerate}
	\end{thm}
	
	A description of the distinguished boundary $b\Pe$ of the pentablock is presented below. 
	
	\begin{thm}[\cite{AglerIV}, Theorem  8.4]\label{thmbP}
		A point $(a, s, p) \in \C^3$ lies in $b\Pe$ if and only if $|a|^2+|s|^2 \slash 4=1$ and $(s, p) \in b\Gamma_2$.
	\end{thm} 	
	
	We are now in a position to explore the relationships between the domains $\Pe$ and $\Q$. As discussed earlier, $(x_1+x_2, x_3) \in \Gg$ for every $(x_1, x_2, x_3) \in \E$ and a proof to this fact can be easily obtained using Theorems \ref{tetrablock} and \ref{thmG_2}. An analogous result holds for $\Pe$ and $\Q$, which is presented below.
	
	\begin{lem}\label{lem_H_P}
		If $(a, x_1, x_2, x_3) \in \Q$ $($respectively, $\CQ)$, then $(a, x_1+x_2, x_3) \in \mathbb{P} \ ($respectively, $\overline{\mathbb{P}})$.
	\end{lem} 
	
	\begin{proof}
		Let $(a, x_1, x_2, x_3) \in \Q$. Then  $(x_1, x_2, x_3) \in \E$ and
		\[
		\sup_{z_1, z_2 \in \mathbb{D}}\bigg|\frac{a\sqrt{(1-|z_1|^2)(1-|z_2|^2)}}{1-x_1z_1-x_2z_2+x_3z_1z_2}\bigg| <1.
		\]
		Consequently, we have
		\[
		\sup_{z \in \mathbb{D}}\bigg|\frac{a(1-|z|^2)}{1-(x_1+x_2)z+x_3z^2}\bigg| \leq \sup_{z_1, z_2 \in \mathbb{D}}\bigg|\frac{a\sqrt{(1-|z_1|^2)(1-|z_2|^2)}}{1-x_1z_1-x_2z_2+x_3z_1z_2}\bigg|  <1.
		\]
		By the discussion preceding the lemma, $(x_1+x_2,x_3)\in\Gg$. Thus, we have by Theorem \ref{pentablock} that $(a, x_1+x_2, x_3) \in \mathbb{P}$. Using continuity arguments,  $(a, x_1+x_2, x_3) \in \overline{\mathbb{P}}$ for every $(a, x_1, x_2, x_3) \in \CQ$.	
	\end{proof} 
	
	The following theorem establishes a relationship between the domains $\Pe$ and $\Q$. Recall that a continuous map $f$ between two topological spaces is said to be \textit{proper} if the inverse image of every compact set under $f$ is again compact. 
	
	\begin{thm}\label{thmP_in_H}$($Connection of hexablock with pentablock$)$
		A point $(a, s, p) \in \Pe$ $($respectively, $\Pbar)$ if and only if $(a, s\slash 2, s \slash 2, p) \in \Q$ $($respectively, $\CQ)$. In particular, the map $e_{\Pe}: \Pe \to \Q$ defined by $e_\Pe(a, s, p)=(a, s\slash 2, s \slash 2, p)$ is an embedding. Moreover, $e_\Pe$ is a proper holomorphic map such that $e_\Pe(b\Pe) \subseteq b\Q$, and $\Pe$ is an an analytic retract of $\Q$. 
	\end{thm}	
	
	\begin{proof}
		Let $(a, s, p) \in \mathbb{P}$. We have by Theorem \ref{pentablock} that $(s, p) \in \Gg$ and $\displaystyle \sup_{z \in \D} \bigg|\frac{a(1-|z|^2)}{1-sz+pz^2}\bigg|<1$. It follows from part-$(4)$ of Theorem \ref{tetrablock} that $(x_1, x_2, x_3) \in \E$ if and only if $|x_1-\overline{x}_2x_3|+|x_2-\overline{x}_1x_3| < 1-|x_3|^2$. Thus, $( s\slash 2, s \slash 2, p) \in \E$ if $|s-\overline{s}p|<1-|p|^2$. By part-$(2)$ of Theorem \ref{thmG_2}, the latter condition holds since $(s, p) \in \Gg$ and so, $(s\slash 2, s \slash 2, p) \in \E$. Let $(x_1, x_2, x_3)=(s \slash 2, s \slash 2, p)$. We have by Corollary 2.8 in \cite{Biswas-Pal} that  the supremum of 
		\[
		\bigg|\displaystyle\frac{\sqrt{(1-|z_1|^2)(1-|z_2|^2)}}{1-x_1z_1-x_2z_2+x_3z_1z_2}\bigg|\] 
		over $\D^2$ is attained uniquely at a point 
		$(\xi_1, \xi_2)\in \D^2$ with $\xi_1=\xi_2$. Therefore, we have  		
		\begin{equation*}
			\begin{split}
				\sup_{z_1, z_2 \in \mathbb{D}}\bigg|\frac{a\sqrt{(1-|z_1|^2)(1-|z_2|^2)}}{1-x_1z_1-x_2z_2+x_3z_1z_2}\bigg|
				= \bigg|\frac{a\sqrt{(1-|\xi_1|^2)(1-|\xi_2|^2)}}{1-x_1\xi_1-x_2\xi_2+x_3\xi_1\xi_2}\bigg|
				&=	\bigg|\frac{a(1-|\xi_1|^2)}{1-s\xi_1+p\xi_1^2}\bigg|\\
				& \leq \sup_{z \in \D} \bigg|\frac{a(1-|z|^2)}{1-sz+pz^2}\bigg|\\
				&<1,
			\end{split}
		\end{equation*}	
		where the last inequality follows from Theorem \ref{pentablock}. Consequently, $(a, s\slash 2, s\slash 2, p) \in \Q$ and so, $e_\Pe(\mathbb{P}) \subset \Q$. By continuity arguments, $(a, s\slash 2, s\slash 2, p) \in \CQ$ for every $(a, s, p) \in \Pbar$ and so, $e_\Pe(\overline{\mathbb{P}}) \subseteq \CQ$. We show that $e_\Pe(\partial \Pe) \subseteq \partial \Q$. Let $(a, s, p) \in \partial \Pe$.  We have proved above that $(a, s\slash 2, s\slash 2, p) \in \CQ$. It follows from Theorem \ref{pentablock} that at least one of the following holds:
		\[
		(s, p) \in \partial \Gg \quad \text{or} \quad (s, p) \in \Gg, \ \ \sup_{z \in \D} \left|\frac{a(1-|z|^2)}{1-sz+pz^2} \right|=1.
		\]
		Let $(s, p) \in \partial \Gg$. We have by part-$(3)$ of Theorem \ref{thmG_2} that $|s|^2+2|s-\overline{s}p|+|s^2-4p|=4$. It follows from part-$(3)$ of Theorem \ref{tetrablock} that $(s\slash 2, s\slash 2, p) \in \partial \E$. By definition of $\Q$, $(a, s\slash 2, s\slash 2, p) \notin \Q$. Now, suppose $(s, p) \in \Gg$ and
		$\displaystyle
		\sup_{z \in \D} \left|\frac{a(1-|z|^2)}{1-sz+pz^2} \right|=1.
		$
		By part-$(3)$ of Theorem \ref{thmG_2}, we have that $|s|^2+2|s-\overline{s}p|+|s^2-4p|<4$. It follows from part-$(2)$ of Theorem \ref{tetrablock} that $(s\slash 2, s\slash 2, p) \in \E$. Moreover, we have
		\[
		1=\sup_{z \in \D} \left|\frac{a(1-|z|^2)}{1-sz+pz^2} \right| \leq 	\sup_{z_1, z_2 \in \mathbb{D}}\bigg|\frac{a\sqrt{(1-|z_1|^2)(1-|z_2|^2)}}{1-(s\slash 2)z_1-(s\slash 2)z_2+pz_1z_2}\bigg| \leq 1,
		\]
		where the last inequality follows from \eqref{eqn5.3}. Hence, 
		\[
		\sup_{z_1, z_2 \in \mathbb{D}}\bigg|\frac{a\sqrt{(1-|z_1|^2)(1-|z_2|^2)}}{1-(s\slash 2)z_1-(s\slash 2)z_2+pz_1z_2}\bigg|=1
		\]
		and so, $(a, s\slash 2, s\slash 2, p) \notin \Q$. In either case, $(a, s\slash 2, s\slash 2, p) \in \partial \Q$. Hence, $e_\Pe$ is a proper map as it extends continuously to the closures and preserves the topological boundaries. Next, suppose $(a, s, p) \in b\Pe$. By Theorem \ref{thmbP}, $(s, p) \in b\Gamma_2$ and $|a|^2+|s|^2\slash 4=1$. It follows from Theorems \ref{thm6.1} and \ref{thm_bGam} that $(s\slash 2, s\slash 2, p) \in b\E$. Thus, Theorem \ref{thmbQ} implies that the point $(a, s\slash 2, s\slash 2, p) \in b\Q$ and so, $e_\Pe(b\mathbb{P}) \subseteq b\Q$. It only remains to show that $\Pe$ is an analytic retract of $\Q$. To see this, consider the map $\psi_{\Pe}: \Q \to \mathbb{P}$ defined by $\psi_{\Pe}(a, x_1, x_2, x_3)=(a, x_1+x_2, x_3)$. By Lemma \ref{lem_H_P}, $\psi_{\Pe}$ is a well-defined holomorphic map such that $\psi_{\Pe}\circ e_{\Pe}$ is the identity map on $\Pe$.
	\end{proof}

	The above theorem presents a description of $\Pe$ in terms of the hexablock. Our next result presents an alternative characterization of $\Pe$ in terms of the normed hexablock $\QN$. Recall that 
	\[
	\QN=\{\pi(A) : A \in M_2(\C), \ \|A\|<1\}, \ \text{where $\pi(A)=(a_{21}, a_{11}, a_{22}, \det(A))$ for $A=(a_{ij}) \in M_2(\C)$}.
	\]
	Since $\QN \subseteq \Q$, it follows from Theorem \ref{thmP_in_H} that $(a, x_1+x_2, x_3)\in \mathbb{P}$ for every $(a, x_1, x_2, x_3) \in \QN$. We prove its converse using the ideas and computations of the proof of Proposition 2.3 in \cite{AglerIV}. 
	
	\begin{prop}\label{prop4.7}
		Let $a \in \C$ and let $(s, p)=(\lambda_1+\lambda_2, \lambda_1\lambda_2) \in \Gg$. Then $(a, s, p) \in \mathbb{P}$ if and only if either of the following holds:
		\[
		\left(a, \frac{s}{2}, \frac{s}{2}, p\right) \in \QN \quad  \text{or} \quad |a| \leq |w|, \ \ \left(a, \frac{s}{2}+\zeta_0\sqrt{|w|^2-|a|^2}, \frac{s}{2}-\zeta_0\sqrt{|w|^2-|a|^2}, p\right) \in \QN,
		\] 
		where $w=(\lambda_1-\lambda_2)\slash 2$ and $\zeta_0 \in \T$ such that $\lambda_1-\lambda_2=\zeta_0|\lambda_1-\lambda_2|$.  
	\end{prop} 
	
	\begin{proof}
		It follows from Theorem \ref{thmG_2} that $|\lambda_1|, |\lambda_2|, |p|<1$. Let $(a, s, p) \in \mathbb{P}$ and define $c_{\pm}$ by
		\[
		c_{\pm}=\frac{1}{2}|1-\overline{\lambda}_2\lambda_1| \pm \frac{1}{2}\sqrt{(1-|\lambda_1|^2)(1-|\lambda_2|^2)}.
		\]
		Evidently, $0\leq c_{-}<c_{+}$. We have by part-$(5)$ of Theorem \ref{pentablock} that $|a|<c_+$. Then either $c_-<|a|<c_+$ or $|a| \leq c_-$. Let $c_{-}<|a|<c_+$. Consider the matrix given by
		\[
		A=\begin{pmatrix}
			s\slash 2 & w^2 \slash a\\
			a & s \slash 2
		\end{pmatrix} \quad \text{where} \quad w=(\lambda_1-\lambda_2)\slash 2.
		\]
		Then $w^2=s^2\slash 4-p$. A simple calculation gives that $
		|a|^2\det(I-A^*A)=-(|a|^2-c_-^2)(|a|^2-c_+^2)$ and so, $\det(I-A^*A)>0$. It follows from Lemma 4.4 in \cite{Biswas-Pal} that $\|A\|<1$ and thus, the point $\pi(A)=(a,  s \slash 2, s \slash 2, p) \in \QN$. We now assume that $|a| \leq c_-$. Again by simple computations, we have that $|w|^2=c_-c_+$ and so, $c_- \leq |w| \leq c_+$. Hence, $|a| \leq |w|$. Choose $\zeta_0 \in \mathbb{T}$ such that $\lambda_1-\lambda_2=\zeta_0|\lambda_1-\lambda_2|$ and define 
		\[
		A= \begin{pmatrix}
			\frac{s}{2}+\zeta_0\sqrt{|w|^2-|a|^2} & \zeta_0^2\overline{a}\\
			a & \dfrac{s}{2}-\zeta_0\sqrt{|w|^2-|a|^2}	
		\end{pmatrix}.
		\]
		A routine computation shows that $det(I-A^*A)=(1-|\lambda_1|^2)(1-|\lambda_2|^2)>0$ and $tr(I-A^*A)=2-(|\lambda_1|^2+|\lambda_2|^2)>0$. Thus, $I-A^*A>0$ and so, $\|A\|<1$. Hence, 
		\[
		\pi(A)= \left(a, \ \frac{s}{2}+\zeta_0\sqrt{|w|^2-|a|^2}, \ \dfrac{s}{2}-\zeta_0\sqrt{|w|^2-|a|^2}, \ p\right) \in \QN.
		\]
		The converse follows from Lemma \ref{lem_H_P} and the fact that $\QN \subset \Q$, which completes the proof.
	\end{proof} 
	
	Having established the relationships between $\Q$ and the domains $\E$, $\Gg$, and $\Pe$ arising from the $\mu$-synthesis problem, we now turn to the classical domains $\D^2$ and $\mathbb{B}_2$.
	
	\begin{thm}\label{thm_bidisc}$($Connection of hexablock with bidisc$)$
		A point $(a, x_3) \in \C^2$ belongs to $\D^2$ if and only if $(a, 0, 0, x_3) \in \Q$.  	In particular, the map $e_{\D^2}:\D^2 \to \Q$ given by $e_{\D^2}(a, x_3)=(a, 0, 0, x_3)$ is a holomorphic embedding, and $\D^2$ is an analytic retract of $\Q$. 
	\end{thm}

	\begin{proof}
		Given a point $(a, x_1, x_2, x_3) \in \Q$, we have that $(x_1, x_2, x_3)\in \E$ and for every $z_1, z_2 \in \D$, 
		\[
		\left|\frac{a\sqrt{(1-|z_1|^2)(1-|z_2|^2)}}{1-x_1z_1-x_2z_2+x_3z_1z_2}\right|
		<1.
		\]
		Putting $z_1=z_2=0$ in the above inequality, it follows that $|a|<1$. Since $(x_1, x_2, x_3) \in \E$, we have by \eqref{eqn_domains} that there exists $A=(a_{ij})_{i, j=1}^2 \in M_2(\C)$ with $\|A\|<1$ such that $(x_1, x_2, x_3)=(a_{11}, a_{22}, \det(A))$ and so, $|x_3|<1$. Thus, $(a, x_3) \in \D^2$ for every $(a, x_1, x_2, x_3) \in \Q$. Conversely, suppose $(a, x_3) \in \D^2$ and $x_1=x_2=0$. For $z_1, z_2 \in \DC$, note that
		$
		1-x_1z_1-x_2z_2+x_3z_1z_2=1+x_3z_1z_2,
		$	
		which is always non-zero and thus, $(x_1, x_2, x_3) \in \E$. Let $x=(x_1, x_2, x_3)$. We have by Proposition 2.1 in \cite{Biswas-Pal} that the supremum of $\bigg|\displaystyle\frac{\sqrt{(1-|z_1|^2)(1-|z_2|^2)}}{1-x_1z_1-x_2z_2+x_3z_1z_2}\bigg|$ over $(z_1, z_2)$ in $\DC^2$ is attained uniquely at a point in $\D^2$ given by
		\begin{align*}
			z_1(x)&=\frac{2\overline{\be}_1}{1+|\be_1|^2-|\be_2|^2+\sqrt{(1+|\be_1|^2-|\be_2|^2)^2-4|\be_1|^2}},\\	z_2(x)&=\frac{2\overline{\be}_2}{1+|\be_2|^2-|\be_1|^2+\sqrt{(1+|\be_2|^2-|\be_1|^2)^2-4|\be_2|^2}},
		\end{align*}
		where $\displaystyle 
		\be_1=\frac{x_1-\overline{x}_2x_3}{1-|x_3|^2}$ and $\displaystyle \be_2=\frac{x_2-\overline{x}_1x_3}{1-|x_3|^2}$. Since $x_1=x_2=0$, it is easy to see that 
		$z_1(x)=z_2(x)=0$. Then for every $z_1, z_2 \in \D^2$,
		\[
		\left|\frac{a\sqrt{(1-|z_1|^2)(1-|z_2|^2)}}{1-x_1z_1-x_2z_2+x_3z_1z_2}\right|=|a|<1.
		\]  
		Consequently, $(a, 0, 0, x_3) \in \Q$ and the map $e_{\D^2}: \D^2 \to \Q$ defined as $e_{\D^2}(a, x_3)=(a, 0, 0, x_3)$ is a holomorphic embedding. Consider the map $\psi_{\D^2}: \Q \to \D^2$ given by $\psi_{\D^2}(a, x_1, x_2, x_3)=(a, x_3)$. By above discussion, $\psi_{\D^2}$ is a well-defined holomorphic map. It is evident that $\psi_{\D^2}\circ e_{\D^2}$ is the identity map on $\D^2$. Therefore, $\D^2$ is an analytic retract of $\Q$. The proof is now complete. 
	\end{proof}	
	
	Next, we explore the connection of $\Q$ with biball, which is a bounded domain in $\C^2$ given by
	\[
	\mathbb{B}_2=\{(z_1, z_2) \in \C^2: |z_1|^2+|z_2|^2<1\}.
	\]
	The following lemma from \cite{Biswas-Pal} establishes a connection between the hexablock and biball.
	
	\begin{lem}[\cite{Biswas-Pal}, Lemma 5.5]\label{lem6.5}
		Let $(a, x_1, x_2, x_3) \in \C^4$. If $(a, x_1, x_2, x_3) \in \CQ$, then $|a|^2+|x_1|^2 \leq 1$. Moreover, if $(a, x_1, x_2, x_3) \in \Q$, then $|a|^2+|x_1|^2<1$.
	\end{lem}
	
	Capitalizing on the above lemma, we show that $\mathbb{B}_2$ is in fact an analytic retract of $\Q$. To do so, let us recall from \cite{Biswas-Pal} that the normed hexablock $\Q_N$ is the image of the norm unit ball in $M_2(\C)$ under the map $\pi: M_2(\C) \to \C^4$ given by $\pi(A)=(a_{21}, a_{11}, a_{22}, \det(A))$ for $A=(a_{ij})_{i, j=1}^2 \in M_2(\C)$. As discussed in the `Introduction', we have that $\QN$ is contained in $\Q$.  
	
	\begin{thm} \label{lem_biball_hexa}$($Connection of hexablock with biball$)$
		Let $(a, x) \in \C^2$. Then $(a,x) \in \mathbb{B}_2$ if and only if $(a, x, 0, 0) \in \Q$. In particular, the map $e_{\mathbb{B}_2}: \mathbb{B}_2 \to \Q$ defined by $e_{\mathbb{B}_2}(a, x)=(a, x, 0, 0)$ is a holomorphic embedding, and $\mathbb{B}_2$ is an analytic retract of $\Q$.
	\end{thm} 
	
	\begin{proof}
		By Lemma \ref{lem6.5}, if $(a, x, 0, 0) \in \Q$, then $(a, x) \in \mathbb{B}_2$. Conversely, let $(a, x) \in \mathbb{B}_2$. Define 
		\[ 
		B=\begin{pmatrix}
			x & 0\\
			a & 0
		\end{pmatrix} \quad \text{and so,} \quad B^*B=\begin{pmatrix}
			|a|^2+|x|^2 & 0\\
			0 & 0
		\end{pmatrix}.
		\] 
		Thus, $\|B\|^2=\|B^*B\|=|a|^2+|x|^2$.  If $(a, x) \in \mathbb{B}_2$, then $\pi(B)=(a, x, 0, 0) \in \QN$. Since $\QN \subset \Q$, it follows that $(a, x, 0, 0) \in \Q$ and so, the map $e_{\mathbb{B}_2}$ is a holomorphic embedding. Consider the map $\psi_{\mathbb{B}_2}: \Q \to \mathbb{B}_2$ given by  $\psi_{\mathbb{B}_2}(a, x_1, x_2, x_3)=(a, x_1)$. Clearly, $\psi_{\mathbb{B}_2}$ is a well-defined holomorphic map such that  $\psi_{\mathbb{B}_2}\circ e_{\mathbb{B}_2}$ is the identity map on $\mathbb{B}_2$. Consequently, $\mathbb{B}_2$ is an analytic retract of $\Q$.
	\end{proof}

	It is evident from the above lemma that $(a,x) \in \overline{\mathbb{B}}_2$ if and only if $(a, x, 0, 0) \in \CQ$. This gives a continuous injective map from $\overline{\mathbb{B}}_2$ into $\CQ$. Our next result presents another way to construct a continuous injective map from $\overline{\mathbb{B}}_2$ into $\CQ$.
	
	\begin{lem}\label{lem_B2_HexaII}
		Let $(a, x) \in \C^2$. Then $(a,x) \in \overline{\mathbb{B}}_2$ if and only if $(a, x, \overline{x}, 1) \in \CQ$.
	\end{lem}
	
	\begin{proof}
		Suppose $(a, x, \overline{x}, 1) \in \CQ$. 	We have by Lemma \ref{lem6.5} that $(a,x) \in \overline{\mathbb{B}}_2$. To see the converse, we first assume that $(a, x) \in \mathbb{B}_2$. Then $|x| \leq \sqrt{|a|^2+|x|^2}<1$ and by Theorem \ref{thm6.1}, $(x, \overline{x}, 1) \in b\E$. Let $(x_1, x_2, x_3)=(x, \overline{x}, 1)$, which belongs to $b\E$ with $|x_1|<1$. It follows from Theorem 2.12 in \cite{Biswas-Pal} that
		\[
		\sup_{z_1, z_2 \in \D} \bigg|a\frac{\sqrt{(1-|z_1|^2)(1-|z_2|^2)}}{1-x_1z_1-x_2z_2+x_3z_1z_2}\bigg|=\frac{|a|}{\sqrt{1-|x_1|^2}}=\frac{|a|}{\sqrt{1-|x|^2}}<1.
		\]
		Therefore, we have by \eqref{eqn5.3} that	 $(a, x, \overline{x}, 1) \in \CQ$. Now, suppose $(a, x) \in \overline{\mathbb{B}}_2$. Define $r_n=1-1\slash n$ for every $n \in \mathbb{N}$. Then $(r_n a,r_n x) \in \mathbb{B}_2$ for all $n \in \mathbb{N}$ and so, $(r_n a, r_n x, r_n \overline{x}, 1) \in \CQ$. Since $(r_n a, r_n x, r_n \overline{x}, 1)$ converges to $(a, x, \overline{x}, 1)$ as $n \to \infty$, we have that $(a, x, \overline{x}, 1) \in \CQ$, which completes the proof.
	\end{proof} 
	
	The relationships of $\Q$ with $\D^2$, $\mathbb{B}_2$, $\Gg$, $\E$ and $\Pe$ are summarized below. The result follows directly from Theorems \ref{thm_EQ}, \ref{thm_GQretract}, \ref{thmP_in_H}, \ref{lem_biball_hexa} and \ref{thm_bidisc}. The details are left to the reader.
	
	\begin{thm}\label{thm_final_embed}
		For $(a, s, p, x_1, x_2, x_3) \in \C^6$, the following hold:
		\begin{enumerate}
			\item[$(1)$] $(a, x_3) \in \D^2$ $($respectively, $\DC^2$$)$ if and only if $(a, 0, 0, x_3) \in \Q$ $($respectively, $\CQ$$)$; \vspace{0.1cm}
			
			\item[$(2)$] $(a, x_1) \in \mathbb{B}_2$ $($respectively, $\overline{\mathbb{B}}_2$$)$ if and only if $(a, x_1, 0, 0) \in \Q$ $($respectively, $\CQ$$)$; \vspace{0.1cm}
			
			\item[$(3)$] $(s, p) \in \Gg$ $($respectively, $\Gamma_2$$)$ if and only if $(0, s\slash 2, s\slash 2, p) \in \Q$ $($respectively, $\CQ$$)$; \vspace{0.1cm}
			
			\item[$(4)$] $(x_1, x_2, x_3) \in \E$ $($respectively, $\Ebar$$)$ if and only if $(0, x_1, x_2, x_3) \in \Q$ $($respectively, $\CQ$$)$; \vspace{0.1cm}
			
			\item[$(5)$] $(a, s, p) \in \Pe$ $($respectively, $\Pbar$$)$ if and only if $(a, s\slash 2, s\slash 2, p) \in \Q$ $($respectively, $\CQ$$)$.
		\end{enumerate}	
	\end{thm}

	We conclude this section with an application of the above theorem. To do so, we recall from \cite{Harris, HarrisII} the notion of $J^*$-algebras and its rank. For complex Hilbert spaces $\mathcal{H}$ and $\mathcal{H}'$, we denote by  $\mathcal{B}(\mathcal{H}, \mathcal{H}')$ the Banach space of all bounded linear operators from $\mathcal{H}$ to $\mathcal{H}'$.

	\begin{defn} A \textit{$J^*$-algebra} is a closed linear subspace $\mathcal{L}$ of$\mathcal{B}(\mathcal{H}, \mathcal{H}')$, for some Hilbert spaces $\mathcal{H}, \mathcal{H}'$, such that $TT^*T \in \mathcal{L}$ for all $T \in \mathcal{L}$. The \textit{rank} of a $J^*$-algebra $\mathcal{L}$ is the supremum of the number of non-zero elements in the spectrum of $T^*T$ over all $T \in \mathcal{L}$.
	\end{defn}
	
	The domains $\Gg, \E$ and $\Pe$ are not analytic retracts of the open unit ball of any $J^*$-algebra of finite rank. An interested reader is referred to the works \cite{Young, KosinskiII}. Indeed, the proof in \cite{KosinskiII} works for any domain such that $\Gg$ is its analytic retract. One such domain is $\Q$ as proved in Theorem \ref{thm_GQretract}. Using this, we simply have that $\Q$ is not an analytic retract of the open unit ball of any $J^*$-algebra of finite rank. One can also use the theory of $\E$ and $\Pe$ to prove the same result.
	
	\begin{prop}\label{prop_JQ}
		The hexablock is not an analytic retract of the open unit ball of any $J^*$-algebra of finite rank.
	\end{prop}
	
	\begin{proof}
		Assume on the contrary that there is a $J^*$-algebra $\mathcal{L}$ of finite rank with analytic maps $\widetilde{\phi}: \Q \to U$ and $\widetilde{\psi}: U \to \Q$ such that $\widetilde{\psi} \circ \widetilde{\phi}=\text{id}_\Q$, where $U$ is the open unit ball in $\mathcal{L}$. By Theorem \ref{thm_EQ}, the maps given by
		\[
		\phi:\E \to \Q, \ \phi(x_1, x_2, x_3)=(0, x_1, x_2, x_3)\quad  \text{and} \quad \psi: \Q \to \E,\  \psi(a, x_1, x_2, x_3)=(x_1, x_2, x_3)
		\]
		are holomorphic. Evidently, $\psi\circ\phi=\text{id}_\E$. Then $\phi_0=\widetilde{\phi}\circ \phi : \E \to U$ and $\psi_0=\psi \circ \widetilde{\psi} : U \to \E$ are holomorphic maps such that $\psi_0 \circ \phi_0=\text{id}_\E$. This shows that $\E$ is an analytic retract of $U$ which is a contradiction to the fact that $\E$ cannot be an analytic retract of the open unit ball of any $J^*$-algebra of finite rank (see Theorem 3.1 in \cite{Young}). The proof is now complete. 
	\end{proof} 
	
	One can also apply the theory of pentablock to prove the above result. We briefly discuss it here. Suppose there is a $J^*$-algebra $\mathcal{L}$ of finite rank with analytic maps $\widetilde{\phi}: \Q \to U$ and $\widetilde{\psi}: U \to \Q$ such that $\widetilde{\psi} \circ \widetilde{\phi}=\text{id}_\Q$. Here, $U$ is the open unit ball in $\mathcal{L}$. By Theorem \ref{thmP_in_H}, the holomorphic maps
	\[
	\phi:\Pe \to \Q, \ \phi(a, s ,p)=(a, s\slash 2, s\slash 2, p) \quad  \text{and} \quad \psi: \Q \to \Pe,\  \psi(a, x_1, x_2, x_3)=(a, x_1+x_2, x_3)
	\]
	satisfy $\psi\circ\phi=\text{id}_\Pe$. Then $\phi_0=\widetilde{\phi}\circ \phi : \Pe \to U$ and $\psi_0=\psi \circ \widetilde{\psi} : U \to \Pe$ are holomorphic maps such that $\psi_0 \circ \phi_0=\text{id}_\Pe$. This shows that $\Pe$ is an analytic retract of $U$, which is a contradiction as $\Pe$ is not an analytic retract of the open unit ball of any $J^*$-algebra of finite rank (see Proposition 17 in \cite{KosinskiII}). Therefore, $\Q$ is not an analytic retract of the open unit ball of any $J^*$-algebra of finite rank.
	
	\section{Rational $\CQ$-inner functions}\label{inner}
	
	\noindent Let $\Omega$ be a bounded domain in $\C^n$. Recall that a holomorphic map $f=(f_1, \dotsc, f_n): \D \to \overline{\Omega}$ is said to be rational $\overline{\Omega}$-inner (or simply, rational $\Omega$-inner) if $f_1, \dotsc, f_n$ are rational maps that are holomorphic on $\D$ and $f(\T) \subseteq b\Omega$. Here, $b\Omega$ denotes the distinguished boundary of $\Omega$. A comprehensive description of rational $\Gamma_2$-inner functions was provided by Agler et al. in \cite{AglerV}. Building on the key insights from \cite{AglerV}, Alsalhi and Lykova \cite{Alsalhi} characterized rational $\Ebar$-inner functions. Following this, the rational $\overline{\Pe}$-inner functions were studied in \cite{Alsheri} and \cite{Jindal}. In their work, the authors of \cite{Alsheri, Jindal} leveraged the representation of rational $\Gamma_2$-inner functions to prove that $(s, p)$ is a rational $\Gamma_2$-inner function if $(a, s, p)$ is a rational $\Pbar$-inner function. This reduces the problem of characterizing rational $\Pbar$-inner function to finding necessary and sufficient conditions on the first component of $(a,s,p)$ so that $(a, s, p)$ becomes a rational $\Pbar$-inner function. Following the same approach, we first show that if $(a, x_1, x_2, x_3)$ is a rational $\CQ$-inner function, then $(x_1, x_2, x_3)$ is a rational $\Ebar$-inner function. Finally, building on the theory of rational $\Ebar$-functions from \cite{Alsalhi}, we provide a characterization of the rational $\CQ$-inner functions in this section. To do so, we recall the notions of inner and outer functions in $H^\infty$, where $H^\infty$ is the space of all bounded holomorphic functions on $\D$ equipped with the supremum norm.
	
	\smallskip 
	
	An \textit{inner function} is an analytic map $f: \D\to \overline{\D}$ such that the radial limit ${\displaystyle \lim_{r \to 1^{-}} f(r\lambda)}$	exists and belongs to $\T$ for almost all $\lambda \in \T$ with respect to Lebesgue measure. A \textit{finite Blaschke product} is a function of the form 
	\[
	B(\lm)=e^{i\theta} \ \overset{n}{\underset{i=1}{\prod}} \ \frac{\lm-\al_i}{\overline{\alpha}_i\lm-1} \qquad (\lm \in \overline{\D}),
	\]
	where $\theta \in \mathbb{R}$ and $\alpha_1, \dotsc, \alpha_n \in \D$. The number of factors in the product $B$ is said to be its \textit{degree}. It is evident that a finite Blaschke product is a rational $\DC$-inner function. In fact, these are the only rational $\DC$-inner maps as the following result shows.
	
	\begin{thm}[\cite{Blaschke} \& \cite{Grimm}, Theorem 3] 
		Any rational $\overline{\D}$-inner function is a finite Blaschke product. 
	\end{thm}

	An \textit{outer function} $f$ is a function on $\D$ that can be represented as
	\[
	f(\lambda)=\xi \exp \left[ \frac{1}{2\pi} \int_{0}^{2\pi} \frac{e^{it}+\lambda}{e^{it}-\lambda}\log(g(t))dt \right],
	\]
	where $g(t) \geq 0, \log (g(t)) \in L^1$ and $\xi \in \T$. 	An \textit{inner function} is an analytic map $f: \D\to \overline{\D}$ such that the radial limit $
	{\displaystyle \lim_{r \to 1^{-}} f(r\lambda)}$	exists and belongs to $\T$ for almost all $\lambda \in \T$ with respect to Lebesgue measure. The outer function $f$ is in $H^\infty$ if and only if $g \in L^\infty$. Moreover, we have that $|f(e^{it})|=g(t)$ almost everywhere in this case. It is evident that an outer function has no zeros in $\D$ and the only functions which are both outer and inner are the constant functions of unit modulus. We mention a very well-known result related to inner and outer functions.
	
	\begin{thm}[\cite{Nagy}, Chapter III]
		Every function $f \in H^\infty$ that is not identically $0$ has a canonical factorization 
		$
		f=f_{in}f_{out}
		$
		into the product of an inner function $f_{in}$ and an outer function $f=f_{out}$ which is unique up to constant factors of modulus $1$.
	\end{thm}
	
	This canonical factorization of non-zero functions in $H^\infty$ is referred to as the \textit{inner-outer factorization}. Next, we define the notion of $\CQ$-inner function.
	
	\begin{defn}
		An \textit{$\CQ$-inner} function is an analytic map $f: \D \to \CQ$ such that 
		the radial limit ${\displaystyle	\lim_{r \to 1^{-}} f(r\lambda )}$ exists and belongs to $b\Q$ for almost all $\lambda \in \T$ with respect to Lebesgue measure. Similarly, an $\Ebar$-inner function is an analytic map $g: \D \to \Ebar$ such that  ${\displaystyle	\lim_{r \to 1^{-}} g(r\lambda )}$ exists and belongs to $b\E$ for almost all $\lambda \in \T$ with respect to Lebesgue measure. 
	\end{defn}
	
	By Fatou's theorem, the limit ${\displaystyle \lim_{r \to 1^{-}} f(r\lambda)}$ exists for almost every $\lambda \in \T$ whenever $f$ is either an $\Ebar$-inner function or $\CQ$-inner function. It is not difficult to see that every rational $\CQ$-inner function admits a continuous extension to $\DC$. To see this, let $f: \D \to \CQ$ be a rational $\CQ$-inner function. Then $f$ is a rational bounded function with no singularities in $\overline{\D}$ and so, $f$ is continuous on $\overline{\D}$. Indeed, the map $\widetilde{f}: \T \to b\Q$ defined as 
	\[
	\widetilde{f}(\lambda)=	\lim_{r \to 1^{-}} f(r\lambda)
	\] 
	is a continuous extension of $f$ with $\widetilde{f}(\T) \subseteq b\Q$. Similarly, every rational $\Ebar$-inner function admits a continuous extension to $\DC$. A canonical relation between $\Ebar$ and $\CQ$ follows from \eqref{eqn5.3}, namely, $(x_1, x_2, x_3) \in \Ebar$ for every $(a, x_1, x_2, x_3)\in \CQ$. A similar connection between $\CQ$-inner functions and $\Ebar$-inner functions exists, as the following result illustrates.
	
	\begin{lem}\label{lem10.5} 
		Let $f=(a, x_1, x_2, x_3) : \D \to \CQ$ be an analytic function. Then 
		$
		h=(x_1, x_2, x_3): \D \to \Ebar 
		$
		is an analytic function. Moreover, if $f$ is an $\CQ$-inner function, then
		$h$ is an $\Ebar$-inner function.
	\end{lem}
	
	\begin{proof}
		By assumption, $f=(a, x_1, x_2, x_3)$ is analytic on $\D$ and for all $\lambda \in \D$, we have 
		\[
		f(\lambda)=(a_1(\lambda), \ x_1(\lambda), \ x_2(\lambda), \ x_3(\lambda)) \in \CQ.
		\]
		We have by \eqref{eqn5.3} that $(x_1(\lambda), x_2(\lambda), x_3(\lambda)) \in \Ebar$ for all $\lambda \in \D$. Thus $h=(x_1, x_2, x_3): \D \to \Ebar$ is a well-defined analytic function on $\D$. Assume that $f$ is an $\CQ$-inner function, and so, for almost all $\lambda \in \T, f(\lambda) \in b\Q$. For almost all $\lambda \in \T$, it follows from Theorem \ref{thmbQ} that $h(\lambda) \in b\E$. Hence, $h$ is an $\Ebar$-inner function. 
	\end{proof} 
	
	The above result shows that one obtains an $\Ebar$-function from an $\CQ$-inner function. We now show how an $\Ebar$-inner function gives rise to an $\CQ$-inner function. 
	
	\begin{prop}\label{prop_Ein_Hin}
		Let $x=(x_1, x_2, x_3): \D \to \Ebar$ be an $\Ebar$-inner function. Then there exists an $\CQ$-inner function such that $f=(\textbf{a}, x_1, x_2, x_3)$ for some holomorphic map $\textbf{a}$ on $\D$.   
	\end{prop}
	
	\begin{proof}
		Let $x=(x_1, x_2, x_3): \D \to \Ebar$ be an $\Ebar$-inner function. It follows from the proof of Theorem 2.2.5 in \cite{AbouhajarI} that there exists an analytic map $F: \D \to M_2(\C), \lambda \mapsto [F_{ij}(\lambda)]_{i, j=1}^2$ such that $\|F(\lambda)\| \leq 1$ for all $\lm \in \D$, $|F_{21}|=|F_{12}|$ on $\T$ and $x=(F_{11}, F_{21}, \det(F))$. Since $\CQN$ is the image of closed norm unit ball in $M_2(\C)$ under the map $\pi(A)=(a_{21}, a_{11}, a_{22}, \det(A))$ for $A=(a_{ij}) \in M_2(\C)$, and $\CQN \subseteq \CQ$, it follows from \eqref{eqn5.3} that
		$
		\pi(F(\lambda))=(F_{21}(\lambda), \ x_1(\lambda), \ x_2(\lambda), \ x_3(\lambda)) \in \CQN \subseteq \CQ
		$
		for all $\lambda \in \D$. Thus, the map
		\[
		f: \D \to \CQ, \quad f(\lambda)=\pi(F(\lambda))=(F_{21}(\lambda), x_1(\lambda), x_2(\lambda), x_3(\lambda))
		\]
		is well-defined and analytic. Since $x=(x_1, x_2, x_3)$ is an $\Ebar$-inner function, $x_1=\overline{x}_2x_3$ and $|x_3|=1$ almost everywhere on $\T$. Thus, $|F_{21}(\lambda)|^2=|F_{21}(\lambda)F_{12}(\lambda)|=|x_1(\lambda)x_2(\lambda)-x_3(\lambda)|=1-|x_1(\lambda)|^2$ for almost all $\lambda \in \T$. We have by Theorem \ref{thmbQ} that $f(\lambda) \in b\Q$ for almost all $\lambda \in \T$. Thus, $f$ is an $\CQ$-inner function.
	\end{proof} 
	
	Note that for any $\CQ$-inner function $f=(a, x_1, x_2, x_3)$ from $\D$ to $\CQ$, the map
	$ \lambda \mapsto a(\lambda)$ is a bounded analytic function from $\D$ to $\overline{\D}$ as $\CQ \subseteq \overline{\D}^4$. Let $a\ne 0$. Then $a$ admits an inner-outer factorization and let it be given by $a=a_{in}a_{out}$. The next results shows that $(a_{out}, x_1, x_2, x_3)$ is also an $\CQ$-inner function.
	
	\begin{prop}
		Let $f=(a, x_1, x_2, x_3)$ be an $\CQ$-inner function and $a \ne 0$. Let $a_{in}a_{out}$ be the inner-outer factorization of $a$. Then $\tilde{f}=(a_{out}, x_1, x_2, x_3)$ is an $\CQ$-inner function.
	\end{prop}
	
	\begin{proof}
		Let $f=(a, x_1, x_2, x_3)$ be an $\CQ$-inner function. Then $(a(\lm), x_1(\lm), x_2(\lm), x_3(\lm)) \in \CQ$ for all $\lm \in \D$. We have by \eqref{eqn5.3} that
		\[
		\bigg|\frac{a(\lm)\sqrt{(1-|z_1|^2)(1-|z_2|^2)}}{1-x_1(\lm)z_1-x_2(\lm)z_2+x_3(\lm)z_1z_2}\bigg| \leq 1 \quad \text{for all} \quad \lambda, z_1, z_2 \in \D.
		\]
		Since $a_{in}$ is inner, we have that $|a_{in}(\lm)|=1$ for almost all $\lambda \in \T$. Therefore,  
		\[
		\bigg|a_{in}(\lm)\frac{a_{out}(\lm)\sqrt{(1-|z_1|^2)(1-|z_2|^2)}}{1-x_1(\lm)z_1-x_2(\lm)z_2+x_3(\lm)z_1z_2}\bigg|=	\bigg|\frac{a_{out}(\lm)\sqrt{(1-|z_1|^2)(1-|z_2|^2)}}{1-x_1(\lm)z_1-x_2(\lm)z_2+x_3(\lm)z_1z_2}\bigg| \leq 1
		\]
		for every $z_1, z_2 \in \D$ and for almost all $\lambda \in \T$. Note that for every $z_1, z_2 \in \D$, the function 
		\[
		\lambda \mapsto \frac{a_{out}(\lm)\sqrt{(1-|z_1|^2)(1-|z_2|^2)}}{1-x_1(\lm)z_1-x_2(\lm)z_2+x_3(\lm)z_1z_2}
		\]
		is analytic on $\D$. By maximum modulus principle, we have for all $z_1, z_2 \in \D$ that
		\[
		\bigg|\frac{a_{out}(\lm)\sqrt{(1-|z_1|^2)(1-|z_2|^2)}}{1-x_1(\lm)z_1-x_2(\lm)z_2+x_3(\lm)z_1z_2}\bigg| \leq 1, \quad \text{for all} \ \lm \in \D.
		\]
		It follows from  \eqref{eqn5.3} that $(a_{out}(\lm), \ x_1(\lm), \ x_2(\lm), \ x_3(\lm)) \in \CQ$ for all $\lm \in \D$. Therefore, $\tilde{f}=(a_{out}, x_1, x_2, x_3)$ is a holomorphic map from $\D$ to $\CQ$. It only remains to show that 
		\[
		\tilde{f}(\lm)=(a_{out}(\lm), x_1(\lm), x_2(\lm), x_3(\lm)) \in b\Q
		\]
		almost everywhere on $\T$. It follows from Theorem \ref{thmbQ} that 
		\[
		b\Q=\left\{(a, x_1, x_2, x_3) \in \C^4 : (x_1, x_2, x_3) \in b\E, \ |a|^2+|x_1|^2=1    \right\}.
		\] 
		By Lemma \ref{lem10.5}, $(x_1(\lm), x_2(\lm), x_3(\lm)) \in b\E$ for almost all $\lambda \in \T$. Since $f=(a, x_1, x_2, x_3)$ is an $\CQ$-inner function, we have that $|a(\lm)|^2+|x_1(\lm)|^2=1$  for almost all $\lambda \in \T$. Since $a=a_{in}a_{out}$ is the inner-outer factorization of $a$ and $|a_{in}(\lm)|=1$ for almost all $\lm \in \T$, we have
		\[
		|a_{out}(\lm)|^2+|x_1(\lm)|^2=|a(\lm)|^2+|x_1(\lm)|^2=1 \quad \text{ for almost all } \lambda \in \T.
		\]
		Therefore, $\tilde{f}=(a_{out}, x_1, x_2, x_3)$ is an $\CQ$-inner function and the proof is complete.
	\end{proof} 
	
	We recall the notion of the degree of a rational $\Ebar$-inner function $x: \D \to \Ebar$. Needless to mention, $x$ is a rational bounded function with no singularities in $\overline{\D}$ and so, $x$ is continuous on $\overline{\D}$. Moreover, $x(\T) \subseteq b\E$ and so, $x: \T \to b\E$ is a continuous map too. Then we have a group homomorphism $x_*: \mathbb{Z}=\pi_1(\T) \to \pi_1(b\E)$ of fundamental groups induced by $x$. Capitalizing on the map $x_*$, Alsalhi and Lykova \cite{Alsalhi} defined the degree of $x$ in the following way. 
	
	\begin{defn}
		Let $x$ be a rational $\Ebar$-inner function and let $x_*:\mathbb{Z}=\pi_1(\T) \to \pi_1(b\E)$ be the homomorphism of fundamental groups induced by $x$. The degree $deg(x)$ of $x$ is defined as $x_*(1)$. 
	\end{defn}
	
	We briefly discuss a useful characterization of the degree for a rational $\Ebar$-function obtained in \cite{Alsalhi}. Let $x=(x_1, x_2, x_3)$ be a rational $\Ebar$-inner function. It follows from Lemma 4.3 in \cite{Alsalhi} that $x_3$ is a rational inner function on $\D$ and so, $x_3$ is a finite Blaschke product. We have by Proposition 4.14 in \cite{Alsalhi} that $deg(x)$ is equal to the degree $deg(x_3)$ of the finite Blaschke product $x_3$, that is, the number of factors in the Blaschke product. Next, we need a few preparatory results and notations to describe the rational $\CQ$-inner functions. For a polynomial $g$ of degree less than or equal to $n$, let us define the polynomial $g^{\sim n}$ by 
	\[
	g^{\sim n}(\lm)=\lm^n\overline{g(1\slash \overline{\lm})}.
	\]
	The following result gives a characterization for rational $\Ebar$-inner functions.
	
	\begin{thm}[\cite{Alsalhi},  Theorem 4.15]\label{thm_Einner}
		Let $x=(x_1, x_2, x_3)$ be a rational $\Ebar$-inner function of degree $n$. Then there exist polynomials $E_1, E_2, D$ of degree at most $n$ such that the following hold. 
		\begin{equation}\label{eqn10.2}
			\begin{split}
				&(1)  \	\ D(\lm) \ne 0 \  \text{on} \  \overline{\D}, \qquad (2) \ E_1=E_2^{\sim n} \ \text{on} \ \overline{\D}, 
				\qquad (3) \  |E_1(\lm)|, |E_2(\lm)| \leq |D(\lm)| \  \text{for} \ \lm \in \overline{\D}, \\
				&(4) \ \	x_1=\frac{E_1}{D}\ \ \text{on} \ \ \overline{\D},\qquad (5) \ \	x_2=\frac{E_2}{D} \ \ \text{on} \ \ \overline{\D} \  \qquad (6) \ \	x_3=\frac{D^{\sim n}}{D} \ \ \text{on} \ \ \overline{\D}. \\
			\end{split}	
		\end{equation}
		Conversely, if $E_1, E_2$ and $D$ are polynomials of degree at most $n$ that satisfy conditions $(1), (2)$ and $(3)$, then $x=(x_1, x_2, x_3)$ with $(x_1, x_2, x_3)$ defined as in $(4)-(6)$ is a rational $\Ebar$-inner function of degree at most $n$.
	\end{thm}

	\begin{thm}[\cite{NagyII}, Section 53]\textbf{$($Fej\'{e}r-Riesz Theorem$)$}\label{Fejer}
		If $f(\lm)=\overset{n}{\underset{i=-n}{\sum}}a_i\lm^i$ is a trigonometric polynomial of degree $n$ with $f(\lm) \geq 0$ for all $\lm \in \T$, then there exists an analytic polynomial $D(\lm)=\overset{n}{\underset{i=0}{\sum}}b_i\lm^i$ of degree $n$ such that $D(\lm) \ne 0$ for all $\lm \in \D$ and $f(\lm)=|D(\lm)|^2$	
		for all $\lm \in \T$.
	\end{thm}
	
	Let $x=(x_1, x_2, x_3)$ be an $\Ebar$-inner function. We have by Proposition \ref{prop_Ein_Hin} the existence of a holomorphic function on $\D$ such that $f=(a, x_1, x_2, x_3)$ is an $\CQ$-inner function. We explicitly find such an $f$ when $x$ is a rational $\Ebar$-inner function as the next theorem explains.
	
	\begin{thm}\label{thm10.14}
		Let $x=(x_1, x_2, x_3)$ be a rational $\Ebar$-inner function of degree $n$ and let $E_1, E_2, D$ be polynomials as in \eqref{eqn10.2}. If $A$ is a polynomial such that $A(\lambda) \ne 0$ for all $\lm \in \D$ and $|A(\lambda)|^2=|D(\lambda)|^2-|E_1(\lambda)|^2$ for all  $\lm \in \T$, then 
		\[
		f=\bigg(cB\frac{A}{D}, \ \frac{E_1}{D},\ \frac{E_2}{D}, \ \frac{D^{\sim n}}{D}\bigg)
		\]	
		defines a rational $\CQ$-inner function for every finite Blaschke product $B$ and for every $c \in \T$.
	\end{thm}

	\begin{proof}
		Let $a, x_1, x_2, x_3$ be defined by 
		\[
		a=cB\frac{A}{D}, \quad x_1=\frac{E_1}{D}, \quad x_2=\frac{E_2}{D} \quad \text{and} \quad x_3=\frac{D^{\sim n}}{D}.
		\]
		Since $D(\lm) \ne 0$ on $\overline{\D}$, it follows that $a, x_1, x_2, x_3$ are continuous on $\overline{\D}$. Consequently $f$ is a continuous function on $\overline{\D}$. Let us show that $f=(a, x_1, x_2, x_3)$ is a rational $\CQ$-inner function. We have to prove that $f(\lm) \in \CQ$ and, for almost all $\lm \in \T, f(\lm) \in b\Q$. By assumption $x=(x_1, x_2, x_3): \D \to \Ebar$ is a rational $\Ebar$-inner function, which means $|x_3(\lambda)|=1$ for almost all $\lm \in \T$.  Now we need to show that $|a(\lm)|^2=1-|x_1(\lm)|^2$ for almost all $\lm \in \T$. Indeed, for all $\lm$ in $\T$, we have 
		\[
		|a(\lm)|^2=\bigg|cB(\lm)\frac{A(\lm)}{D(\lm)}\bigg|^2=\bigg|\frac{A(\lm)}{D(\lm)}\bigg|^2=1-\bigg|\frac{E_1(\lm)}{D(\lm)}\bigg|^2=1-|x_1(\lm)|^2.
		\] 	
		By Theorem \ref{thmbQ}, $f(\T)\subseteq b\Q$. Let us show that $f=(a, x_1, x_2, x_3)=\displaystyle\bigg(cB\frac{A}{D}, \ \frac{E_1}{D},\ \frac{E_2}{D}, \ \frac{D^{\sim n}}{D}\bigg)$ maps $\D$ to $\CQ$, that is, 
		\[
		f(\lm)=(a(\lm), \ x_1(\lm), \ x_2(\lm), \ x_3(\lm)) \in \CQ \quad \text{for all} \ \lm \in \D.
		\]
		For any $\lm \in \D$, we must have $(x_1(\lm), \ x_2(\lm), \ x_3(\lm)) \in \Ebar$ as $x=(x_1, x_2, x_3)$ is an $\Ebar$-inner function. In view of \eqref{eqn5.3}, we need to show that
		\[
		\bigg|\frac{a(\lm)\sqrt{(1-|z_1|^2)(1-|z_2|^2)}}{1-x_1(\lm)z_1-x_2(\lm)z_2+x_3(\lm)z_1z_2}\bigg| \leq 1 \quad \text{for all} \quad \lambda, z_1, z_2 \in \D.
		\]
		For every $z_1, z_2 \in \D$, consider the function 
		\[
		\Pz(f(.)): \overline{\D} \to \C, \quad  \lambda \mapsto \frac{a(\lm)\sqrt{(1-|z_1|^2)(1-|z_2|^2)}}{1-x_1(\lm)z_1-x_2(\lm)z_2+x_3(\lm)z_1z_2}.
		\]
		For all $\lambda \in \overline{\D}$, we have
		$
		x(\lambda) \in \Ebar$ and so, $1-x_1(\lm)z_1-x_2(\lm)z_2+x_3(\lm)z_1z_2 \ne 0$ for any $z_1, z_2 \in\D$.
		Therefore, $\Pz(f(.))$ is continuous on $\overline{\D}$ and holomorphic on $\D$. Since $f(\lm) \in b\Q \subseteq \CQ$ for $\lambda \in \T$, we have by \eqref{eqn5.3} that for all $\lambda \in \T$,
		\[
		|\Pz(f(\lm))|=|\Pz(a(\lm), \ x_1(\lm), \ x_2(\lm), \ x_3(\lm))| \leq 1, \quad z_1, z_2 \in \D.
		\]
		By maximum principle, it follows that 
		\[
		|\Pz(f(\lm))|=|\Pz(a(\lm), \ x_1(\lm), \ x_2(\lm), \ x_3(\lm))| \leq 1 \quad \text{ for all } \lm \in \D.
		\]
		Again by \eqref{eqn5.3}, $f(\lm)=(a(\lm), \ x_1(\lm), \ x_2(\lm), \ x_3(\lm)) \in \CQ$ for all $\lm \in \overline{\D}$. Thus, $f=(a, x_1, x_2, x_3)$ is a rational map from $\D$ to $\CQ$ that maps $\T$ to $b\Q$. This shows that $f=(a, x_1, x_2, x_3)$ is a rational $\CQ$-inner function. 
	\end{proof} 
	
	We now present a complete description of the rational $\CQ$-inner functions and this is the main result of this section.
	
	\begin{thm}\label{thm10.12}
		Let $f=(a, x_1, x_2, x_3)$ be a rational $\CQ$-inner function and let $deg(x_3)=n$. Let $a \ne 0$ and let $a=a_{in}a_{out}$ be an inner-outer factorization of $a$. Then there exist polynomials $A, E_1, E_2, D$ of degree at most $n$ such that  the following hold.
		\begin{enumerate}
			\item $E_1(\lm)=E_2^{\sim n}(\lm)$ for $\lm \in \overline{\D}$,\\
			
			\item $A$ is outer and $D(\lm) \ne 0$ for all $\lambda  \in \overline{\D}$,\\
			
			\item $|E_i(\lm)| \leq |D(\lm)|$ on $\overline{\D}$, for $i=1,2$,\\
			
			\item $|A(\lambda)|^2=|D(\lambda)|^2-|E_1(\lambda)|^2$ on $\T$,\\	
			
			\item 
			$\displaystyle 
			(a, x_1, x_2, x_3)=\left(a_{in}\frac{A}{D}, \frac{E_1}{D}, \frac{E_2}{D}, \frac{D^{\sim n}}{D} \right)$ on $\overline{\D}$.
		\end{enumerate}
		Conversely, if $A, E_1, E_2$ and $D$ are polynomials of degree at most $n$ that satisfy conditions $(1)-(4)$ and $a, x_1, x_2$ and $x_3$ are defined as in $(5)$, then $f=(a, x_1, x_2, x_3)$ is a rational $\CQ$-inner function with degree of $x_3$ at most $n$.	
	\end{thm}
	
	\begin{proof}
		Let $f=(a, x_1, x_2, x_3)$ be a rational $\CQ$-inner function and let $deg(x_3)=n$. By Lemma \ref{lem10.5}, $x=(x_1, x_2, x_3)$ is a rational $\Ebar$-inner function of degree $n$. It follows from Theorem \ref{thm_Einner} that there exist polynomials $E_1, E_2, D$ of degree at most $n$ such that the following hold: 
		\begin{equation*}
			\begin{split}
				&(1)  \	\ D(\lm) \ne 0 \  \text{on} \  \overline{\D}, \qquad (2) \ E_1=E_2^{\sim n} \ \text{on} \ \overline{\D}, 
				\qquad (3) \  |E_i(\lm)| \leq |D(\lm)| \  \text{for} \ \lm \in \overline{\D}, i=1,2,\\
				&(4) \ \	x_1=\frac{E_1}{D}\ \ \text{on} \ \ \overline{\D},\qquad (5) \ \	x_2=\frac{E_2}{D} \ \ \text{on} \ \ \overline{\D} \  \qquad (6) \ \	x_3=\frac{D^{\sim n}}{D} \ \ \text{on} \ \ \overline{\D}. \\
			\end{split}	
		\end{equation*}
		Since $f$ is a rational $\CQ$-inner function, $(a(\lm), x_1(\lm), x_2(\lm), x_3(\lm)) \in b\Q$ for all $\lm \in \T$. Thus, $|a_{out}(\lm)|^2=1-|x_1(\lm)|^2$ as $|a_{in}(\lm)|=1$ everywhere on $\T$. Then
		\begin{equation}\label{eqn10.3}
			|a_{out}(\lm)|^2=1-\frac{|E_1(\lm)|^2}{|D(\lm)|^2}  \quad \text{and so,} \quad 0 \leq |a_{out}(\lambda)|^2|D(\lm)|^2=|D(\lm)|^2-|E_1(\lm)|^2 \ \text{on} \ \T.
		\end{equation}
		By Theorem \ref{Fejer}, there is a polynomial $A$ of degree at most $n$ such that $A(\lm) \ne 0$ for $\lm \in \D$ and 
		\begin{equation}\label{eqn10.4}
			|A(\lambda)|^2=|D(\lambda)|^2-|E_1(\lambda)|^2 \quad \text{for all} \quad \lm \in \T.
		\end{equation}
		It follows from \eqref{eqn10.3} \& \eqref{eqn10.4} that $|A(\lm)|^2=|a_{out}(\lm)|^2|D(\lm)|^2$ for every $\lm \in \T$. Note that $D(\lambda) \ne 0$ on $\overline{\D}$. Thus, 
		\[
		|a_{out}(\lm)|=\bigg|\frac{A(\lm)}{D(\lm)}\bigg| \quad \text{for all} \ \lambda \in \T.
		\]
		Hence, $A \slash D$ is an outer function such that $|a(\lm)|=|A(\lm)\slash D(\lm)|$ for all $\lm \in \T$. Since outer factors are unique up to scalars of modulus one, there exists $\om \in \T$ such that 
		\[
		a_{out}=\om \frac{A}{D} \quad \text{on} \ \overline{\D}. 
		\]
		Therefore, $a=a_{in}A\slash D$ on $\overline{\D}$, if $A$ is replaced by $\om A$. Conversely, assume that $A, E_1, E_2$ and $D$ are polynomials of degree at most $n$ which satisfy conditions $(1)-(4)$ and $a, x_1, x_2$ and $x_3$ are defined as in $(5)$. It follows from Theorem \ref{thm_Einner} that $x=(x_1, x_2, x_3)$ is a rational $\Ebar$-inner function of degree at most $n$. Since the rational inner functions on $\D$ are precisely the finite Blaschke products, the rest of the conclusion follows from Theorem \ref{thm10.14}.
	\end{proof}

	We now present examples of $\CQ$-inner maps based on the idea of $\overline{\Pe}$-inner functions from \cite{AglerIV}. 	
	
	\begin{eg}
		Let $h_1, h_2$ be non-constant inner functions. Then the map given by
		\begin{equation}\label{eqn10.1}
			h: \D \to M_2(\C), \quad h(\lambda)=\begin{pmatrix}
				h_1(\lambda) &  0\\
				0 & h_2(\lambda) \\
			\end{pmatrix}
		\end{equation}
		is a non-constant analytic map with $\|h(\lambda)\|=\max\{|h_1(\lambda)|, |h_2(\lambda)| \} <1$ for all $\lambda \in \D$. Then the map
		$
		f: \D \to \C^4$ given by $f(\lambda) =\pi(h(\lambda))=(0,\ h_1(\lambda),\ h_2(\lambda),\ h_1(\lambda)h_2(\lambda))$ defines an analytic map. It is clear from the definition of $\QN$ that $f(\lambda)=\pi(h(\lambda)) \in \QN$. We have by \eqref{eqn5.3} that $f(\lambda) \in \CQ$ for all $\lambda \in \D$. Since $h_1$ and $h_2$ are inner functions on $\D$, it follows that 
		\[
		|h_1(\lambda)|=1, \quad |h_2(\lambda)|=1, \quad |h_1(\lambda)h_2(\lambda)|=1 \quad \text{for almost all } \lambda \in \T.
		\] 
		and so, by Theorem \ref{thmbQ}, $f(\lambda) \in b\Q$ for almost every $\lambda \in \T$. Hence, $f$ is an $\CQ$-inner function. \qed 
	\end{eg}
	
	\begin{eg}
		Let $h: \D \to M_2(\C)$ be as in \eqref{eqn10.1}. Define $h_0=Uh$, where  $U$ is given by
		\[
		U=\begin{pmatrix}
			\frac{1}{\sqrt{2}} & \frac{1}{\sqrt{2}} \\
			\frac{i}{\sqrt{2}} & \frac{-i}{\sqrt{2}}
		\end{pmatrix} \quad \text{and so,} \quad h_0(\lm)=\frac{1}{\sqrt{2}}\begin{pmatrix}
			h_1(\lm) & h_2(\lm)\\
			ih_1(\lm) & -ih_2(\lm)	
		\end{pmatrix}
		\]
		for all $\lm \in \D$. Note that $U$ is a unitary matrix and
		$
		\|h_0(\lm)\| \leq \|U\| \ \|h(\lm)\|=\|h(\lm)\|<1
		$ for all $\lm \in \D$. Define $f=\pi \circ h_0$ on $\D$, which is precisely given by
		\begin{equation*}
			\begin{split}
				f(\lm)=\pi(h_0(\lm))
				=\bigg(\frac{ih_1(\lm)}{\sqrt{2}}, \frac{h_1(\lm)}{\sqrt{2}}, \frac{-ih_2(\lm)}{\sqrt{2}}, -ih_1(\lm)h_2(\lm) \bigg) \quad (\lm \in \D).
			\end{split}
		\end{equation*}
		Note that almost for all $\lm \in \T, |h_1(\lm)|=1=|h_2(\lm)|$. Using Theorem \ref{thmbQ}, it is not difficult to see that $f(\lm) \in b\Q$ for almost all $\lm \in \T$. Hence, $f$ is an $\CQ$-inner function. \qed
	\end{eg}
	
	\begin{eg}
		Let $v, h_1, h_2$ be inner functions on $\D$. Consider the functions 
		\[
		h_0(\lm)=\begin{pmatrix}
			h_1(\lm) & 0	\\
			0 & h_2(\lm)
		\end{pmatrix} \quad  \text{and} \quad 	V(\lm)= U_0\begin{pmatrix}
			1 & 0 \\
			0 & v(\lm)
		\end{pmatrix}, \quad \text{where} \quad U_0=\begin{pmatrix}
			\frac{1}{\sqrt{2}} & \frac{1}{\sqrt{2}} \\
			\frac{-1}{\sqrt{2}} & \frac{1}{\sqrt{2}}
		\end{pmatrix}  
		\]
		for  $\lm \in \D$. Note that $U_0$ is a unitary matrix. Define 
		\begin{equation*}
			\begin{split}
				U: \D \to M_2(\C), \, \, U(\lm)=(V^*h_0V)(\lm)=\frac{1}{2}\begin{pmatrix}
					h_1(\lm)+h_2(\lm) & (h_1(\lm)-h_2(\lm))v(\lm)\\
					(h_1(\lm)-h_2(\lm))\overline{v(\lm)} & (h_1(\lm)+h_2(\lm))|v(\lambda)|^2
				\end{pmatrix}.
			\end{split}
		\end{equation*}
		Let $\lambda \in \D$. Then $\|V^*(\lm)\|=\max\{1, |\overline{v(\lm)}|\} \leq 1
		$ and so, $\|U(\lm)\| \leq \|V(\lm)\|^2 \|h_0(\lm)\| <1$. Define
		\[
		f(\lm) =\pi(U(\lm))=\bigg(\frac{1}{2}(h_1-h_2)\overline{v}, \ \frac{1}{2}(h_1+h_2), \ \frac{1}{2}(h_1+h_2)|v|^2, \ \frac{1}{4}((h_1+h_2)^2-(h_1-h_2)^2|v|^2) \bigg)
		\]
		for $\lm \in \D$. Evidently, $f$ is analytic on $\D$ if and only if $v$ is constant. We denote by $f=(a, x_1, x_2, x_3)$ on $\D$. Since $|v(\lm)|=|h_1(\lm)|=|h_2(\lm)|=1$ for almost every $\lm \in \T$, then
		\begin{align*} |x_3(\lm)|&=\frac{1}{4}|(h_1(\lm)+h_2(\lm))^2-(h_1(\lm)-h_2(\lm))^2|=|h_1(\lm)h_2(\lm)|=1,\\
			|a(\lm)|^2+|x_1(\lm)|^2&=\frac{1}{4}|h_1(\lm)-h_2(\lm)|^2+\frac{1}{4}|h_1(\lm)+h_2(\lm)|^2=1
		\end{align*} 
		for almost all $\lm \in \T$. It follows from Theorem \ref{thmbQ} that $f(\lm) \in b\Q$ for almost all $\lm$ in $\T$. Hence, $f$ is an $\CQ$-inner function if and only if $v$ is constant. \qed 
	\end{eg}
	
	\begin{eg}
		Let $m, n \in \mathbb{N}$. We have by part-$(4)$ of Theorem \ref{tetrablock} that  $(0, 0, \lm^n) \in \Ebar$ for all $\lm \in \overline{\D}$. Define 
		\[
		A: \overline{\D} \to M_2(\C), \quad A(\lambda)=\begin{pmatrix}
			0 & -\lm^n\\
			1 & 0\\
		\end{pmatrix} \quad \text{and so,} \quad \|A(\lambda)\| \leq 1.
		\]
		Let $\lambda \in \overline{\D}$. It follows from Proposition 4.7 in \cite{Biswas-Pal} that 	$\CQN=\{\pi(A): A \in M_2(\C), \ \|A\| \leq 1\}$ and so, $\pi(A(\lambda))=(1, 0, 0, \lm^n) \in \CQN$. By \eqref{eqn5.3}, $(1, 0, 0, \lm^n) \in \CQ$ and $\underset{z_1, z_2 \in \D}{\sup}|\Pz(1, 0, 0, \lm^n)| \leq 1$. Then
		\[
		\sup_{z_1, z_2 \in \D}|\Pz(\lm^m, 0, 0, \lm^n)|= |\lm|^m\sup_{z_1, z_2 \in \D}|\Pz(1, 0, 0, \lm^n)| \leq 1.
		\]
		By \eqref{eqn5.3}, $(\lm^m, 0, 0, \lm^n) \in \CQ$ for all $\lm \in \overline{\D}$. Hence, the function $f: \overline{\D} \to \CQ$ defined  by $f(\lm)=(\lm^m, 0, 0, \lm^n)$ is analytic. By Theorem \ref{thmbQ}, $f(\T) \subseteq b\Q$. Hence, $f$ is a rational $\CQ$-inner function. \qed
	\end{eg}

	\section{A Schwarz Lemma for the hexablock}
	\label{Schwarz_lem}

	\noindent For sets $\Omega_1 \subseteq \C^n$ and $\Omega_2 \subseteq \C^m$, we denote by Hol$(\Omega_1, \Omega_2)$ the space of holomorphic functions on $\Omega_1$ taking values in $\Omega_2$. The classical Schwarz Lemma provides a characterization of the two-point interpolation problem in $\D \setminus \{0\}$. In this section, we explore a Schwarz Lemma for $\Q$, that is, we characterize those pairs $\lambda_0 \in \D$ and $(a, x_1, x_2, x_3) \in \Q$ for which there exists a rational $\CQ$-inner function $h$ satisfying $h(0)=0$ and $h(\lambda_0)=(a,x_1,x_2,x_3)$. It follows from Proposition 3.16 in \cite{Biswas-Pal2} that if $h:\D\to\CQ$ is a non-constant holomorphic map such that $h(\lambda)\in\Q$ for some $\lambda\in\D$, then $h(\D)\subseteq\Q$. As an immediate consequence, we have the following lemma.

	\begin{lem} Let $\lambda_0 \in \D \setminus \{0\}$ and let $(a, x_1, x_2, x_3) \in \Q$. Then the following are equivalent:
		\begin{enumerate}
			\item[(i)] there is an analytic map $h: \D \to \overline{\Q}$ such that $h(0)=(0, 0, 0, 0)$ and $h(\lambda_0)=(a, x_1, x_2, x_3)$;
			
			\item[(ii)] there is an analytic map $h: \D \to \Q$ such that $h(0)=(0, 0, 0, 0)$ and $h(\lambda_0)=(a, x_1, x_2, x_3)$.
		\end{enumerate}
	\end{lem}
	
	The following necessary conditions for a Schwarz lemma for $\Q$ follow from the classical Schwarz lemma and the Schwarz lemma for the tetrablock.
	
		\begin{thm}\label{necessary}
		Let $\lm_0 \in \D\setminus\{0\}$ and let $(a, x_1, x_2, x_3) \in \Q$. Then, in the following, $(1)\implies(2)\implies(3)\implies(4)\implies(5):$
		
		\begin{enumerate}
			\item there is a rational $\CQ$-inner function $f$ with $f(0)=(0, 0, 0, 0)$ and $f(\lm_0)=(a, x_1, x_2, x_3)$; 
			
			\item there exists $h \in$ Hol$(\D, \CQ)$ such that $h(0)=(0, 0, 0, 0)$ and $h(\lm_0)=(a, x_1, x_2, x_3)$;
			
			\item 
			$
			\underset{z_1, z_2 \in \D}{\sup} |\psi_{z_1, z_2}(a, x_1, x_2, x_3)| \leq |\lambda_0|
			$
			and
			\[
			\max\bigg\{\frac{|x_1-\overline{x}_2x_3|+|x_1x_2-x_3|}{1-|x_2|^2}, \ \frac{|x_2-\overline{x}_1x_3|+|x_1x_2-x_3|}{1-|x_1|^2}\bigg\} \leq |\lambda_0|;
			\]
			\item $(\lm_0^{-1}a, x_1, x_2, x_3) \in \CQ$	 and
			\[
			\max\bigg\{\frac{|x_1-\overline{x}_2x_3|+|x_1x_2-x_3|}{1-|x_2|^2}, \ \frac{|x_2-\overline{x}_1x_3|+|x_1x_2-x_3|}{1-|x_1|^2}\bigg\} \leq |\lambda_0|;
			\]		
			\item 
			$
			|a| \leq |\lm_0|\sqrt{1-|x_1|^2}
			$ 
			and
			\[
			\max\bigg\{\frac{|x_1-\overline{x}_2x_3|+|x_1x_2-x_3|}{1-|x_2|^2}, \ \frac{|x_2-\overline{x}_1x_3|+|x_1x_2-x_3|}{1-|x_1|^2}\bigg\} \leq |\lambda_0|.
			\]		
		\end{enumerate}
	\end{thm}

\begin{proof}
	The condition $(1) \implies (2)$ follows trivially. Suppose $(2)$ holds and let $h=(h_1, h_2, h_3, h_4)$. Then $h=(h_2, h_3, h_4) \in$ Hol$(\D, \E)$ that maps $0$ to $(0, 0, 0)$ and $\lambda_0$ to $(x_1, x_2, x_3)$. By Schwarz Lemma for $\E$ (see Theorem 1.2 in \cite{Abouhajar}), the second inequality in condition $(3)$ holds. Let $z_1, z_2 \in \D$. It is clear from the definition of $\Q$ that the function 
	\[
	\psi_{z_1, z_2}(b, y_1, y_2, y_3)=\frac{b\sqrt{(1-|z_1|^2)(1-|z_2|^2)}}{1-y_1z_1-y_2z_2+y_3z_1z_2}
	\]
	maps $\Q$ analytically to $\D$. It also maps $(0, 0, 0, 0)$ to $0$. Hence, $\psi_{z_1, z_2}\circ h$ is an analytic self-map of $\D$  that maps $0$ to $0$ and $\lambda_0$ to $\psi_{z_1, z_2}(a, x_1, x_2, x_3)$. It follows from the Schwarz Lemma for $\D$ that
	\[
	|\psi_{z_1, z_2}(a, x_1, x_2, x_3)| \leq |\lambda_0| \quad \text{for all} \quad z_1, z_2 \in \D,
	\]	
	which proves condition $(3)$. The condition $(3) \implies (4)$ follows from \eqref{eqn5.3}. If $(\lm_0^{-1}a, x_1, x_2, x_3) \in \CQ$, then Lemma \ref{lem6.5} gives
	$
	|\lm_0^{-1}a|^2+|x_1|^2 \leq 1$ and so, $|a| \leq |\lm_0|\sqrt{1-|x_1|^2}$. Thus, $(4) \implies (5)$ follows. The proof is now complete.
\end{proof}

\begin{rem}
	In the initial version of our long draft \cite{Biswas-Pal} (see version 1 of \cite{Biswas-Pal} in arXiv), the above Theorem \ref{necessary} (which was stated as Theorem 13.9 in \cite{Biswas-Pal}) was presented with all five conditions being necessary and sufficient for a Schwarz lemma for $\Q$. However, the sufficiency part in those conditions are not true. In the proof of the sufficiency part (see the proof of the sufficiency part of Theorem 13.9 in \cite{Biswas-Pal}), we used Theorem  11.3 of \cite{Alsheri} which provides a set of similar necessary and sufficient conditions for a Schwarz lemma for the pentablock $\Pe$. For the sake of transparency, we state that result from \cite{Alsheri} below.
	
		\begin{thm}[\cite{Alsheri}, Theorem 11.3]\label{thm_AZ_Sch_P}
		Let $\lm_0 \in \D\setminus \{0\}$ and let $(a, s, p) \in \Pbar$. Then the following are equivalent:
		\begin{enumerate}
			\item[(i)] there exists a rational $\Pbar$-inner map $h$ such that $h(0)=(0, 0, 0)$ and $h(\lm_0)=(a, s, p)$; \smallskip 
			
			\item[(ii)] there exists an analytic function $h: \D \to \Pbar$ such that $h(0)=(0, 0, 0)$ and $h(\lm_0)=(a, s, p)$, and  $\displaystyle |a| \leq |\lm_0|\sqrt{1-\frac{|s|^2}{4}}$; \smallskip 
			
			\item[(iii)]
			$\displaystyle 
			\frac{2|s-\overline{s}p|+|s^2-4p|}{4-|s|^2} \leq |\lambda_0|, \  |s|<2$ and  $\displaystyle |a| \leq |\lm_0|\sqrt{1-\frac{|s|^2}{4}}$.
		\end{enumerate}  
		
	\end{thm}  
	
Note that, the sufficiency part in those conditions in the Schwarz lemma for pentablock (as in Theorem \ref{thm_AZ_Sch_P}) is not correct as the following counter example shows. For this reason, we have stated in Theorem \ref{necessary} those five conditions as necessary conditions only.

	\begin{eg}\label{eg_false_Pe}
	Take
	\[
	\lambda_0=\frac12 \quad \text{and} \quad
	(a_0,s_0,p_0)=\left(\frac25,\frac45,\frac4{25}\right).
	\]
	The matrix $A=(a_{ij})_{i, j=1}^2=\begin{pmatrix} 2 \slash 5 & 0\\[2pt] 2 \slash 5 & 2 \slash 5 \end{pmatrix}$ satisfies $I-A^*A > 0$ and thus $\|A\| < 1$. Consequently,  $(a_{21},\operatorname{tr}A,\det A)=(a_0,s_0,p_0)=(2 \slash 5, 4 \slash 5, 4 \slash 25) \in \Pe$. Since $s_0^2=4p_0$, we have
	\begin{align*}
		\frac{2|s_0-\overline{s_0}p_0|+|s_0^2-4p_0|}{4-|s_0|^2}
		&=\frac{2\cdot\frac45\left(1-\frac4{25}\right)}{4-\frac{16}{25}}
		=\frac25<\frac12=|\lambda_0|.
	\end{align*}
	Moreover, $|s_0|<2$ and
	\[
	|a_0|=\frac25
	<\frac{\sqrt{21}}{10}
	=\frac12\sqrt{1-\frac14\left(\frac45\right)^2}=|\lambda_0|\sqrt{1-\frac14|s_0|^2}.
	\]
	Thus condition \textup{(iii)} of Theorem \ref{thm_AZ_Sch_P} holds. The holomorphic map
	\[
	e_{\mathbb P \to \mathbb B_2}:\overline{\Pe}\to \overline{\mathbb B}_2,
	\qquad (a,s,p) \mapsto \left(a,\frac{s}{2}\right)
	\]
	is well-defined since $|a|^2+|s|^2\slash 4 \leq 1$ for every $(a, s, p) \in \overline{\mathbb P}$. The latter inequality follows  directly from Theorem \ref{thm_final_embed}. Let if possible, there exist a holomorphic function $f: \D \to \overline{\mathbb P}$ such that $f(0)=(0, 0, 0)$ and $f(\lambda_0)=(a_0, s_0, p_0)$. Then
	$g=e_{\mathbb P \to \mathbb B_2}\circ f:\D\to\overline{\mathbb B}_2$ is a holomorphic map that satisfies $g(0)=(0, 0)$. It follows from the Schwarz lemma  for the Euclidean ball (see Chapter 8, Section 8.1 in \cite{Rudin}) that $\|g(\lambda)\|_2\le|\lambda|$ for every $\lambda \in \D$. Also, we have 
	\[
	\|g(\lambda_0)\|_2^2
	=|a_0|^2+\frac14|s_0|^2
	=\frac4{25}+\frac4{25}
	=\frac8{25}>\frac14=|\lambda_0|^2,
	\]
	which is a contradiction. \qed
\end{eg}

\end{rem}

	We now provide a solvability criterion for a Schwarz lemma for $\Q$ for a special case. The proof of the next result runs along the same lines as that of Theorem 11.2 in \cite{Alsheri}.
	
	\begin{thm} 
		Let $\lm_0 \in \D\setminus\{0\}$ and let $(a, x_1, x_2, x_3) \in \Q$ be such that $|x_3|=|\lm_0|$. Then the following are equivalent:
		\begin{enumerate}
			\item there is a rational $\CQ$-inner map $f$ such that $f(0)=(0, 0, 0, 0)$ and $f(\lm_0)=(a, x_1, x_2, x_3)$; 
			
			\item there exists $f \in$ Hol$(\D, \CQ)$ such that $f(0)=(0, 0, 0, 0)$ and $f(\lm_0)=(a, x_1, x_2, x_3)$;
			
			\item 
			$
			\underset{z_1, z_2 \in \D}{\sup} |\psi_{z_1, z_2}(a, x_1, x_2, x_3)| \leq |\lambda_0|, \quad
			$
			and
			\[
			\max\bigg\{\frac{|x_1-\overline{x}_2x_3|+|x_1x_2-x_3|}{1-|x_2|^2}, \ \frac{|x_2-\overline{x}_1x_3|+|x_1x_2-x_3|}{1-|x_1|^2}\bigg\} \leq |\lambda_0|.
			\]
			
		\end{enumerate}
	\end{thm}

	\begin{proof}
		First, note that $(x_1, x_2, x_3) \in \E$ and so, $|x_j|<1$ for $1 \leq j \leq 3$. Let $\om$ be a scalar of unit modulus such that $x_3=\om \lm_0$. Note that $(1) \implies (2)$ follows trivially and $(2) \implies (3)$ follows from Theorem \ref{necessary}. We prove $(3) \implies (1)$. Suppose that the inequalities in  $(3)$ hold. So, by given hypothesis we have $|x_3|=|\lm_0|$ and so,
		\[
		\frac{|x_1-\overline{x}_2x_3|+|x_1x_2-x_3|}{1-|x_2|^2} \leq |x_3|.
		\]
		Then 
		\begin{equation*}
			\begin{split}
				(1-|x_2|^2)|x_3|  \geq |x_1-\overline{x}_2x_3|+|x_1x_2-x_3| \geq |x_1|-|\overline{x}_2x_3| +|x_3|-|x_1x_2|
				=(1-|x_2|)(|x_1|+|x_3|),
			\end{split}
		\end{equation*}
		which gives $|x_1| \leq |x_2x_3|$ as $|x_2|<1$. Similarly, using the given inequality 
		\[
		\frac{|x_2-\overline{x}_1x_3|+|x_1x_2-x_3|}{1-|x_1|^2} \leq |x_3|,
		\]	
		one can show that $|x_2| \leq |x_1x_3|$. Therefore, $|x_1| \leq |x_2| \ |x_3| \leq |x_1| \ |x_3|^2$ and so, $|x_1|(1-|x_3|^2) \leq 0$. This is possible only when $x_1=0$ as $|x_3| <1$. Consequently, $x_1=x_2=0$. The function $g: \D \to \Ebar$ given by $g(\lm)=(0, 0, \om \lm)$ is a rational $\Ebar$-inner function. Note that 
		\[
		|a| \leq \sup_{z_1, z_2 \in \D}\bigg|\frac{a\sqrt{(1-|z_1|^2)(1-|z_2|^2)}}{1-x_1z_2-x_2z_2+x_3z_1z_2}\bigg|=	\sup_{z_1, z_2 \in \D} |\psi_{z_1, z_2}(a, x_1, x_2, x_3)| \leq |\lambda_0|.
		\]
		Let $\omega_1 \in \T$ be such that $\omega_1^2=\omega$. For $\lm \in \D$, define $(E_1(\lm), E_2(\lm), D(\lm))=(0, 0, \overline{\om}_1)$. Let
		\[
		f_1(\lm)=\frac{E_1(\lm)}{D(\lm)}, \quad f_2(\lm)=\frac{E_2(\lm)}{D(\lm)} \quad \text{and} \quad f_3(\lm)=\frac{D^{\sim 1}(\lm)}{D(\lm)}.
		\]	
		It is evident that $g=(f_1, f_2, f_3)$. By Theorem \ref{thm10.14}, we can construct a rational $\CQ$-inner function 
		\[
		f=\bigg(cB\frac{A}{D}, \ \frac{E_1}{D},\ \frac{E_2}{D}, \ \frac{D^{\sim 1}}{D}\bigg)
		\]	
		for every finite Blaschke product $B$ and $c \in \T$, where $A$ is a polynomial such that  
		\[
		A(\lm) \ne 0 \ \  \text{for all} \ \lm \in {\D}\quad  \text{and} \quad  |A(e^{it})|^2=|D(e^{it})|^2-|E_1(e^{it})|^2=|D(e^{it})|^2=|\overline{\om}_1|^2=1.
		\]
		We choose $A=1$, the constant polynomial. It is sufficient to construct $f_1: \D \to \C$ such that $f_1=cB$ with $|c|=1$ such that $f_1(0)=0$ and $f_1(\lm_0)=a$.
		
		\medskip 
		
		If $|a|=|\lm_0|$, then we define $f_1(\lm)=m\lm$ for $\lm \in \D$, where $ m \in \T$ such that $a=m\lm_0$. It is easy to see that $f_1(0)=0$ and $f_1(\lm_0)=a$. 
		
		\medskip 
		
		Let $|a|<|\lm_0|$ and let $\eta_0=a \slash \lm_0$. It is clear that $\eta_0 \in \D$. Consider the map $f_1: \D \to \D$ given by$f_1(\lm)=\lm(h_1^{-1} \circ h_2(\lm))$, 
		where $h_1$ and $h_2$ are the Blaschke factors given by 
		\[
		h_1^{-1}(z)=\frac{z+\eta_0}{1+\overline{\eta_0}z} \quad \text{and} \quad h_2(z)=\frac{z-\lm_0}{1-\overline{\lm_0}z} \quad \text{for} \ z \in \D.
		\]
		It follows easily that $f_1(0)=0$ and $f_1(\lm_0)=a$. Putting everything together, we have a rational $\CQ$-inner function $f: \D \to \CQ$ with $f(\lm)=(f_1(\lm), 0, 0, \om \lm)$ defined as follows: 
		\begin{enumerate}
			\item[(a)] if $|a|=|\lm_0|$, then $f_1(\lm)=m\lm$ for $\lm \in \D$, where $|m|=1$ and $a=m\lm_0$;
			\item[(b)] if $|a|<|\lm_0|$, then 
			\[
			f_1(\lm):=\lm\frac{\lm-\lm_0+\eta_0(1-\overline{\lm}_0\lm)}{1-\overline{\lm}_0\lm+\overline{\eta_0}(\lm-\lm_0)} \quad (\lm \in \D, \eta_0=a\lm_0^{-1}).
			\]
		\end{enumerate}
		Moreover, $f$ satisfies $f(0)=(0, 0, 0, 0)$ and $f(\lm_0)=(a, x_1, x_2, x_3)$. The proof is now complete. \end{proof}

	We recall the following interpolation theorems for $\E$ from the literature.
	
	\begin{thm}[\cite{Abouhajar}, Theorem 1.2]\label{thm_Sch_E_1}
		Let $\lm_0 \in \D\setminus \{0\}$ and let $(x_1, x_2, x_3) \in \E$. Then the following are equivalent:
		\begin{enumerate}
			\item there exists $h \in$ Hol$(\D, \Ebar)$ such that $h(0)=(0, 0, 0)$ and $h(\lm_0)=(x_1, x_2, x_3)$;
			\item  
			\[
			\max\bigg\{\frac{|x_1-\overline{x}_2x_3|+|x_1x_2-x_3|}{1-|x_2|^2}, \ \frac{|x_2-\overline{x}_1x_3|+|x_1x_2-x_3|}{1-|x_1|^2}\bigg\} \leq |\lambda_0|.
			\]
		\end{enumerate}
	\end{thm}
	
	For more than two interpolation points, Brown et al. proved the following theorem in \cite{Brown}.
	
	\begin{thm}[\cite{Brown}, Theorem 8.1]\label{thm_Brown}
		Let $\lm_1, \dotsc, \lm_n$ be distinct points in $\D$ and let $(x_{1j}, x_{2j}, x_{3j}) \in \Ebar$ be such that $x_{1j}x_{2j} \ne x_{3j}$ for $j=1, \dotsc, n$. The following are equivalent:
		\begin{enumerate}
			\item there exists $h \in$ Hol$(\D, \Ebar)$ such that $h(\lm_j)=(x_{1j}, x_{2j}, x_{3j})$ for $j=1, \dotsc, n$;
			
			\item there exists a rational $\Ebar$-inner function $h$ such that $h(\lm_j)=(x_{1j}, x_{2j}, x_{3j})$ for $j=1, \dotsc, n$.
		\end{enumerate}
	\end{thm}
	
	\begin{rem}\label{rem_Brown}
		Let $\lm_0 \in \D\setminus \{0\}$ and let $(x_1, x_2, x_3) \in \E$. Suppose there is a function $h\in$  Hol$(\D, \Ebar)$ such that $h(0)=(0, 0, 0)$ and $h(\lm_0)=(x_1, x_2, x_3)$. We wish to construct a rational $\Ebar$-inner function $g$ such that $g(0)=(0, 0, 0)$ and $g(\lm_0)=(x_1, x_2, x_3)$. One cannot directly apply Theorem \ref{thm_Brown} as $(0, 0, 0)$ is a triangular point of $\Ebar$. However, the proof of Theorem \ref{thm_Brown} ensures the existence of such a function. For this reason, we discuss its proof here. As mentioned earlier, an interested reader is referred to Theorem 8.1 in \cite{Brown} for the original proof.
		
		\smallskip 
		
		Let $x_1x_2=x_3$. Without loss of generality, assume that $|x_2| \leq |x_1|$. It follows from the proof of Theorem \ref{thm_Sch_E_1} (which is Theorem 1.2 in \cite{Abouhajar}) that $|x_2| \leq |x_1| \leq |\lm_0|<1$. Choose $g_1, g_2 \in \A(\D)$ such that $g_j(0)=0$ and $g_j(\lm_0)=x_j$ for $j=1,2$. Then $g=(g_1, g_2, g_1g_2)$ is a rational $\Ebar$-inner function such that $g(0)=(0, 0, 0)$ and $g(\lm_0)=(x_1, x_2, x_3)$. 
		
		\smallskip 
		
		Let $x_1x_2 \ne x_3$. As mentioned earlier, one can simply employ the proof of Theorem \ref{thm_Brown} (see Theorem 8.1 in \cite{Brown}) to obtain the desired rational $\Ebar$-inner function. In this case, $h_1(\lm_0)-h_2(\lm_0)h_3(\lm_0)=x_1x_2-x_3 \ne 0$ and so, $h_1-h_2h_3 \ne 0$ on $\D$. It follows from Theorem 7.1 in \cite{Brown} that there is an analytic map \[
		F=\begin{pmatrix}
			h_1 & f_1 \\
			f_2 & h_2
		\end{pmatrix}: \D \to M_2(\C)
		\]
		such that $f_2$ is non-vanishing on $\D, \|F(\lambda)\| \leq 1$ for all $\lm \in \D$ and 
		\[
		1-\overline{\Psi(w, h(\mu))}\Psi(z, h(\lm))=(1-\overline{w}z)\overline{\gamma(\mu, w)}\gamma(\lm, z)+(1-\overline{\mu}\lm)\eta(\mu, w)^*\frac{1-F(\mu)^*F(\lm)}{1-\overline{\mu}\lm}\eta(\lm, z)
		\]
		for all $z, w, \lm, \mu \in \D$, where
		\[
		\Psi(z, y_1, y_2, y_3)  =\frac{y_3 z-y_1}{y_2 z-1}, \quad \gamma(\lm, z)=(1-h_2(\lm)z)^{-1}f_2(\lm) \quad \text{and} \quad \eta(\lm, z)=\begin{pmatrix}
			1 \\
			\gamma(\lm, z)z
		\end{pmatrix}.
		\]
		Let $z_1,z_2, z_3$ be distinct points in $\D$. Also, let $(\lm_1, \lm_2)=(0, \lm_0)$, $(x_{11}, x_{12}, x_{13})=(0, 0, 0)$ and $(x_{21}, x_{22}, x_{23})=(x_1, x_2, x_3)$. Since $F(\lm)$ is a contraction for all $\lm \in \D$ and $f_2 \ne 0$, Proposition 5.1 in \cite{Brown} shows the existence of the following positive semidefinite matrices of order $6$: 
		\[
		N=(N_{il, jk})=(\overline{\gamma(\lm_i, z_l)}\gamma(\lm_j, z_k)) 
		\]
		and 
		\[
		M=(M_{il, jk})=\left(\eta(\lm_i, z_l)^*\frac{1-F(\lm_i)^*F(\lm_j)}{1-\overline{\lm}_i\lm_j}\eta(\lm_j, z_k) \right),
		\]
		where $1 \leq i, j \leq 2$ and $1 \leq k, l \leq 3$. Consequently, there is a Hilbert space $\mathcal{H}$ of dimension at most $6$ and vectors $v_{jk} \in \mathcal{H}$ such that $M_{il, jk}=\langle v_{jk}, v_{il} \rangle_\mathcal{H}$ for $1 \leq i, j \leq 2$ and $1 \leq k, l \leq 3$. It can be easily verified that 
		\[
		\left \langle  \begin{pmatrix}
			\Psi(z_k, x_{1j}, x_{2j}, x_{3j}) \\
			\gamma(\lm_j, z_k)\\
			v_{jk}
		\end{pmatrix}, \begin{pmatrix}
			\Psi(z_l, x_{1i}, x_{2i}, x_{3i}) \\
			\gamma(\lm_i, z_l)\\
			v_{il}
		\end{pmatrix}\right \rangle
		=\left \langle  \begin{pmatrix}
			1 \\
			z_k\gamma(\lm_j, z_k)\\
			\lm_jv_{jk}
		\end{pmatrix}, \begin{pmatrix}
			1 \\
			z_l\gamma(\lm_i, z_l)\\
			\lm_iv_{il}
		\end{pmatrix}\right \rangle
		\]
		for $1 \leq i, j \leq 2$ and $1 \leq k, l \leq 3$, where the inner product is considered over $\C^2 \oplus \mathcal{H}$. It follows from Proposition B.1.22 in \cite{Brown_thesis}, there is a unitary $L$ on $\C^2 \oplus \mathcal{H}$ such that 
		\[
		L\begin{pmatrix}
			1 \\
			z_k\gamma(\lm_j, z_k)\\
			\lm_jv_{jk}
		\end{pmatrix}=\begin{pmatrix}
			\Psi(z_k, x_{1j}, x_{2j}, x_{3j}) \\
			\gamma(\lm_j, z_k)\\
			v_{jk}
		\end{pmatrix} 
		\] 
		for $1 \leq  j \leq 2$ and $1 \leq k \leq 3$. Let us write $\displaystyle L=\begin{pmatrix}
			A & B \\
			C & D \\
		\end{pmatrix}$. Following the proof of Theorem 8.1 in \cite{Brown}, we have that
		\[
		\begin{pmatrix}
			\Psi(z_k, x_{1j}, x_{2j}, x_{3j}) \\
			\gamma(\lm_j, z_k)
		\end{pmatrix}=(A+\lm_jB(I-\lm_j D)^{-1}C)\begin{pmatrix}
			1 \\
			z_k \gamma(\lm_j, z_k)
		\end{pmatrix} 
		\] 
		for $1 \leq j \leq 2$ and $1 \leq k \leq 3$. The function given by
		\[
		E: \D \to M_2(\C), \quad E(\lm)=A+\lm B(I-\lm D)^{-1}C=\begin{pmatrix}
			a(\lm) & b(\lm) \\
			c(\lm) & d(\lm)
		\end{pmatrix}
		\]
		defines a rational $2 \times 2$ inner function such that $E(\lm)$ is unitary for all $\lm \in \T$, as $L$ is unitary and $\mathcal{H}$ has a finite dimension. Therefore, the map $g: \D \to \Ebar$ defined as $g=(a, d, \det(E))$ is a rational $\Ebar$-inner function. The fact that $g(\lm_0)=(x_1, x_2, x_3)$ follows from the proof of Theorem 8.1 in \cite{Brown}. We show that $g(0)=(0, 0, 0)$. Note that 
		\begin{equation}\label{eqn_Brown}
			\begin{pmatrix}
				\Psi(z_k, 0, 0, 0) \\
				\gamma(0, z_k)
			\end{pmatrix}=A\begin{pmatrix}
				1 \\
				z_k \gamma(0, z_k)
			\end{pmatrix} =E(0)\begin{pmatrix}
				1 \\
				z_k \gamma(0, z_k)
			\end{pmatrix} =\begin{pmatrix}
				a(0) & b(0) \\
				c(0) & d(0)
			\end{pmatrix}\begin{pmatrix}
				1 \\
				z_k \gamma(0, z_k)
			\end{pmatrix} 
		\end{equation} 
		for $1 \leq k \leq 3$. Since $\Psi(z_k, 0, 0, 0)=0$, it follows that $
		a(0)+b(0)z_k\gamma(0, z_k)=0$ and $\gamma(0, z_k)=c(0)+d(0)z_k\gamma(0, z_k)$.
		Thus, we have that $0=a(0)+b(0)z_k(1-d(0)z_k)^{-1}c(0)$ and so,
		\[
		\frac{a(0)-\det(E(0)) \ z_k}{1-d(0)z_k}=0.
		\]
		Thus, the M\"obius map
		\[
		\frac{a(0)-\det(E(0)) \ z}{1-d(0)z}
		\]
		has three distinct zeros in $\D$ and hence is identically equal to the zero map. So, it follows that $a(0)=0$ and $\det(E(0))=0$. Therefore, $g(0)=(0, d(0), 0)$. Since $h_2(0)=0$ and $f_2(0) \ne 0$, we have that $\gamma(0, z_k)=(1-h_2(0)z_k)^{-1}f_2(0)=f_2(0)$ and so, $\gamma(0, z_k) \ne 0$ for all $k=1,2,3$. From \eqref{eqn_Brown}, we have 
		\[
		\begin{pmatrix}
			0 \\ 
			f_2(0)
		\end{pmatrix}=E(0)\begin{pmatrix}
			1 \\
			z_k f_2(0)
		\end{pmatrix}=\begin{pmatrix}
			0 & b(0) \\
			c(0) & d(0)
		\end{pmatrix} \begin{pmatrix}
			1 \\
			z_k f_2(0)
		\end{pmatrix}=\begin{pmatrix}
			z_k b(0) f_2(0) \\
			c(0)+z_k f_2(0) d(0)
		\end{pmatrix} 
		\]
		for $k=1, 2, 3$. Since $z_1,z_2$ are distinct and $f_2(0) \ne 0$, it follows that 
		\[
		f_2(0)=c(0)+z_1f_2(0) d(0)=c(0)+z_2f_2(0) d(0) \quad \text{and so,} \quad d(0)=0. 
		\] 
		Thus, $g(0)=(0, 0, 0)$ and the desired rational $\Ebar$-inner function is obtained. 
	\end{rem} 
	
	Combining Theorem \ref{thm_Sch_E_1}, Theorem \ref{thm_Brown} and Remark \ref{rem_Brown}, we have the following result.
	
	\begin{thm}\label{thm_BrownII}
		Let $\lm_0 \in \D\setminus \{0\}$ and let $(x_1, x_2, x_3) \in \E$. Then the following are equivalent:
		\begin{enumerate}
			\item there exists a rational $\Ebar$-inner function $h$ such that $h(0)=(0, 0, 0)$ and $h(\lm_0)=(x_1, x_2, x_3)$;
			
			\item there exists $h \in$ Hol$(\D, \Ebar)$ such that $h(0)=(0, 0, 0)$ and $h(\lm_0)=(x_1, x_2, x_3)$;
			\item  
			$\displaystyle 
			\max\bigg\{\frac{|x_1-\overline{x}_2x_3|+|x_1x_2-x_3|}{1-|x_2|^2}, \ \frac{|x_2-\overline{x}_1x_3|+|x_1x_2-x_3|}{1-|x_1|^2}\bigg\} \leq |\lambda_0|.
			$
		\end{enumerate}
	\end{thm}
	
	Note that Theorem \ref{thm_BrownII} adds one feather to the existing Schwarz lemma for the tetrablock. We now present necessary and sufficient conditions for the solvability of Schwarz type lemma on $\Q$ when the interpolating function is a rational $\CQ$-inner map. To do so, we need the following matricial Nevanlinna-Pick interpolation result.

	\begin{thm}[S. C. Power, \cite{Power}]\label{lem_4100}
		Let $z_1,\ldots,z_n$ be distinct points of $\mathbb D$ and let
		$W_1,\ldots,W_n\in M_d(\mathbb C)$. Then the following
		conditions are equivalent:
		\begin{enumerate}
			\item There exists a rational $d\times d$ matrix-inner
			function $U$ such that $U(z_i)=W_i$ for $1\leq i\leq n$.
			\item The matrix
			$\displaystyle 
			\left(
			\frac{I-W_iW_j^*}{1-z_i\overline{z_j}}
			\right)_{i,j=1}^n
			$
			is positive semi-definite.
		\end{enumerate}
	\end{thm}
	
	We are now in a position to present the following Schwarz type lemma for $\Q$ where the interpolating function is a rational $\CQ$-inner map.
	
	\begin{thm}\label{thm_hexa_schwarz}
		Let $\lambda_0\in\mathbb D\setminus\{0\}$ and
		$(a,x_1,x_2,x_3)\in\mathbb H$. Then the following
		are equivalent:
		
		\begin{enumerate}
			\item there exists a rational $\CQ$-inner
			map $f$ such that $f(0)=(0,0,0,0)$ and $f(\lambda_0)=(a,x_1,x_2,x_3)$;

			\item there exist $w_{12}, w_{21}, z_{12}\in\mathbb C$ satisfying $w_{12}w_{21}=x_1x_2-x_3$ and $|a| \leq |\lm_0w_{12}|$ such that 
			\begin{equation}\label{eqn_hexa_pick}
				\begin{pmatrix}
					I-W_0W_0^*&I-W_0W_1^*\\[2mm]
					I-W_1W_0^*&
					\dfrac{I-W_1W_1^*}{1-|\lambda_0|^2}
				\end{pmatrix}
				\geq0,
				\quad 	\text{where} \quad W_0=
				\begin{pmatrix}
					0&z_{12}\\
					0&0
				\end{pmatrix} \ \ \text{and} \ \  W_1=
				\begin{pmatrix}
					x_1& w_{12}\\
					w_{21}&x_2
				\end{pmatrix}.
			\end{equation}
		\end{enumerate}
	\end{thm}
	
	\begin{proof}
		Let us assume that there exists a rational $\CQ$-inner
		map $f=(f_0, f_1, f_2, f_3)$ such that $f(0)=(0,0,0,0)$ and $f(\lambda_0)=(a,x_1,x_2,x_3)$. By Lemma \ref{lem10.5}, the map $h=(f_1,f_2,f_3)$ is rational $ \Ebar$-inner. In particular, the map $f_1: \D \to \C$ satisfies $|f_1(\lm)|\leq 1$ for every $\lm \in \D$. We have by Theorem \ref{thm10.12} that $f_1$ is a rational holomorphic map on $\DC$. Suppose $f_1$ is inner, that is, $|f_1(\zeta)|=1$ almost everywhere in $\T$. Then $|f_1|=1$ everywhere on $\T$. Since $h$ is rational $\Ebar$-inner, it follows that $|f_2|=1$ on $\T$ and $f_3(\zeta)=f_1(\zeta)f_2(\zeta)$ for every $\zeta \in \T$. As $f$ is rational $\CQ$-inner map, we have that $|f_0(\zeta)|^2+|f_1(\zeta)|^2=1$ for $\zeta \in \T$ and thus $|f_0|=0$ on $\T$. By identity theorem, $f_3=f_1f_2$ and $f_0=0$ on $\DC$. In particular, $x_1x_2=x_3$ and $a=0$ since $f(0)=(0, 0, 0, 0)$ and $f(\lm_0)=(a, x_1, x_2, x_3)$. Consider the 
		function
		\[
		V: \DC \to M_2(\C), \quad \lm \mapsto \begin{pmatrix}f_1(\lm)&0\\0&f_2(\lm)\end{pmatrix}.
		\]
		Evidently, $V$ is a matrix-valued rational holomorphic map on $\DC$ such that $\|V(\lm)\|\leq 1$ for $\lm \in \DC$. Moreover, $\|V(\zeta)\|=1$ for $\zeta \in \T$.  Note that  $V(0)=0$ and 	$V(\lambda_0)=
		\begin{pmatrix}x_1&0\\0&x_2\end{pmatrix}$. Set $w_{12}=w_{21}=z_{12}=0$. Then $w_{12}w_{21}=x_1x_2-x_3$ and $|a|=|w_{12}\lm_0|=0$. Also, \eqref{eqn_hexa_pick} follows from Theorem \ref{lem_4100}. 
		
		\smallskip Next, consider the case when $f_1$ is not inner. It follows from Theorem \ref{thm10.12} that there exist polynomials $E_1$ and $D$ such that $D(\lm) \neq 0$ for every $\lm \in  \DC$ and $f_1=E_1\slash D$. Note that $|D(\zeta)|^2-|E_1(\zeta)|^2\geq0$ for $\zeta\in\mathbb T$. By the Fej\'er-Riesz theorem, there exists a polynomial $B$ having no zeros in $\mathbb D$ such that $|D(\zeta)|^2-|E_1(\zeta)|^2=|B(\zeta)|^2$ for $\zeta\in\mathbb T$. Consider the rational function $q=B/D$ on $\DC$. Clearly, $q$ has no zeros in $\D$ and for every $\zeta \in \T$, we have
		\[
		|q(\zeta)|^2=|B(\zeta)\slash D(\zeta)|^2=1-|E_1(\zeta)\slash D(\zeta)|^2=1-|f_1(\zeta)|^2.
		\]
		Define
		\[
		U: \D \to M_2(\C), \quad \lm \mapsto 
		\begin{pmatrix}
			f_1(\lm) & q(\lm)\\[1mm]
			\dfrac{f_1(\lm)f_2(\lm)-f_3(\lm)}{q(\lm)}& f_2(\lm)
		\end{pmatrix}.
		\]
		Evidently, $U$ is a rational analytic matrix-valued function on
		$\mathbb D$. We claim that $U$ is contractive on $\mathbb D$. Put
		$g= (f_1f_2-f_3)\slash q$. Since $(f_1,f_2,f_3)$ is rational $\Ebar$-inner, it follows that
		\[
		f_1=\overline{f_2}f_3,\qquad
		|f_3|=1,\qquad |f_1|=|f_2|
		\quad\text{a.e. on }\mathbb T.
		\]
		As these functions are rational and holomorphic on a
		neighbourhood of $\overline{\mathbb D}$, the identities
		hold everywhere on $\mathbb T$. Consequently, whenever
		$q(\zeta)\neq0$,
		\[
		g(\zeta)
		=\frac{f_3(\zeta)(|f_2(\zeta)|^2-1)}{q(\zeta)}
		=-f_3(\zeta)\overline{q(\zeta)}.
		\]
		In particular, $|g(\zeta)|=|q(\zeta)|\leq1$ at every
		such point. Now, we show that $g$ has no poles on $\mathbb T$.
		Indeed, if $\zeta_0\in\mathbb T$ were a pole of order
		$m\geq1$, then $|g(\zeta)|\to\infty$ as $\zeta\to\zeta_0$ along
		$\mathbb T$, contradicting $|g(\zeta)|\leq1$ on
		a punctured arc around $\zeta_0$, since $q$ has only
		finitely many zeros. Thus $g$, and consequently $U$,
		extends holomorphically to a neighbourhood of
		$\overline{\mathbb D}$. Moreover, for every $\zeta\in\mathbb T$,
		\[
		U(\zeta)=
		\begin{pmatrix}
			f_1(\zeta)&q(\zeta)\\
			-f_3(\zeta)\overline{q(\zeta)}&f_2(\zeta)
		\end{pmatrix}.
		\]
		Since $|f_1|^2+|q|^2=1$,
		$|f_3|=1$ and
		$f_1=\overline{f_2}f_3$ on $\mathbb T$, we have that
		$U(\zeta)U(\zeta)^*=I$ for every $\zeta\in\mathbb T$.
		Thus $U$ is unitary on $\mathbb T$. By maximum modulus principle, $\|U(\lambda)\|\leq1$ for every $\lambda\in\mathbb D$. Thus, $U$ is a rational matrix-valued inner function. Set $w_{12}=q(\lm_0), w_{21}=(x_1x_2-x_3)\slash q(\lm_0)$ and $z_{12}=q(0)$. Then 
		$U(0)=
		\begin{pmatrix}
			0& z_{12}\\
			0&0
		\end{pmatrix}, U(\lm_0)=\begin{pmatrix}
			x_1 & w_{12}\\
			w_{21} & x_2
		\end{pmatrix}$ and so, \eqref{eqn_hexa_pick} follows from Theorem \ref{lem_4100}. It remains to show that $|a| \leq |\lm_0w_{12}|$, i.e., $|a| \leq |\lm_0q(\lm_0)|$. To do so, we prove that $f_0/q$ is rational inner. Since $q$ has no
		zeros in $\mathbb D$, the rational function $f_0/q$ is
		holomorphic on $\mathbb D$. Also, we have
		\[
		|f_0(\zeta)|^2=1-|f_1(\zeta)|^2
		=|q(\zeta)|^2,\qquad \zeta\in\mathbb T.
		\]
		As $q$ is a nonzero rational function, it has only finitely
		many zeros on $\mathbb T$. Consequently,
		$|f_0/q|=1$ almost everywhere on $\mathbb T$.
		If $f_0/q$ admits a pole at some $\zeta_0\in\mathbb T$, then $|f_0(\zeta_0)\slash q(\zeta_0)| \to \infty$ as $\zeta\to\zeta_0$
		along $\mathbb T$, which is a contradiction as $|f_0\slash q|=1$ a.e. on $\T$. Thus $f_0/q$
		extends holomorphically to a neighbourhood of
		$\overline{\mathbb D}$. By maximum modulus principle,
		$|f_0/q|\leq1$ on $\mathbb D$. Since $(f_0\slash q)(0)=0$, it follows from Schwarz's lemma on $\D$  that
		$|f_0(\lambda_0)\slash q(\lambda_0)|\leq|\lambda_0|$ and thus $|a|\leq |w_{12}\lm_0|$.
		
		\medskip
		
		Conversely, suppose $(2)$ holds. By Theorem \ref{lem_4100},
		there exists a rational matrix-valued inner map
		\[
		U=
		\begin{pmatrix}
			u_{11}&u_{12}\\
			u_{21}&u_{22}
		\end{pmatrix}
		\]
		such that $U(0)=W_0$ and $U(\lambda_0)=W_1$. Define $h=(u_{11},u_{22},\det U)$. Since $U$ is rational matrix-inner, $h$ is a
		rational $\Ebar$-inner function. Suppose $h$ has degree $n$. Also, $h(0)=(0,0,0)$ and $h(\lambda_0)
		=(x_1,x_2,x_1x_2-w_{12}w_{21})
		=(x_1,x_2,x_3)$. Evidently, the map $u_{11}: \D \to \C$ is a rational holomorphic map that satisfies $|u_{11}(\lm)| \leq 1$ for $\lm \in \D$. Assume that $u_{11}$ is not inner. It follows from Theorem \ref{thm_Einner} that there exist polynomials $E_1, E_2, D$ of degree at most $n$ such that the following hold: 
		\begin{equation*}
			\begin{split}
				&(1)  \	\ D(\lm) \ne 0 \  \text{on} \  \overline{\D}, \qquad (2) \ E_1=E_2^{\sim n} \ \text{on} \ \overline{\D}, 
				\qquad (3) \  |E_i(\lm)| \leq |D(\lm)| \  \text{for} \ \lm \in \overline{\D}, i=1,2,\\
				&(4) \ \	u_{11}=\frac{E_1}{D}\ \ \text{on} \ \ \overline{\D},\qquad (5) \ \	u_{22}=\frac{E_2}{D} \ \ \text{on} \ \ \overline{\D} \  \qquad (6) \ \	\det U=\frac{D^{\sim n}}{D} \ \ \text{on} \ \ \overline{\D}. \\
			\end{split}	
		\end{equation*}
	 By Fej\'er-Riesz theorem, there exists an outer polynomial $A$ such that $|A(\zeta)|^2=|D(\zeta)|^2-|E_1(\zeta)|^2$ for $\zeta \in \T$. Choose $q_{11}=A\slash D$. Then $|q_{11}|^2=1-|u_{11}|^2$
		 a.e. on $\mathbb T$. Since $U$ is unitary matrix-valued a.e. on $\T$, we have that
		$|u_{12}|^2=1-|u_{11}|^2=|q_{11}|^2$ a.e. on $\mathbb T$. Since $q_{11}$ has no zeros in $\D$ and $|u_{12}|=|q_{11}|$ a.e. on $\T$, the rational function $u_{12}/q_{11}$ is holomorphic on $\D$ that has only removable singularities on $\T$. By maximum modulus principle,	$u_{12}/q_{11}$ is an inner function and so, $|u_{12}(\lambda_0)|
		\leq|q_{11}(\lambda_0)|$. Since $u_{12}(\lambda_0)=w_{12}$, we have that $|a|\leq |\lm_0w_{12}|
		\leq |\lm_0q_{11}(\lambda_0)|$. Put $w=a \slash q_{11}(\lambda_0)$. Then $|w|\leq|\lambda_0|$. By the scalar
		two-point Nevanlinna-Pick theorem, there is
		a finite Blaschke product $B$ such that $B(0)=0$ and  $B(\lambda_0)=w$. Define
		\[
		f=(Bq_{11}, h_1, h_2, h_3)=(Bq_{11}, u_{11}, u_{22}, \det U)=\bigg(B\frac{A}{D}, \ \frac{E_1}{D},\ \frac{E_2}{D}, \ \frac{D^{\sim n}}{D}\bigg).
		\]
		It follows from Theorem \ref{thm10.14} that $f$ is rational
		$\CQ$-inner. Also, $f(0)=(0,0,0,0)$ and
		\[
		f(\lambda_0)
		=(B(\lambda_0)q_{11}(\lambda_0),
		x_1,x_2,x_3)
		=(a,x_1,x_2,x_3).
		\]
		It remains to consider the case when
		$u_{11}$ is inner. Then $u_{12}=0$ a.e. on $\T$ as $U$ is unitary matrix-valued a.e. on $\T$. Since $u_{12}$ is rational and belongs to $H^\infty(\D)$, it follows from maximum modulus principle that $u_{12}=0$ on $\D$. In particular, $w_{12}=u_{12}(\lm_0)=0$ and so, $a=0$ since $|a| \leq |\lm_0w_{12}|$.  In this case, we consider a rational holomorphic map on $\D$ given by
		\[
		f=(0,h_1,h_2,h_3)=(0, u_{11}, u_{22}, \det U).
		\]
		Clearly, $f(\D) \subseteq \{0\} \times \Ebar \subseteq \CQ$. Since $|u_{11}|=|u_{22}|=1$ a.e. on $\T$, we have that $f(\zeta) \in b\Q$ a.e. on $\T$. Hence, $f$ is a rational $\CQ$-inner function. The proof is now complete. 
	\end{proof}

	\section{Some geometry and function theory of the pentablock via hexablock}\label{sec_appl}
	
	\noindent	At every stage of this paper, we explore and find interaction of the hexablock with the symmetrized bidisc, the tetrablock and the pentablock. In this section, we provide alternative proofs to several results on the geometry and function theory of pentablock from \cite{AglerIV, Alsheri, Jindal, KosinskiII} using the theory of hexablock. It is already proved in Theorem \ref{thmP_in_H} that $(a, s, p) \in \Pe$ if and only if $(a, s\slash 2, s\slash 2, p) \in \Q$. Consequently, many geometric and function theoretic properties of $\Pe$ can be deduced from the corresponding properties of $\Q$, and this constitutes the main content of this section. Throughout the section, we shall use the following notations: 
	\[
	e_0(a, s, p)=(a, s\slash 2, s \slash 2, p) \quad \& \quad e_*(a, x_1, x_2, x_3)=(a, x_1+x_2, x_3) \quad (a, x_1, x_2, x_3, s, p \in \C).
	\]
	We mention that the restriction of $e_0$ to $\Pbar$ appears in Section \ref{connect}. We shall also follow the notation $e_\Pe=e_0|_{\Pbar}$. Let us begin with the following result showing interplay between $\Pe$ and $\Q$.
	
	\begin{thm}\label{thm_HP_final} 
		The following assertions hold:
		\begin{equation*}
			\begin{split} 
				\Pe&=\{(a, s, p) : (a, s\slash 2, s \slash 2, p) \in \Q\}, \\ 
				\Pbar&=\{(a, s, p) : (a, s\slash 2, s \slash 2, p) \in \CQ\}, \\
				\partial \Pe&=\{(a, s, p) : (a, s\slash 2, s \slash 2, p) \in \partial \Q\}.
			\end{split}
		\end{equation*} 
	\end{thm}
	
	\begin{proof}
		Note that $(1)$ and $(2)$ follow from Theorem \ref{thmP_in_H} and so, $(3)$ holds. 
	\end{proof} 
	
	We recall some geometric notions for subsets in $\C^n$ that will be used in the sequel.  For a compact set $K \subset \C^n$, the \textit{polynomial convex hull} $\widehat{K}$ of $K$ is defined as
	\[
	\widehat{K}=\left\{z \in \C^n \ : \ |p(z)| \leq \|p\|_{\infty, K}  \ \ \text{for all} \ p \in \C[z_1, \dotsc, z_n] \right\},
	\]
	where $\|p\|_{\infty, K}=\sup\{|p(z)|: z \in K\}$. 	The set $K$ is said to be \textit{polynomially convex} if $\widehat{K}=K$ or equivalently, if for every $z \notin K$, there exists $p \in \C[z_1, \dotsc, z_n]$ such that $\|p\|_{\infty, K} \leq 1<|p(z)|$.  A domain $\Omega$ in $\C^n$ is said to be \textit{polynomially convex} if for each compact subset $K$ of $\Omega$, the polynomial hull $\widehat{K}$ of $K$ is contained in $\Omega$. The authors of \cite{Biswas-Pal} proved that $\Q$ is a polynomially convex domain and $\CQ$ is a polynomially convex compact set. 
	
	
	\begin{thm}
		The set $\Pbar$ is polynomially convex. Moreover, the domain $\Pe$ is polynomially convex.
	\end{thm}
	
	\begin{proof}
		This result appears as Theorem 6.3 in  \cite{AglerIV}. Let $(a, s, p) \notin \Pbar$. It follows from Theorem \ref{thm_HP_final} that $(a, s\slash 2, s\slash 2, p) \notin \CQ$. Since $\CQ$ is polynomially convex, there exists a polynomial $g$ in four variables such that $\|g\|_{\infty, \CQ} \leq 1$ and $|g(a, s\slash 2, s\slash 2, p)|>1$. Then $h(z_1, z_2, z_3)=g(z_1, z_2\slash 2, z_2\slash 2, z_3)$ is a polynomial with $|h(a, s, p)|>1$. Let $(\tilde{a}, \tilde{s}, \tilde{p}) \in \Pbar$. Again by Theorem \ref{thm_HP_final}, $(\tilde{a}, \tilde{s}\slash 2, \tilde{s}\slash 2, \tilde{p}) \in \CQ$. Thus, $|h(\tilde{a}, \tilde{s}, \tilde{p})|=|g(\tilde{a}, \tilde{s}\slash 2, \tilde{s}\slash 2, \tilde{p})| \leq 1$ and so, $\|h\|_{\infty, \Pbar} \leq 1$. Hence, $\Pbar$ is polynomially convex.
		
		\smallskip
		
		Next, we prove that the domain $\Pe$ is polynomially convex. Let $K \subset \Pe$ be compact. We need to show that the polynomial convex hull $\widehat{K}$ is in $\Pe$. Consider $e_0(K)=\{(a, s\slash 2, s\slash 2, p): (a, s, p) \in K\}$. Evidently, $e_0(K)$ is compact and by Theorem \ref{thm_HP_final}, $e_0(K) \subset \Q$. Since $\Q$ is a polynomially convex domain, it follows that $\widehat{e_0(K)} \subset \Q$.  Let $(a, s, p) \in \widehat{K}$ and let $h \in \C[z_1, z_2, z_3, z_4]$. Then $h \circ e_0$ is a polynomial in three variables and so, 
		\begin{equation*}
			\begin{split} 
				|h(a,s\slash 2, s\slash 2, p)|=|h(e_0(a, s, p))|  \leq \|h\circ e_0\|_{\infty, K}
				=\|h\|_{\infty, e_0(K)}.
			\end{split}
		\end{equation*} 
		Thus, $(a, s\slash 2, s\slash 2, p) \in \widehat{e_0(K)}$, which itself is contained in $\Q$. We have by Theorem \ref{thm_HP_final} that $(a, s, p) \in \Pe$ and thus, $\widehat{K} \subset \Pe$. Therefore, the domain $\Pe$ is a polynomially convex. 	\end{proof} 
	
	A domain $\Omega$ in $\C^n$ is said to be \textit{$(m_1, \dotsc, m_n)$-quasi-balanced} if there exist relatively prime natural numbers $m_1, \dotsc, m_n$ such that $(\lambda^{m_1}z_1, \dotsc, \lambda^{m_n}z_n) \in \Omega$ for every $\lambda \in \overline{\D}$ and $(z_1, \dotsc, z_n) \in \Omega$.  We say that $\Omega$ is \textit{linearly convex} if for any $z \notin \Omega$, there is a complex hyperplane passing through $z$ that does not intersect $\Omega$. It was proved in \cite{Biswas-Pal} that the domain $\Q$ is a $(1, 1, 1, 2)$-quasi-balanced linearly convex domain. As an application of this, we prove analogous results for the pentablock.
	
	\begin{prop}\label{prop_P_lc_qb}
		$\Pe$ is linearly convex and $(1, 1, 2)$-quasi-balanced.
	\end{prop}
	
	\begin{proof}
		A reader is referred to \cite{AglerIV} and \cite{KosinskiII} for the original proofs. Let $(a, s, p) \notin \Pe$. It follows from Theorem \ref{thm_HP_final} that $(a, s\slash 2, s\slash 2, p) \notin \Q$. Since $\Q$ is a linearly convex domain, there exists a complex hyperplane 
		\[
		H_0=\left\{(z_1, z_2, z_3, z_4) \in \C^4 : a_1z_1+a_2z_2+a_3z_3+a_4z_4=b \right\} 
		\]
		passing through $(a, s\slash 2, s\slash 2, p)$ such that $H_0 \cap \Q =\emptyset$. Consider the hyperplane in $\C^3$ given by 
		\[
		P_0=\left\{(w_1, w_2, w_3) \in \C^3: a_1w_1+\left(\frac{a_2+a_3}{2}\right)w_2+a_4w_3=b\right\}.
		\]
		Since $(a, s\slash 2, s\slash 2, p) \in H_0$, we have that $(a, s,p) \in P_0$. If $(\tilde{a}, \tilde{s}, \tilde{p}) \in P_0 \cap \Pe$, then it follows from Theorem \ref{thm_HP_final} that $(\tilde{a}, \tilde{s}\slash 2, \tilde{s} \slash 2, \tilde{p}) \in H_0 \cap \Q$, a contradiction. Thus, $P_0$ is a complex hyperplane passing through $(a, s, p)$ without intersecting $\Pe$. Therefore, $\Pe$ is linearly convex. Since $\Q$ is $(1, 1, 1, 2)$-quasi-balanced, it follows from Theorem \ref{thm_HP_final} that $\Pe$ is $(1, 1, 2)$-balanced.
	\end{proof} 
	
	\subsection*{A subgroup of $\A(\Pe)$}  For $\omega \in \T$ and $v=-\xi B_z \in Aut(\D)$, consider the map given by
	\begin{equation}\label{eqn_fwv}
		f_{\omega, v}(a, s, p)=\frac{\xi}{1-\overline{z}s+\overline{z}^2p}\bigg( \omega  (1-|z|^2)a, -2z+(1+|z|^2)s-2\overline{z}p, \xi(z^2-z s+p)  \bigg). 
	\end{equation}
	The authors of \cite{AglerIV} proved that each $f_{\Omega, v}$ defines an automorphism of $\Pe$. Furthermore, the set $G(\Pe)=\{f_{\om, v} : \om \in \T, v \in \A(\D)\}$ forms a subgroup of $\A(\Pe)$. We prove the same result as an application of the subgroup of automorphisms of $\A(\Q)$ presented in \cite{Biswas-Pal2}. First, we recall from \cite{Young} that $\text{Aut}(\mathbb E)$ is completely determined by a pair $(v, \chi)$ of functions in $\text{Aut}(\D)$. Consider
	\[
	v(\lm)=-\xi_1B_{z_1}(\lm)=-\xi_1\left(\frac{\lm-z_1}{\overline{z}_1\lm-1}\right) \quad \text{and} \quad  \chi(\lm)=-\xi_2B_{-\overline{z}_2}(\lm)=\xi_2 \left(\frac{\lm+\overline{z}_2}{z_2\lm+1}\right),
	\]
	where $\xi_1, \xi_2 \in \T$ and $z_1, z_2 \in \mathbb{D}$. The map $\tau_{v,\chi}(x_1,x_2,x_3)$ given by
	{\small 
		\begin{equation}\label{lem5.1}
			\left(\xi_1\frac{(x_1-z_1)+\xi_2\overline{z}_2(z_1x_2-x_3)}{(1-\overline{z}_1x_1)-\xi_2\overline{z}_2(x_2-\overline{z}_1x_3)}, \ \
			\frac{z_2(\overline{z}_1x_1-1)+\xi_2(x_2-\overline{z}_1x_3)}{(1-\overline{z}_1x_1)-\xi_2\overline{z}_2(x_2-\overline{z}_1x_3)}, \ \
			\xi_1\frac{z_2(z_1-x_1)-\xi_2(z_1x_2-x_3)}{(1-\overline{z}_1x_1)-\xi_2\overline{z}_2(x_2-\overline{z}_1x_3)}\right)  \smallskip 
		\end{equation}
	}
	\par \noindent	defines an automorphism of $\mathbb{E}$. Indeed, the only automorphisms of $\mathbb{E}$ are $(x_1, x_2, x_3) \mapsto \tau_{v,\chi}(x_1,x_2,x_3)$ and $(x_1, x_2, x_3) \mapsto \tau_{v,\chi, F}(x_1,x_2,x_3)=\tau_{v,\chi}(x_2, x_1, x_3)$. It was proved in \cite{Biswas-Pal2} that every automorphism of $\mathbb{E}$ induces an automorphism of $\mathbb{H}$. In particular, the maps 
	\begin{equation*}
		T_{v, \chi, \omega}(a, x_1, x_2, x_3)=\bigg(\frac{a\w\xi_2\sqrt{(1-|z_1|^2)(1-|z_2|^2)}}{1-x_1\overline{z}_1-x_2\overline{z}_2\xi_2+x_3\overline{z}_1\overline{z}_2\xi_2}, \tau_{v, \chi}(x_1, x_2, x_3)\bigg) \quad (|\omega|=1)
	\end{equation*}
	and $T_{v, \chi, F, \omega}(a, x_1, x_2, x_3)=T_{v, \chi, \omega}(a, x_2, x_1, x_3)$ are automorphisms of $\mathbb{H}$. Moreover, it is proved that 
	\[
	G(\mathbb{H})=\left\{T_{v, \chi, \omega}, \ T_{v, \chi, F, \omega}: v, \chi \in \text{Aut}(\mathbb{D}) \ \text{and} \ \omega \in \mathbb{C} \ \text{with} \ |\omega|=1 \right\}
	\] 
	forms a subgroup of $\text{Aut}(\mathbb{H})$. Capitalizing on this, we prove that $G(\Pe)$ is a subgroup of $\A(\Pe)$. We mention here that the equality $G(\Pe)=\A(\Pe)$ was proved in \cite{KosinskiII, GuicongII}, thereby giving a complete description of $\A(\Pe)$.
	
	\begin{thm}\label{thm_AutPsub} 
		Let $\om \in \T$ and $v \in \A(\D)$. With the above notations and \eqref{eqn_fwv}, we have that the set
		\[
		G(\Pe)=\{f_{\om, v} : v \in \A(\D), \om \in \T\}
		\]
		forms a subgroup of $\A(\Pe)$. Also, $f_{\om, v}=e_* \circ T_{v, \ v_*, \ \omega} \circ e_0, \ f_{\om, v}^{-1}=f_{\overline{\om}, v^{-1}}$ and $f_{\om, v}\circ f_{\theta, \chi}=f_{\om \theta, v \circ \chi}$.
	\end{thm}
	
	\begin{proof}
		This theorem appears as Theorem 7.1 in \cite{AglerIV}. For the sake of brevity, we recall the involution $* : \A(\D) \to \A(\D)$ defined by $v=\w B_\alpha \mapsto v_*=\w B_{\overline{\w}\overline{\alpha}}$. Let $v=-\xi B_{z}$ and $\om \in \T$. Then $v_*=-\xi B_{-\overline{\xi z}}$. It follows from Theorem \ref{thm_HP_final} that $(a, s\slash 2, s\slash 2, p) \in \CQ$ for every $(a, s, p) \in \Pbar$. Let $g=e_*\circ T_{v, \ v_*, \ \omega}\circ e_0$. Note that
		\[
		T_{v, v_*, \omega}(e_0(a, s, p))= T_{v, v_*, \omega}(a, s\slash 2, s\slash 2, p)=\bigg(\w\xi\frac{a(1-|z|^2)}{1-s\overline{z}+p\overline{z}^2}, \ \tau_{v, v_*}(s\slash 2, s\slash 2, p)\bigg),
		\]
		where 	
		\[
		\tau_{v, v_*}(s\slash 2, s\slash 2, p)=\left(\tau_1(s\slash 2, s\slash 2, p), \ \tau_2(s\slash 2, s\slash 2, p), \ \tau_3(s\slash 2, s\slash 2, p)\right)
		\]
		is as in \eqref{lem5.1}. Some routine calculations give
		\[		
		\tau_j\left(\frac{s}{2}, \frac{s}{2}, p\right)=\xi\left[\frac{(s\slash 2)(1+|z|^2)-p\overline{z}-z}{1-\overline{z}s+\overline{z}^2p}\right] \ (j=1,2), \ \
		\tau_3\left(\frac{s}{2}, \frac{s}{2}, p\right)=\xi^2\left[\frac{z^2-sz+p}{1-\overline{z}s+\overline{z}^2p}\right] .
		\]
		Then 
		\begin{equation*}
			\begin{split}
				g(a, s, p)
				&=\bigg(\w\xi\frac{a(1-|z|^2)}{1-s\overline{z}+p\overline{z}^2}, \ \tau_1(s\slash 2, s\slash 2, p)+\tau_2(s\slash 2, s\slash 2, p), \ \tau_3(s\slash 2, s\slash 2, p)\bigg)\\
				&=\bigg(\w\xi\frac{a(1-|z|^2)}{1-s\overline{z}+p\overline{z}^2}, \ \xi\left[\frac{s(1+|z|^2)-2p\overline{z}-2z}{1-\overline{z}s+\overline{z}^2p}\right], \ \xi^2\left[\frac{z^2-sz+p}{1-\overline{z}s+\overline{z}^2p}\right]\bigg)\\
				&=\frac{\xi}{1-\overline{z}s+\overline{z}^2p}\bigg( \omega  (1-|z|^2)a, \ -2z+(1+|z|^2)s-2\overline{z}p, \ \xi(z^2-z s+p)  \bigg)\\
				&=f_{\om, v}(a, s, p),\\
			\end{split}
		\end{equation*}
		where the last equality follows from the definition of $f_{\om, v}$ as in \eqref{eqn_fwv}.  Thus, we have proved that $f_{\om, v}=e_* \circ T_{v, \ v_*, \ \omega} \circ e_0$. Consequently, it follows from Theorem \ref{thm_HP_final} that $f_{\om, v}(\Pbar) \subseteq (\Pbar)$ and $f_{\om, v}(\Pe) \subseteq (\Pe)$ for all $\om \in \T$ and $v \in \A(\D)$. Let $v, \chi \in \A(\D)$ and let $\om,  \theta \in \T$. For $(a, s, p)$, let $T_{\chi, \ \chi_*, \ \theta} \circ e_0(a, s, p)=(\tilde{a}, \tilde{x}_1, \tilde{x}_2, \tilde{x}_3)$. It follows from the above discussion that $\tilde{x}_1=\tilde{x}_2$. Moreover, it follows from Corollary 2.9 in \cite{Biswas-Pal} that $T_{v, \ v_*, \ \omega}\circ T_{\chi, \ \chi_*, \ \theta}=T_{v \circ \chi, \ (v \circ \chi)_*, \ \omega \theta}$. Then 
		\begin{equation*}
			\begin{split}
				f_{\om, v}\circ f_{\theta, \chi}(a, s, p)=e_* \circ T_{v, \ v_*, \ \omega} \circ e_0\circ e_* \circ T_{\chi, \ \chi_*, \ \theta} \circ e_0(a, s, p)
				&=e_* \circ T_{v, \ v_*, \ \omega} \circ e_0\circ e_*(\tilde{a}, \tilde{x}_1, \tilde{x}_1, \tilde{x}_3)\\
				&=e_* \circ T_{v, \ v_*, \ \omega} \circ e_0(\tilde{a}, 2\tilde{x}_1, \tilde{x}_3)\\
				&=e_* \circ T_{v, \ v_*, \ \omega}(\tilde{a}, \tilde{x}_1, \tilde{x}_1, \tilde{x}_3)\\
				&=e_* \circ T_{v, \ v_*, \ \omega}\circ T_{\chi, \ \chi_*, \ \theta}\circ e_0(a, s, p)\\
				&=e_* \circ T_{v \circ \chi, \ (v \circ \chi)_*, \ \omega \theta} \circ  e_0(a, s, p) \\
				&=f_{\om \theta, v \circ \chi}(a, s, p)
			\end{split}
		\end{equation*}
		and so, $G(\Pe)$ is closed under composition. The associativity property follows trivially. Moreover, the identity map can be written as $f_{\om, v}=\text{id}_\Pe$ for $\om=1$ and $v=-B_0$. Note that 
		\begin{equation*}
			f_{\om, v}\circ f_{\overline{\om}, v^{-1}}(a, s, p)=f_{\om \overline{\om}, v \circ v^{-1}}(a, s, p)=(a, s, p).
		\end{equation*}
		Therefore, $f_{\om ,v}$ is a bijective holomorphic map on $\Pbar$ and so, $G(\Pe)\subseteq \A(\Pe)$. Hence, $G(\Pe)$ forms a subgroup of $\A(\Pe)$ under composition. The proof is now complete.
	\end{proof}

	\subsection*{The distinguished boundary of $\Pe$} 
	Recall that for a bounded domain $\Omega$ in $\C^n$, $A(\Omega)$ is the algebra of continuous functions on $\overline{\Omega}$ that are holomorphic in $\Omega$. It was proved in \cite{AglerIV} that the compact sets in $\C^3$ given by
	\[
	K_0=\left\{(a, s, p): (s, p) \in b\Gamma_2, |a|^2+|s|^2 \slash 4=1 \right\}, \quad K_1=\left\{(a, s, p): (s, p) \in b\Gamma_2, |a|^2+|s|^2\slash 4\leq 1\right\}
	\]
	are boundaries of $A(\Pe)$, i.e., every function in $A(\Pe)$ attains its maximum modulus on $K_0$ and $K_1$. We give an alternative proof to this result using the theory of $\Q$. We mention that all the subsequent results related to $b\Pe$ follow from results in Section 8 of \cite{AglerIV}. We also describe $b\Pe$ in terms of $b\Q$.
	
	\begin{prop}\label{peak_K1}
		Every peak point of $\Pbar$ is contained in $K_1$ and $b\Pe \subseteq K_1$. Also, $K_1$ is a closed boundary for $A(\Pe)$. 
	\end{prop}
	
	\begin{proof} This result appears in parts as Proposition 8.3 and Theorem 8.4 in \cite{AglerIV}. We provide an alternative proof here. The fact that $K_1 \subset \Pbar$ follow from Theorem \ref{pentablock}. Let $(a_0, s_0, p_0)$ be a peak point of $\Pbar$. Then there exists $f \in A(\Pe)$ such that $|f(a_0, s_0, p_0)|=1$ and $|f(q)|<1$ for all $q \in \Pbar$ with $q \ne (a_0, s_0, p_0)$. By Theorem \ref{thm_HP_final}, $(a_0, s_0\slash 2, s_0\slash 2, p) \in \CQ$. Consider the map $\psi: \CQ \to \Pbar$ defined as $\psi(a, x_1, x_2, x_3)=(a, x_1+x_2, x_3)$. We have by Lemma \ref{lem_H_P} that $\psi$ is a well-defined holomorphic map. We define the map 
		\[
		g: \CQ \to \C, \quad g(a, x_1, x_2, x_3)=f \circ\psi(a, x_1, x_2, x_3)=f(a, x_1+x_2, x_3).
		\]
		Evidently, $g \in A(\Q)$, $g|(a_0, s_0\slash 2, s_0\slash 2, p_0)|=1$ and $\|g\|_{\infty, \CQ}\leq 1$. Since $b\Q$ is a boundary of $\CQ$, it follows that there exists $(a, x_1, x_2, x_3) \in b\Q$ such that $\|g\|_{\infty, \CQ}=|g(a, x_1, x_2, x_3)|$. By Lemma \ref{lem_H_P}, $(a, x_1+x_2, x_3) \in \Pbar$ and so, $|f(a, x_1+x_2, x_3)|=1$. Consequently, $(a, x_1+x_2, x_3)=(a_0, s_0, p_0)$. It follows from Theorem \ref{thmbQ} that $|a|^2+|x_1|^2=1, x_1=\overline{x}_2x_3, |x_1| \leq 1$ and $|x_3|=1$. Then
		\[
		s_0-\overline{s}_0p_0=(x_1+x_2)-(\overline{x}_1+\overline{x}_2)x_3=0, \quad |s_0|=|x_1+x_2| \leq 2 \quad \text{and} \quad |p_0|=1. 
		\]
		By Theorem \ref{thm_bGam}, $(s_0,p_0) \in b\Gamma_2$. Since $|x_1|=|x_2|=\sqrt{1-|a|^2}$, it follows that 
		\[
		|a_0|^2+\frac{|s_0|^2}{4}=|a_0|^2 +\frac{|x_1|^2}{4}+\frac{|x_1|^2}{4}+\frac{2 \text{Re}\overline{x}_1x_2}{4} \leq |a|^2+|x_1|^2 \leq 1
		\]
		and so, $(a_0, s_0, p_0) \in K_1$. Consequently, every peak point of $\Pbar$ is contained in $K_1$. We have from \cite{Honary} that the peak points of a Banach function algebra $B$  on a compact metric space form a dense subset of the Shilov boundary of $B$ and so, $b\Pe \subset K_1$. Hence, $K_1$ is a boundary for $A(\Pe)$. 	\end{proof} 
	
	We now prove describe $K_0$ using the distinguished boundary of $b\Q$. We also show that every point in $K_0$ a peak point of $\Pbar$ and $K_0 \subseteq b\Pe$, which appear as parts in Theorem 8.4 of \cite{AglerIV}.
	
	\begin{prop}\label{prop14.7}
		$K_0=\{(a, s, p): (a, s\slash 2, s\slash 2, p) \in b\Q\}$. Moreover, every point in $K_0$ is a peak point of $\Pbar$ and $K_0 \subseteq b\Pe$. 
	\end{prop}
	
	\begin{proof} Let $(a, s, p) \in K_0$. Then $|a|^2+|s|^2\slash 4 = 1$ and $(s, p) \in b\Gamma_2$. By Theorem \ref{thm_bGam}, $s=\overline{s}p, |s| \leq 2$ and $|p|=1$. Thus, we have by Theorem \ref{thmbQ} that $(a, s\slash 2, s\slash 2, p) \in b\Q$ if and only if $(a, s,p) \in K_0$. So, $K_0=\{(a, s, p): (a, s\slash 2, s\slash 2, p) \in b\Q\}$ and so, by Theorem \ref{thm_HP_final}, $K_0 \subseteq \Pbar$. Let $(a, s, p) \in K_0$. Then $(a, s\slash 2, s\slash 2, p) \in b\Q$. By Theorem \ref{thmbQ}, there exists $f \in A(\Q)$ such that $f(a, s\slash 2, s\slash 2, p)=1$ and $|f(q)|<1$ for all $q \in \CQ$ with $q \ne (a, s\slash 2, s\slash 2, p)$. It follows from Theorem \ref{thm_HP_final} that the map $g: \Pbar \to \C^3$ given by $g(\tilde{a}, \tilde{s}, \tilde{p})=f(\tilde{a}, \tilde{s}\slash 2, \tilde{s}\slash 2, \tilde{p})$ is well-defined and belongs to $A(\Pe)$. Evidently, $|g(a, s, p)|=1,$ and $\|g\|_{\infty, \Pbar} \leq 1$. If $|g(\tilde{a}, \tilde{s}, \tilde{p})|=1$, then $|f(\tilde{a}, \tilde{s}\slash 2, \tilde{s}\slash 2, \tilde{p})|=1$ and so, $(\tilde{a}, \tilde{s}, \tilde{p})=(a, s, p)$. Thus, $g$ is a peaking function for $(a, s, p)$. Therefore, every point in $K_0$ is a peak point of $\Pbar$. It  follows from the definition of $b\Pe$ that $K_0 \subseteq b\Pe$. The proof is now complete.
	\end{proof} 
	
	Though the previous result tells us that $K_0 \subseteq b\Pe$, Theorem \ref{thmbP} yields that these two sets must coincide. We obtain the same using the methods involved in determining $b\Q$ in \cite{Biswas-Pal2}. To do so, we need to describe the topological boundary $\partial \Pe$ of $\Pe$ in a way similar to that in Section 3 in \cite{Biswas-Pal2}. We have by Theorem \ref{pentablock} that
	\begin{align*}
		\Pe&=\{(a, s, p) \in \C \times \Gg : |\psi_z(a, s, p)| <1 \ \text{for all} \ z \in \D  \}, 
		\\
		\Pbar&=\{(a, s, p) \in \C \times \Gamma_2 \ :  |\psi_z(a, s, p)| \leq 1 \ \text{for all} \ z \in \D  \},
	\end{align*}
	where 
	\[
	\psi_z(a, s, p)=\frac{a(1-|z|^2)}{1-sz+pz^2}.
	\]
	Note that $\psi_z(a, s, p)=\psi_{z, z}(a, s\slash 2, s\slash 2, p)$ for all $(a, s, p) \in \C \times \Gamma_2$ and $z \in \D$. Consequently, a point $(a, s, p) \in \partial \Pe=\Pbar \setminus \Pe$ if and only if either $(s, p) \in \partial \Gg$ and $|\psi_z(a, s, p)| \leq 1$ for every $z \in \D$, or $(s, p) \in \Gamma_2$ and $ |\psi_z(a, s, p)|=1$ for some $z \in \D$. Let us consider the subsets of $\partial \Pe$ given by	
	\begin{align*}
		\partial_0 \Pe&=\{(a, s, p) \in \C \times \partial \Gg : a=0 \}, \\
		\partial_1 \Pe&=\left\{(a, s, p) \in \C \times \Gamma_2 : a \ne 0,  |\psi_{z}(a, s, p)|=1 \ \text{for some} \ z \in \D \right\}, \\
		\partial_2 \Pe&=\left\{(a, s, p) \in \C \times \partial \Gg : a \ne 0, \ \underset{z \in \D}{\sup}|\psi_{z}(a, s, p)| \leq 1 \right\}.
	\end{align*}
	It is not difficult to see that $\partial \Pe= \partial_0 \Pe \cup \partial_1 \Pe \cup \partial_2 \Pe$. We now briefly recall the description of $\A(\Gg)$ from the literature, e.g., see \cite{Agler2008}. For $v \in \A(\D)$, the map 
	\[
	\tau_v(z_1+z_2, z_1z_2)=(v(z_1)+v(z_2), v(z_1)v(z_2)) \quad (z_1, z_2 \in \overline{\D})
	\] 
	is in $\A(\Gg)$ and conversely all automorphisms of $\Gg$ is of this kind. It follows from Theorem \ref{thm_bGam} that $\tau_v(b\Gamma_2)=b\Gamma_2$ for all $v \in \A(\D)$. A reader is refer to \cite{Jarnicki, Agler2008, EdigarianII} for further details. Moreover, if $\omega \in \T$ and $v=-\xi B_z \in Aut(\D)$, the map $f_{\om, v} \in \A(\Pe)$ can be written as 
	\begin{equation}\label{eqn_fwvII}
		f_{\om, v}(a, \lm_1+\lm_2, \lm_1\lm_2)=\left(\om\xi \psi_{\bar z}(a, \lm_1+\lm_2, \lm_1\lm_2), \tau_v(a, \lm_1+\lm_2, \lm_1\lm_2) \right)
	\end{equation}
	for all $(\lm_1+\lm_2, \lm_1\lm_2) \in \Gamma_2$. One can directly use Theorem \ref{thm_AutPsub} for a proof of \eqref{eqn_fwvII}. Furthermore, the maps $f_{\omega, v}$ constitute group $\A(\Pe)$ (see \cite{KosinskiII, GuicongII}). We prove below that any automorphism of $\Pe$ preserves $K_0$ and the peak points of $\Pbar$. In fact, one can also refer to Section 8 of \cite{AglerIV} for a more general result, which shows that an automorphism of the form $f_{\om, v}$ of $\Pe$ preserves its distinguished boundary. However, for our purpose, we do not need the general result.

	\begin{lem}\label{lem_peakfwv}
		Let $\om \in \T$ and $v \in \A(\D)$. Then $f_{\om, v}$ as in \eqref{eqn_fwv} preserves the peak points of $\Pbar$ and $f_{\om, v}(K_0)=K_0$.
	\end{lem}
	
	\begin{proof}
		If $g$ is a peaking function for a peak point $(a, s, p) \in \Pbar$, then $f_{\om , v}(a, s, p)$ is also a peak point of $\Pbar$ with peaking function $g\circ f_{\om , v}^{-1}$. Hence, the peak points of $\Pbar$ are preserved under $f_{\om , v}$. Let $(a, s, p) \in K_0$. We have by Proposition \ref{prop14.7} that $(a, s\slash 2, s\slash 2, p) \in b\Q$. It follows from the proof of Theorem \ref{thm_AutPsub} that $T_{v, v_*, \om}e_0(a, s, p)=T_{v, v_*, \om}(a, s\slash 2, s\slash 2, p)=(\tilde{a}, \tilde{x}_1, \tilde{x}_2, \tilde{p})$ with $\tilde{x}_1=\tilde{x}_2$. Moreover, by Theorem \ref{thmbQ}, $T_{v, v_*, \om}(a, s\slash 2, s\slash 2, p)=(\tilde{a}, \tilde{x}_1, \tilde{x}_2, \tilde{x}_3) \in b\Q$ and so, $|\tilde{a}|^2+|\tilde{x}_1|^2=1$ and $(\tilde{x}_1, \tilde{x}_2, \tilde{x}_3)=(\tilde{x}_1, \tilde{x}_1, \tilde{x}_3) \in b\E$. By Theorems \ref{thm_bGam} \& \ref{thm6.1}, $(2\tilde{x}_1, \tilde{x}_3) \in b\Gamma_2$. Then
		\begin{equation*}
			\begin{split}
				f_{\om, v}(a, s, p)
				&=e_*T_{v, v_*, \om}e_0(a, s, p) \qquad [\text{by Theorem \ref{thm_AutPsub}}]\\ 
				&=e_*T_{v, v_*, \om}(a, s\slash 2, s\slash 2, p)\\
				&=e_*(\tilde{a}, \tilde{x}_1, \tilde{x}_1, \tilde{x}_3)\\
				&=(\tilde{a}, 2\tilde{x}_1, \tilde{x}_3).
			\end{split}
		\end{equation*}
		Since $|\tilde{a}|^2+|\tilde{x}_1|^2=1$ and $(2\tilde{x}_1, \tilde{x}_3) \in b\Gamma_2$, it follows that $(\tilde{a}, 2\tilde{x}_1, \tilde{x}_3) \in K_0$. Therefore, $f_{\om, v}(K_0)\subseteq K_0$ for all $\om \in \T$ and $v \in \A(\D)$. Again, by Theorem \ref{thm_AutPsub} we have $f_{\om, v}^{-1}=f_{\overline{\om}, v^{-1}}$ and so, $f_{\om, v}(K_0)=K_0$ for all $\om \in \T$ and $v \in \A(\D)$. The proof is now complete.
	\end{proof} 
	
	Recall from the literature that an \textit{analytic disc} in a domain $\Omega \subseteq \C^n$ is a non-constant analytic map $h: \D \to \overline{\Omega}$.	The authors of \cite{Biswas-Pal2} proved that the image of an analytic disc in $\CQ$ is a subset of $\Q$ if the image contains at least one element in $\Q$. We now prove an analogue of this result for $\Pbar$. As per our knowledge, we could not locate the next two propositions in the literature. So, we prove them independently. However, we adopt similar techniques as in the works of Abouhajar et al. \cite{Abouhajar}, Edigarian et al. \cite{EdigarianI} and Kosi\'{n}ski \cite{Kosinski, KosinskiII} in proving these results. To this end, let us recall the definition of analytic disc.
	
	\begin{prop}\label{prop14.10}
		Let $h$ be an analytic disc in $\Pbar$. If $h$ maps some point of $\D$ into $\Pe$, then $h(\D) \subseteq \Pe$.
	\end{prop}
	
	\begin{proof}
		Suppose that $h(\lambda_0) \in \Pe$ for some $\lambda_0 \in \D$. Let $h=(h_1, h_2, h_3)$. It follows from Theorem \ref{pentablock} that $(h_2(\lambda), h_3(\lambda)) \in \Gamma_2$ for every $\lambda \in \D$. Let us consider the map $h_0: \D \to \Gamma_2$ defined as $\lambda \mapsto \left((h_2(\lambda), h_3(\lambda)\right)$. Evidently, $h_0$ is an analytic map. It follows from Theorem \ref{pentablock} that $h_0(\lambda_0) \in \Gg$. Using Lemma 3.4 in \cite{Abouhajar}, we conclude that $h_0(\D) \subseteq \Gg$. Let $z \in \D$. Then the map 
		$
		\lambda \mapsto \psi_{z}(h(\lambda))
		$
		is an analytic map on $\D$. It follows from Theorem \ref{pentablock} that $|\psi_{z}(h(\lambda))| \leq 1$. Since $h(\lambda_0) \in \Pe$ and $z \in \D$, we have by Theorem \ref{pentablock} that $|\psi_z(h(\lambda_0))|<1$. By open mapping theorem, $|\psi_z(h(\lambda))|<1$ for every $\lambda \in \D$ and $z \in \D$. By Theorem \ref{pentablock} and the fact that the supremum of $|\psi_z(.)|$ over $z \in \D$ is attained at some point of $\D$, we have that $
		\sup_{z \in \D} |\psi_z(h(\lambda))|<1.
		$
		It follows from Theorem \ref{pentablock} that $h(\lambda) \in \Pe$ for every $\lambda \in \D$. The proof is now complete.
	\end{proof}

	The above result can also be proved using the theory of hexablock. To do so, consider an analytic disc $h$ in $\Pbar$ and the map $e_0: \Pbar \to \C^4, e_0(a, s, p)=(a, s\slash 2, s \slash 2, p)$. It follows from Theorems \ref{thmP_in_H} \& \ref{thm_HP_final} that $e_0(\Pe) \subseteq \Q, e_0(\Pbar)\subseteq \CQ$ and $e_0(\partial \Pe) \subseteq \partial \Q$. Consequently, the map
	$g=e_0 \circ h $ is an analytic disc in $\CQ$. Suppose that $h(\lambda_0) \in \Pe$ for some $\lambda_0 \in \D$. Then $g(\lambda_0)=e_0(h(\lambda_0)) \in e_0(\Pe) \subseteq \Q$. We have by Proposition 3.16 in \cite{Biswas-Pal2} that $g(\D) \subseteq \Q$. Let $\lm \in \D$ and let if possible, $h(\lambda) \notin \Pe$. Then $h(\lm) \in \partial \Pe$ and  $g(\lm)=e_0(h(\lm)) \in e_0(\partial \Pe) \subseteq \partial \Q$. Thus, $g(\lm) \in \partial \Q$, a contradiction. Hence, $h(\D) \subseteq \Pe$.

	\begin{prop}\label{adiscP}
		Let $h$ be an analytic disc in $\Pbar$. Then no point in $h(\D)$ is a peak point of $\Pbar$.	
	\end{prop}
	
	\begin{proof}
		Let if possible, $h(z_0)$ be a peak point of $\Pbar$ for some $z_0 \in \D$. Let us denote $h(z_0)=q=(a, s, p)$. Then there exists $g\in A(\Pe)$ such that $g(q)=1$ and $|g(q')|<1$ for all $q'$ in $\Pbar\setminus \{q\}$. It follows from Proposition \ref{prop_P_lc_qb} that $(ra, rs, r^2p) \in \Pe$ for all $r \in (0, 1)$. Choose $r_n \in (0,1)$ such that $\displaystyle \lim_{n \to \infty}r_n=1$. Let $h=(h_1, h_2, h_3)$. For $n \in \mathbb{N}$, let us consider the following maps:
		\[
		h^{(n)}: \D \to \Pbar, \quad h^{(n)}(z)=\left(r_nh_1(z), r_nh_2(z), r_n^2h_3(z)\right).
		\] 
		Each $h^{(n)}$ is a non-constant well-defined analytic function. It follows from Proposition \ref{prop14.10} that $h^{(n)}(\D) \subseteq \Pe$ as $h^{(n)}(\lambda_0) \in \Pe$. Therefore, $g\circ h^{(n)}$ is a holomorphic map on $\D$. It is not difficult to see that $h^{(n)}$ converges locally uniformly to $h$ on $\D$. Since $g$ is uniformly continuous on $\Pbar$, we have that $g\circ h^{(n)}$ converges locally uniformly to $g\circ h$ on $\D$. Consequently, the map $g\circ h$ is analytic on $\D$ and $\max_{z \in \D}|g\circ h(z)|=|g(h(z_0))|$. Since $z_0 \in \D$, it follows from maximum modulus principle that $g \circ h$ is a constant map. Consequently, $g(h(z))=g(h(z_0))=1$ for all $z \in \D$. Since $g$ peaks at $q$, we must have $h(z)=q$ for all $z \in \D$. This is a contradiction as $h$ is non-constant. Hence, no point of $h(\D)$ can be a peak point of $\Pbar$. 
	\end{proof} 
	
	Proposition \ref{prop14.7} tells us that every peak point of $\Pbar$ is in $K_0$. The converse to this statement is also true, as established in the literature \cite{AglerIV}. Here, we provide a rather different proof based on the theory of hexablock.
	
	\begin{prop}\label{prop_bP_K0}
		Every peak point of $\Pbar$ is in $K_0$ and $b\Pe \subseteq K_0$.
	\end{prop}
	
	\begin{proof} Let $(a, s, p)$ be a peak point of $\Pbar$.  It follows from Proposition \ref{peak_K1} that $(a, s, p) \in K_1$ and so, $(s, p) \in b\Gamma_2$. By Theorem \ref{thm_bGam}, $s=\overline{s}p, |s| \leq 2$ and $|p|=1$. One can choose $\lambda_1, \lambda_2 \in \T$ such that $(s, p)=(\lambda_1+\lambda_2, \lambda_1\lambda_2)$. Such a choice is possible by Theorem \ref{thm_bGam}. It follows from Theorem \ref{pentablock} that $(a, s, p) \in \partial \Pe$. By above discussion, $\partial \Pe= \partial_0 \Pe \cup \partial_1 \Pe \cup \partial_2 \Pe$. We now consider three different cases depending on which of these three parts of $\partial \Pe$ contain $(a, s, p)$.
		
		\smallskip 
		
		\noindent \textbf{Case $1$.} Let $(a, s, p) \in \partial_0 \Pe$. Then $a=0$. Assume that $s^2-4p \ne 0$. Thus, $\lambda_1 \ne \lambda_2$ and so, one can choose $v \in \A(\D)$ such that $v(i)=\lm_1$ and $v(-i)=\lm_2$. We consider the corresponding  $\tau_v$ in $Aut(\Gg)$. Evidently, $\tau_v(0, 1)=(s, p)$. Then the map $f_{\om, v}$ in $\A(\Pe)$ satisfies $f_{\om, v}(0, 0, 1)=(0, s, p)$. By Lemma \ref{lem_peakfwv}, $(0, 0, 1)$ is a peak point of $\Pbar$. For $z \in \overline{\D}$, it follows from Proposition 4.13 in \cite{Biswas-Pal} that $(z, 0, 0, 1) \in \CQ$ and so, by Theorem, \ref{thm_HP_final}, $(z, 0, 1) \in \Pbar$. Let us consider the map 
		\[
		h: \D \to \Pbar, \quad h(z)=(z, 0, 1),
		\]
		which is an analytic disc in $\Pbar$. It follows from Proposition \ref{adiscP} that $h(0)=(0, 0, 1)$ cannot be a peak point for $\Pbar$, which is a contradiction. Thus, $s^2=4p$ and so, $|s|=2$. Hence, $|a|^2+|s|^2\slash 4=1$ and $(s, p) \in b\Gamma_2$ in this case. Consequently, $(a, s, p) \in K_0$.
		
		\smallskip 
		
		\noindent \textbf{Case $2$.} Let $(a, s, p) \in \partial_1 \Pe$. Then $a \ne 0$ and there exists $z \in \D$ such that $|\psi_{\bar z}(a, s, p)|=1$. Defining $v=-B_z$ we have by \eqref{eqn_fwvII} that $f_{1, v}(a, s, p)=\left(\psi_{\bar z}(a, s, p), \tau_v(s, p) \right)$. Let us denote this element by $(\tilde{a}, \tilde{s}, \tilde{p})$. Then $|\tilde{a}|=1$ and $(\tilde{s}, \tilde{p}) \in b\Gamma_2$. Since $(\tilde{a}, \tilde{s}, \tilde{p}) \in \Pbar$, it follows from Theorem \ref{pentablock} that $|\tilde{a}|^2+|\tilde{s}|^2\slash 4 \leq 1$ and so, $\tilde{s}=0$. Hence, $|\tilde{a}|^2+|\tilde{s}|^2\slash 4=1$ and $(\tilde{s}, \tilde{p}) \in b\Gamma_2$. Thus, $f_{1, v}(a, s, p) \in K_0$ and so, by Lemma \ref{lem_peakfwv}, $f_{1, v}^{-1}(f_{1, v}(a, s, p))=(a, s, p) \in K_0$. 
		
		\smallskip 
		
		\noindent \textbf{Case $3$.} Let $(a, s, p) \in \partial_2\Pe$. Then $a \ne 0$ and ${\displaystyle \sup_{z \in \D}|\psi_{z}(a, s, p)| \leq 1}$. If $s^2=4p$, then $|s|^2=4|p|=4$. Since $(a, s, p) \in \Pbar$, it follows from Theorem \ref{pentablock} that $|a|^2+|s|^2\slash 4 \leq 1$ and so, $a=0$ which is a contradiction. Hence, $s^2\ne 4p$ and thus, $\lambda_1 \ne \lambda_2$. Choose $v \in \A(\D)$ such that $v(\lm_1)=i$ and $v(\lm_2)=-i$. The corresponding  $\tau_v$ in $Aut(\Gg)$ satisfies $\tau_v(s, p)=(0, 1)$. Then the map $f_{\om, v}$ in $\A(\Pe)$ satisfies $f_{\om, v}(a, s, p)=(\tilde{a}, 0, 1)$ for some $\tilde{a} \in \C$. Since $\Pbar \subseteq \overline{\D}^3$, we have that $|\tilde{a}| \leq 1$. Again, it follows from Lemma \ref{lem_peakfwv} that $(\tilde{a}, 0, 1)$ is a peak point of $\Pbar$. Consider the map $h$ as in Case $1$ given by 
		\[
		h: \D \to \Pbar, \quad h(z)=(z, 0, 1).
		\]
		Since $h$ is an analytic disc in $\Pbar$ and  $h(\tilde{a})=(\tilde{a}, 0, 1)$ is a peak point of $\Pbar$, we have by Proposition \ref{adiscP} that $|\tilde{a}|=1$. Thus, $(\tilde{a}, 0, 1) \in K_0$ and so, by Lemma \ref{lem_peakfwv}, $f_{\om, v}^{-1}(\tilde{a}, 0, 1))=(a, s, p) \in K_0$. 
		
		\smallskip 
		
		In all three cases, we have that $(a, s, p) \in K_0$ and thus any peak point of $\Pbar$ is in $K_0$. Again, literature (e.g., see \cite{Honary}) tells us that the peak points of a Banach function algebra $B$ on a compact metric space form a dense subset of the Shilov boundary of $B$. Hence, $b\Pe \subseteq K_0$ and the proof is complete. 	\end{proof}

	We now prove the following result that appears first in \cite{AglerIV}. Our proof capitalizes the theory of hexablock and gives a connection between the distinguished boundaries of $\Pe$ and $\Q$.
	
	\begin{thm}\label{thm_bP_2}
		Let $(a,s, p) \in \C^3$. Then the following are equivalent:
		\begin{enumerate}
			\item $(a, s, p) \in K_0$;
			\item $(a, s, p)$ is a peak point of $\Pbar$; 
			\item $(a, s, p) \in b\Pe$. 
		\end{enumerate}
		Therefore, $b\Pe=\{(a, s, p) : (a, s\slash 2, s\slash 2, p) \in b\Q\}$.
	\end{thm}
	
	\begin{proof}
		$(1) \implies (2)$ follows from Proposition \ref{prop14.7}. $(2) \implies (3)$ follows from the definitions of peak points and the distinguished boundary. $(3) \implies (1)$  follows from Proposition \ref{prop_bP_K0}. Hence, $b\Pe=K_0=\{(a, s, p) : (a, s\slash 2, s\slash 2, p) \in b\Q\}$, where the last equality again follows from Proposition \ref{prop14.7}. The proof is now complete.
	\end{proof} 
	
	\subsection*{Rational penta-inner functions} Alshehri and Lykova  \cite{Alsheri} provided a complete description of the rational inner functions from $\D$ to $\Pbar$. We obtain the same result using the theory of hexablock. Let us recall the following definitions.
	
	\begin{defn}
		A \textit{$\Pbar$-inner function} is an analytic map $f: \D \to \Pbar$ such that 
		the radial limit ${\displaystyle	\lim_{r \to 1^{-}} f(r\lambda )}$ exists and belongs to $b\Pe$ for almost all $\lambda \in \T$ with respect to Lebesgue measure. A \textit{$\Gamma_2$-inner function} is an analytic map $f: \D \to \Gamma_2$ such that 
		the radial limit ${\displaystyle \lim_{r \to 1^{-}} f(r\lambda )}$ exists and belongs to $b\Gamma_2$ for almost all $\lambda \in \T$ with respect to Lebesgue measure.
	\end{defn}

	\begin{defn}
		Let $x$ be a rational $\Gamma_2$-inner function and let $x_*:\mathbb{Z}=\pi_1(\T) \to \pi_1(b\Gamma_2)$ be the homomorphism of fundamental groups induced by $x$. The degree $deg(x)$ of $x$ is defined as $x_*(1)$. 
	\end{defn}
	
	We explain a useful characterization of the degree for a rational $\Gamma_2$-function obtained in \cite{AglerV}.
	
	\begin{rem}
		Let $x=(s, p)$ be a rational $\Gamma_2$-inner function. It is not difficult to prove that (see Section 2 in \cite{AglerV}) that $p$ is a rational inner function on $\D$ and so, $p$ is a finite Blaschke product. It follows from Proposition 3.3 in \cite{AglerV} that $deg(x)$ is equal to the degree $deg(p)$ of the finite Blaschke product $p$, i.e., the number of factors in the Blaschke product.
	\end{rem}
	
	Let us recall the description of rational $\Gamma_2$-inner functions from the literature. 
	
	\begin{prop}[\cite{AglerV}, Proposition 2.2]
		If $x=(s, p)$ is a rational $\Gamma_2$-inner function of degree $n$, then there exist polynomials $E$ and $D$ of degree at most $n$ such that the following hold: 
		\begin{equation}\label{eqn_Ginner}
			\begin{split}
				&(1)  \	\ D(\lm) \ne 0 \  \text{on} \  \overline{\D}, \qquad (2) \ E=E^{\sim n} \ \text{on} \ \overline{\D}, 
				\qquad (3) \  |E(\lm)| \leq 2|D(\lm)| \  \text{for} \ \lm \in \overline{\D}, \\
				&(4) \ \	s=\frac{E}{D}\ \ \text{on} \ \ \overline{\D},\qquad (5)  \ \	x_3=\frac{D^{\sim n}}{D} \ \ \text{on} \ \ \overline{\D}. \\
			\end{split}	
		\end{equation}
		Conversely, if $E$ and $D$ are polynomials of degree at most $n$ that satisfy conditions $(1), (2)$ and $(3)$, then $x=(s, p)$ with $s$ and $p$ defined as in $(4)$ and $(5)$ is a rational $\Gamma_2$-inner function of degree at most $n$.
	\end{prop}
	
	The next result follows directly from Theorems \ref{thm_HP_final} \& \ref{thm_bP_2}.
	
	\begin{lem}\label{lem_ratPQ}
		A function $f=(a, s, p)$ is a rational $\Pbar$-inner function if and only if $(a, s\slash 2, s\slash 2, p)$ is a rational $\CQ$-inner function.
	\end{lem}

	\begin{thm}[\cite{Alsheri}, Theorem 7.5]
		Let $f=(a, s, p)$ be a rational $\Pbar$-inner function and let $deg(p)=n$. Let $a \ne 0$ and let $a=a_{in}a_{out}$ be an inner-outer factorization of $a$. Then there exist polynomials $A, E, D$ of degree at most $n$ such that  the following hold.
		\begin{enumerate}
			\item $E(\lm)=E^{\sim n}(\lm)$ for $\lm \in \overline{\D}$,\\
			
			\item $A$ is outer and $D(\lm) \ne 0$ for all $\lambda  \in \overline{\D}$,\\
			
			\item $|E(\lm)| \leq 2|D(\lm)|$ on $\overline{\D}$,\\
			
			\item $|A(\lambda)|^2=|D(\lambda)|^2-\frac{1}{4}|E(\lambda)|^2$ on $\T$,\\	
			
			\item 
			$
			(a, s, p)=\left(a_{in}\dfrac{A}{D}, \dfrac{E}{D}, \dfrac{D^{\sim n}}{D} \right)$ on $\ \overline{\D}$.
		\end{enumerate}
		Conversely, if $A, E$ and $D$ are polynomials of degree at most $n$ that satisfy conditions $(1)-(4)$ and $a, s$ and $p$ are defined as in $(5)$, then $f=(a, s, p)$ is a rational $\Pbar$-inner function with degree of $p$ at most $n$.	
	\end{thm}
	
	\begin{proof} We prove using the theory of rational $\CQ$-inner functions. Let $f=(a, s, p)$ be a rational $\Pbar$-inner function and let $deg(p)=n$. It follows from Lemma \ref{lem_ratPQ} that the map 
		$h=(a, s\slash 2, s\slash 2, p)$ is a rational $\CQ$-inner function. We have by Theorem \ref{thm10.12} that there exist polynomials $A, E_1, E_2, D$ of degree at most $n$ such that the following hold:
		\begin{enumerate}
			\item[(i)] $E_1(\lm)=E_2^{\sim n}(\lm)$ for $\lm \in \overline{\D}$,\\
			
			\item[(ii)] $A$ is outer and $D(\lm) \ne 0$ for all $\lambda  \in \overline{\D}$,\\
			
			\item[(iii)] $|E_i(\lm)| \leq |D(\lm)|$ on $\overline{\D}$, for $i=1,2$,\\
			
			\item[(iv)] $|A(\lambda)|^2=|D(\lambda)|^2-|E_1(\lambda)|^2$ on $\T$,\\	
			
			\item[(v)] 
			$
			(a, s\slash 2, s\slash 2, p)=\left(a_{in}\dfrac{A}{D}, \dfrac{E_1}{D}, \dfrac{E_2}{D}, \dfrac{D^{\sim n}}{D} \right)$ on $\overline{\D}$.		
		\end{enumerate}
		Note that $E_1, E_2$ are polynomials with $E_1=E_2$ on $\overline{\D}$ and hence by identity theorem $E_1=E_2$ everywhere. Choose $E=2E_1$. Then $E(\lm)=E^{\sim n}(\lm)$ and $|E(\lambda)| \leq 2|D(\lambda)|$ for $\lm \in \overline{\D}$. The forward part now follows. Conversely, assume that $A, E$ and $D$ are polynomials of degree at most $n$ that satisfy conditions $(1)-(4)$ and
		\[
		(a, s, p)=\left(a_{in}\frac{A}{D}, \frac{E}{D}, \frac{D^{\sim n}}{D} \right) \quad \text{on} \quad  \overline{\D}.
		\]
		Let $E_1=E_2=E\slash 2$. Then $A, E_1, E_2$ and $D$ satisfy the hypothesis of Theorem \ref{thm10.12}. It follows from Theorem \ref{thm10.12} that $(a, s\slash 2, s\slash 2, p)$ is a rational $\CQ$-inner function. Finally, we have by Lemma \ref{lem_ratPQ} that $f=(a, s, p)$ is a rational $\Pbar$-inner function. The proof is now complete.	\end{proof} 
	
	\subsection*{A Schwarz Lemma for the pentablock} Let $\lm_0 \in \D \setminus \{0\}$ and $(a, s, p) \in \Pe$. Agler et al. in \cite{AglerIV} provided a necessary condition for the existence for a function $f \in$ Hol$(\D, \Pbar)$ such that $f(0)=(0, 0, 0)$ and $f(\lm_0)=(a, s, p)$.
	
	\begin{prop} [\cite{AglerIV}, Proposition 11.1] \label{prop_Schw_P_nec}
		Let $f \in$ Hol$(\D, \Pe)$. If $f(0)=(0, 0, 0)$ and $f(\lambda_0)=(a, s, p)$, then
		\begin{equation*}
			\frac{2|s-\overline{s}p|+|s^2-4p|}{4-|s|^2} \leq |\lambda_0| \quad \text
			{and} \quad 		\sup_{z \in \D} \bigg|\frac{a(1-|z|^2)}{1-sz+pz^2}\bigg| \leq |\lambda_0|.
		\end{equation*}
	\end{prop}
	
	\begin{proof}
		We give an alternative proof using the theory of the hexablock. Let $f=(f_1, f_2, f_3)$. Consider the map $h=e_0 \circ f=(f_1, f_2\slash 2, f_2 \slash 2, f_3)$ on $\D$. It follows from Theorem \ref{thm_HP_final} that $h \in$ Hol$(\D, \Q)$. Evidently, $h(0)=(0, 0, 0, 0)$ and $h(\lm_0)=(a, s\slash 2, s\slash 2, p)$. Let $x_1=x_2=s\slash 2$ and $x_3=p$. Some routine computations and Theorem \ref{necessary} together give
		\begin{equation*}
			\frac{2|s-\overline{s}p|+|s^2-4p|}{4-|s|^2}=	\max\bigg\{\frac{|x_1-\overline{x}_2x_3|+|x_1x_2-x_3|}{1-|x_2|^2}, \ \frac{|x_2-\overline{x}_1x_3|+|x_1x_2-x_3|}{1-|x_1|^2}\bigg\} \leq |\lambda_0|
		\end{equation*}
		and 
		\begin{equation*}
			\sup_{z \in \D} \bigg|\frac{a(1-|z|^2)}{1-sz+pz^2}\bigg| = 	\sup_{z \in \mathbb{D}}\bigg|\frac{a(1-|z|^2)}{1-(x_1+x_2)z+x_3z^2}\bigg| \leq \sup_{z_1, z_2 \in \mathbb{D}}\bigg|\frac{a\sqrt{(1-|z_1|^2)(1-|z_2|^2)}}{1-x_1z_1-x_2z_2+x_3z_1z_2}\bigg| \leq |\lambda_0|.
		\end{equation*}
		The proof is now complete.	 
	\end{proof}

Using the theory of hexablock, we present a necessary and sufficient condition for the solvability of Schwarz lemma for $\Pe$. In fact, the result below follows directly from Theorem \ref{thm_final_embed}.  
		
	\begin{prop}\label{prop_SL_iff}
		Let $\lm_0 \in \D\setminus \{0\}$ and let $(a, s, p) \in \Pe$. Then the following are equivalent:
		\begin{enumerate}
			\item[(i)] there exists
			$f\in\text{Hol}(\mathbb D, \mathbb P)$
			such that $f(0)=(0,0,0)$ and $f(\lambda_0)=(a,s,p)$;

			\item[(ii)] there exists
			$g\in\text{Hol}(\mathbb D, \Q)$
			such that $g(0)=(0,0,0,0)$ and  $g(\lambda_0)=
			\left(a, s\slash 2, s\slash 2,p\right)$.	
		\end{enumerate}
	\end{prop}
	
\begin{proof}
	Suppose there exists
	$f\in\text{Hol}(\mathbb D, \mathbb P)$
	such that $f(0)=(0,0,0)$ and $f(\lambda_0)=(a,s,p)$. Let $f=(f_1, f_2, f_3)$. By Theorem \ref{thm_final_embed}, the map $g: \D \to \Q$ given by $g=(f_1, f_2\slash 2, f_2\slash 2, f_3)$ is a well-defined holomorphic map. Clearly, $g(0)=(0, 0, 0, 0)$ and $g(\lm_0)=(a, s\slash 2, s\slash 2, p)$.  Conversely, suppose there exists
	$g\in\text{Hol}(\mathbb D, \Q)$
	such that $g(0)=(0,0,0,0)$ and  $g(\lambda_0)=
	\left(a, s\slash 2, s\slash 2,p\right)$. Let $g=(g_0, g_1, g_2, g_3)$. Again by Theorem \ref{thm_final_embed}, the function $f$ defined as $f=(g_0, g_1+g_2, g_3)$ is holomorphic on $\D$ that maps $\D$ into $\Pe$. Evidently, $f(0)=(0,0,0)$ and $f(\lambda_0)=(a,s,p)$.
\end{proof}	
	
Finally, we present the following necessary condition for the solvability of Schwarz lemma for $\Pe$ when an interpolating function is rational $\Pbar$-inner.

\begin{prop}
	Let $\lm_0 \in \D\setminus \{0\}$ and let $(a, s, p) \in \Pe$. If there exists a rational $\Pbar$-inner map $f$ such that $f(0)=(0,0,0)$ and $f(\lambda_0)=(a,s,p)$. Then there exist $w_{12}, w_{21}, z_{12}\in\mathbb C$ satisfying $4w_{12}w_{21}=s^2-4p$ and $|a| \leq |\lm_0w_{12}|$ such that 
		\begin{equation*}
			\begin{pmatrix}
				I-W_0W_0^*&I-W_0W_1^*\\[2mm]
				I-W_1W_0^*&
				\dfrac{I-W_1W_1^*}{1-|\lambda_0|^2}
			\end{pmatrix}
			\geq0,
			\quad 	\text{where} \quad W_0=
			\begin{pmatrix}
				0&z_{12}\\
				0&0
			\end{pmatrix} \ \ \text{and} \ \  W_1=
			\begin{pmatrix}
				x_1& w_{12}\\
				w_{21}&x_2
			\end{pmatrix}.
		\end{equation*}	
\end{prop}

\begin{proof}
	Suppose there exists a rational $\Pbar$-inner map
	$f$	such that $f(0)=(0,0,0)$ and $f(\lambda_0)=(a,s,p)$. Let $f=(f_1, f_2, f_3)$. By Theorem \ref{thm_final_embed}, the map $g: \D \to \Q$ given by $g=(f_1, f_2\slash 2, f_2\slash 2, f_3)$ is a well-defined rational holomorphic map. Also, it follows from Theorem \ref{thm_bP_2} that $g$ is a $\CQ$-inner map. The desired conclusion now follows from Theorem \ref{thm_hexa_schwarz}.
\end{proof}

	\section{Concluding Remarks}
	
	\noindent	We conclude this article here with an indication towards several future directions of works and open problems. Along with
	this conjecture, many other intriguing questions still remain open. Below we state some of the open questions that may enhance future research.
	\begin{enumerate}
		\item It is known from \cite{NikolovII} and \cite{Zwonek} that the
		symmetrized bidisc
		$\Gg$ and the tetrablock $\E$ are $\C$-convex domains. More recently, the
		author of \cite{Guicong} established that the pentablock $\Pe$ also
		enjoys $\C$-convexity. Given the close geometric and function theoretic
		connections between $\Q$ and these domains, one can naturally ask if $\Q$ is a $\C$-convex domain too. However, at this point determining $\C$-convexity of the hexablock seems to be harder than that of the aforementioned domains.
		
		\smallskip
		
		\item The Lempert's theorem on coincidence of the Carath\'{e}odory distance and the Kobayashi distance has always been an interesting topic of research. Lempert's path-breaking work \cite{Lempert} shows that the Carath\'{e}odory distance and the Kobayashi distance coincide in a convex domain. However, the converse to this phenomenal result remained an open problem for two decades. Agler and Young \cite{AglerIII} proved that the symmetrized bidisc $\Gg$, being a non-convex domain, bears the success of Lempert's theorem. In fact, the symmetrized bidisc is the first example of a non-convex domain where Lempert's theorem holds. Interestingly, the same is true for another non-convex domain tetrablock as was shown by Edigarian, Kosi\'{n}ski and Zwonek in \cite{EdigarianI}. It is natural to ask if the Carath\'{e}odory distance and the Kobayashi distance coincide in the hexablock $\Q$ too.
		
		\smallskip
		
		\item Let $\Omega$ be a domain in $\C^n$. Given $m$ distinct points
		$\lambda_1, \dotsc, \lambda_m$ in $\D$ and $m$ points $z_1, \dotsc, z_m$
		in $\overline{\Omega}$, an  $\overline{\Omega}$-interpolation problem
		asks if one can find a holomorphic map $f: \D \to \overline{\Omega}$ such
		that $f(\lambda_j)=z_j$ for $1 \leq j \leq m$. A necessary and sufficient condition (e.g., see Chapter X in \cite{Foias_Frazho}) for the solvability of $\DC$-interpolation problem was given by the semi-positive definiteness of the matrix
		\[
		\begin{pmatrix}
			\displaystyle \frac{1-\overline{z}_iz_j}{1-\overline{\lambda}_i\lambda_j}
		\end{pmatrix}_{i, j=1}^m.
		\]
		While the general solution for a $\Gamma_2$-interpolation problem is still
		unknown, Agler and Young \cite{ALY12} obtained a complete solution for a
		two point $\Gamma_2$-interpolation problem. The solvability criteria for an
		$\Ebar$-interpolation problem was explored in \cite{Abouhajar, Brown}. In the same spirit, any general case of interpolation (at least for a two-point interpolation) in $\Q$ remains an open and challenging problem.
		
		\smallskip
		
		\item The $\mu$-synthesis problem is an interpolation problem from $\D$ to the $\mu_E$-unit ball $\mathbb{B}_{\mu_E}$ as was
		discussed in Section \ref{ch_intro}. Also, the $\mu$-synthesis with respect to a closed linear subspace $E$ of $M_n(\C)$ gives
		rise to a domain $\Omega$ in $\C^m$ for some $m\in \mathbb N$ and a holomorphic map $\pi_{\Omega}:
		\mathbb{B}_{\mu_E} \to \C^m$ such that
		$\pi_{\Omega}(\mathbb{B}_{\mu_E})=\Omega$, or at least
		$\pi_{\Omega}(\mathbb{B}_{\mu_E}) \subseteq \Omega$. For example, the symmetrized bidisc $\Gg$, the tetrablock $\E$ and the
		pentablock $\Pe$ are such domains that are realized as images of the corresponding
		$\mu_E$-unit ball under the maps $\Pi_2$, $\pi_{\E}$ and $\pi_{\Pe}$,
		respectively. If $f: \D \to \mathbb{B}_{\mu_E}$ is an analytic function that interpolates a given $n$-point data: (distinct points) $\lambda_1, \dotsc, \lambda_m \in \D$ and $A_1,
		\dotsc, A_m \in \mathbb{B}_{\mu_E}$, then $f_0= \pi_{\Omega} \circ f: \D
		\to \Omega$ is an analytic function that interpolates the corresponding $n$-point data from $\D$ to $\Omega$, i.e., $F(\lambda_j)=\pi(A_j)$
		for $ 1 \leq j \leq k$. An analytic lifting problem asks if the converse
		to this holds, i.e., given a holomorphic map $f_0: \D \to \Omega$, if one
		can find a holomorphic map $f: \D \to \mathbb{B}_{\mu_E}$ such that $f_0=
		\pi_{\Omega} \circ f$. Analytic lifting problems for $\Gg$ and $\E$ were
		studied in \cite{AglerYoung, AglerI, NikolovIII} and \cite{Abouhajar},
		respectively. Recently, an analytic lifting problem for $\Pe$ was
		explored in \cite{AglerIV, Alsheri}. Following the notations in Section \ref{ch_intro}, let
		$\mathbb{B}_{\mu_{hexa}}=\{A \in M_2(\C) : \mu_{hexa}(A) <1\}$, where
		$\mu_{hexa}$ is the cost function $\mu_E$ corresponding to the space $E$
		of all upper triangular matrices in $M_2(\C)$. It is easy to see that the map $\pi: \mathbb{B}_{\mu_{hexa}} \to \Q$ given by
		$\pi(A)=(a_{21}, a_{11}, a_{22}, \det(A))$ for $A=(a_{ij})_{i, j=1}^2$ is
		well-defined and holomorphic. This naturally leads to the study of the
		associated analytic lifting problem for $\Q$ and is a genuine future direction of research.
		
		\smallskip
		
		\item For a domain $\Omega$ in $\C^n$, a set $D \subseteq \Omega$ is said to be a \textit{complex geodesic} in $\Omega$ if there are
		holomorphic maps $f: \D \to \Omega$ and $g: \Omega \to \D$ such that
		$g\circ f=\text{id}_{\D}$ and $D=f(\D)$. It was proved in \cite{ALY14} (also see \cite{PZ}) that the complex geodesics in $\Gg$ coincide with the nontrivial holomorphic retracts of $\Gg$. The authors of \cite{Edi2009} describe all complex geodesics in $\E$ passing through the origin, and further investigations
		on holomorphic retracts and complex geodesics in $\E$ can be found in \cite{Gargi}. An ambitious but far-reaching target would be to determine a geometric classification of all complex geodesics in the hexablock.
	\end{enumerate}

	\subsection*{Funding} The first named author is partially supported in part by a J. C. Bose Fellowship (JBR\slash 2023\slash 000003). The second named author is supported by the `Core Research Grant (CRG)' with Award No. CRG/2023/005223 of Anusandhan National Research Foundation (ANRF), Govt. of India. The third named author was supported by the Prime Minister's Research Fellowship (PMRF), Government of India (PMRF ID No. 1300140), and the IIT Bombay RDF Grant of the second named author (Project Code RI/0115-10001427) during the course of work. He is currently an NBHM Postdoctoral Fellow at the Indian Statistical Institute, Bangalore Centre.

\end{document}